\documentclass{article}
\pdfoutput=1
\usepackage{xparse}
\usepackage[utf8]{inputenc}
\usepackage{lmodern}
\usepackage{textgreek}
\usepackage{csquotes}
\usepackage{subfiles}
\usepackage{xr-hyper}
\usepackage[backend=biber,style=alphabetic, doi = false, isbn=false]{biblatex}
\renewbibmacro{in:}{}
\AtEveryBibitem{%
  \clearfield{primaryClass}%
  \clearfield{EDITOR}
   \clearfield{eid}
}
\usepackage{amsmath}
\usepackage{amsthm}
\usepackage{amssymb}
\usepackage{mathtools}
\usepackage{amsfonts}
\usepackage{wasysym}
\usepackage{bbm}
\usepackage{thmtools} 
\usepackage{faktor}
\usepackage{MnSymbol}
\usepackage{float}
\usepackage{xargs}
\usepackage{enumitem}
\usepackage[hidelinks]{hyperref} 
\usepackage[capitalize, nameinlink]{cleveref}

\usepackage{graphicx}
\usepackage{xcolor}

\usepackage{pdfpages}

\usepackage[margin=3.4cm]{geometry}
 
\definecolor{Bnavy}{RGB}{0, 66, 128}
\definecolor{Bdust}{RGB}{140,179,217}
\definecolor{Bsugarpaper}{RGB}{198, 217, 236}
\definecolor{Bgreen}{RGB}{142, 183, 114}
\definecolor{Blimegreen}{RGB}{202, 222, 189}
\definecolor{Bgreentheme}{RGB}{36, 87, 1}

\usepackage{makecell}
\usepackage{stackengine} 

\usepackage{microtype}

\usepackage{tikz} 
\usepackage{tikz-cd}
\usetikzlibrary{ 
	calc,%
	arrows,%
	shapes,
	shapes.geometric,
    positioning,
    fit
} 
\usetikzlibrary{backgrounds, calc, math, patterns, arrows.meta}
\usetikzlibrary{decorations.markings}

\tikzcdset{arrow style=tikz, diagrams={>=stealth}}
\usepackage[labelsep=quad,indention=10pt]{subfig}
\theoremstyle{plain}
\newtheorem{theorem}{Theorem}[subsection]

\newtheorem{lemma}[theorem]{Lemma}
\newtheorem{corollary}[theorem]{Corollary}
\newtheorem{proposition}[theorem]{Proposition}
\newtheorem{construction}[theorem]{Construction}
\newtheorem{innercustomthm}{Main Result}
\newenvironment{ctheorem}[1]
  {\renewcommand\theinnercustomthm{#1}\innercustomthm}
  {\endinnercustomthm}

\theoremstyle{definition}
\newtheorem{definition}[theorem]{Definition}
\newtheorem{example}[theorem]{Example}
\newtheorem{recollection}[theorem]{Recollection}
\crefname{recollection}{Recollection}{Recollections}
\newtheorem{remark}[theorem]{Remark}
\newtheorem{notation}[theorem]{Notation}

\newcounter{diagram}  
\crefname{diagram}{Diagram}{Diagrams}
\creflabelformat{diagram}{#2\textup{(#1)}#3}
\newenvironment{diagram}[1][]{%
    \crefalias{equation}{diagram}
    \begin{equation}%
    \begin{tikzcd}[#1]%
}{%
    \end{tikzcd}%
    \end{equation}%
}

\usepackage{notation}
 \newlist{PhiProps}{enumerate}{4}
 \setlist[PhiProps]{label*=(\roman*)}
    \crefname{PhiPropsi}{Observation}{Observations}
    \Crefname{PhiPropsi}{Observation}{Observations}
 \newlist{HProps}{enumerate}{4}
 \setlist[HProps]{label*=(\roman*)}
    \crefname{HPropsi}{Property}{Properties}
    \Crefname{HPropsi}{Property}{Properties}
 \newlist{LinkComp}{enumerate}{4}
 \setlist[LinkComp]{label*=(\roman*)}
    \crefname{LinkCompi}{Description}{Descriptions}
    \Crefname{LinkCompi}{Description}{Descriptions}
   \newlist{LocEx}{enumerate}{4}
 \setlist[LocEx]{label*=(\roman*)}
    \crefname{LocExi}{Item}{Items}
    \Crefname{LocExi}{Item}{Items}
 \newlist{LocExClaim}{enumerate}{4}
 \setlist[LocExClaim]{label*=(\alph*)}
    \crefname{LocExClaimi}{Claim}{Claims}
    \Crefname{LocExClaimi}{Claim}{Claims}
\newlist{AssTranfLem}{enumerate}{4}
 \setlist[AssTranfLem]{label*=(\arabic*)}
    \crefname{AssTranfLem}{Assumption}{Assumptions}
    \Crefname{AssTranfLemi}{Assumption}{Assumptions}
\newlist{FundamentalObs}{enumerate}{4}
 \setlist[FundamentalObs, 1]{label*=\arabic{FundamentalObsi}.}
  
    \crefname{FundamentalObs}{Observation}{Observations}
    \Crefname{FundamentalObsi}{Observation}{Observations}
    \Crefname{FundamentalObsii}{Observation}{Observations}
\newlist{approaches}{enumerate}{4}
 \setlist[approaches]{label*=Ap.\arabic*}
 \setlist[approaches,2]{label=\alph*)}
    \crefname{approaches}{Approach}{Approaches}
    \Crefname{approachesi}{Approach}{Approaches}
    \Crefname{approachesii}{Approach}{Approaches}
\newlist{motivatingQuestions}{enumerate}{4}
 \setlist[motivatingQuestions]{label*=Q.\arabic*}
    \crefname{motivatingQuestions}{Question}{Questions}
    \Crefname{motivatingQuestionsi}{Question}{Questions}
\newlist{conditionsTriang}{enumerate}{4}
 \setlist[conditionsTriang]{label*=C.\arabic*}
    \crefname{conditionsTriang}{Condition}{Conditions}
    \Crefname{conditionsTriangi}{Condition}{Conditions}
\newlist{versionHoHy}{enumerate}{4}
 \setlist[versionHoHy]{label*=V.\arabic*}
    \crefname{versionHoHy}{Version}{Versions}
    \Crefname{versionHoHyi}{Version}{Versions}   
    \title{Homotopy links of stratified cell complexes}
\author{Lukas Waas}
\date{January 2025}
\begin{document}
\maketitle
\begin{abstract}
 Homotopy links have proven to be one of the most powerful tools of stratified homotopy theory. In previous work, we described combinatorial models for the generalized homotopy links of a stratified simplicial set. For many purposes, in particular to investigate the stratified homotopy hypothesis, a more general version of this result pertaining to stratified cell complexes is needed. Here we prove that, given a stratified cell complex $X$, the generalized homotopy links can be computed in terms of certain subcomplexes of a subdivision of $X$. 
 As a consequence, it follows that generalized homotopy links map certain pushout diagrams of stratified cell complexes into homotopy pushout diagrams.
 This result is crucial to the development of (semi)model structures for stratified homotopy theory in which geometric examples of stratified spaces, such as Whitney stratified spaces, are bifibrant.
\end{abstract}
Stratified spaces were first introduced by Whitney, Thom and Mather to investigate spaces with singularities (see \cite{Whitney,mather1970notes,mather1973strat,thom1969ensembles}). One of the central insights of \cite{mather1973strat} was that a key ingredient in the study of stratified spaces with smooth manifold strata was having a theory of tubular neighborhoods of strata available.
These made it possible to study stratified spaces in terms of their strata and the so-called link bundles, connecting the latter.
In a less geometric scenario, such tubular (or regular) neighborhoods may generally not be available. To avoid this difficulty, \cite{quinn1988homotopically} introduced the notion of a homotopy link - a homotopy-theoretic proxy for the boundary of a regular neighborhood.
Given two strata $\utstr_p$ and $\utstr_q$ in a poset-stratified space $\ststr \colon \utstr \to \pos$, with $p<q \in \pos$, the associated homotopy link is the space of paths starting in $\utstr_p$ and immediately exiting into $\utstr_q$. \\
It turns out that much of the homotopy theory of stratified spaces may be understood in terms of the homotopy types of homotopy links and strata (and the structure maps between them). For example, \cite{miller2013} proved that a stratum-preserving map between two sufficiently regular stratified spaces is a stratum-preserving homotopy equivalence, if and only if it induces equivalences on strata and (pairwise) homotopy links. \cite{douteauEnTop,Henriques} built on this insight, and developed a homotopy theory of stratified spaces in which weak equivalences are defined as such stratified maps that induce weak equivalences on all generalized homotopy links, which replace exit-paths by more general stratified singular simplices. We call this theory the Douteau-Henriques homotopy theory, henceforth. It turns out that the Douteau-Henriques homotopy theory is in some sense minimal amongst many stratified homotopy theories (see \cite{douSimp}). Thus, it is not surprising that many other approaches to stratified homotopy theory turn out to be localizations, global versions, or subtheories of the latter (see \cite{ComModelWa}). It follows from this that much about stratified homotopy theory (not only about the Douteau-Henriques one) can be understood in terms of generalized homotopy links. For example, in \cite{douteauwaas2021}, we obtained explicit combinatorial models - in terms of a subobject of a subdivision - for the homotopy link of a stratified simplicial set. We used this to prove a stratified $\infty$-categorical analogue of the Kan-Quillen equivalence between topological spaces and simplicial sets (\cite[Thm. 5.1, Rem. 5.4]{douteauwaas2021}). The latter can be used to prove a Quillen equivalence version of the topological stratified homotopy hypothesis (see \cite{TSHHWa}). \\
The applications in \cite{douteauwaas2021} show the strength of a general paradigm: Homotopy links as a mathematical tool become most powerful when geometric or combinatorial, as well as homotopy-theoretic models, are available.  \\
\\
The main goal of this article is to extend the availability of such models to the case of so-called \textit{stratified cell complexes}. Roughly speaking, a stratified cell complex is a stratified space obtained by inductively gluing in stratified simplices along stratum-preserving maps defined on the boundaries of the simplices.
From a technical point of view, constructing and verifying models for generalized homotopy links in terms of subcomplexes of a subdivision of a stratified cell complex turns out to be significantly more involved than the simplicial set case. This is mainly due to the fact that the case of cell complexes allows arbitrarily complicated gluing maps, rather than only allowing for (piecewise) linear ones (see \cref{ex:aspire_gen_not_ext}). 
Nevertheless, we show the following.\\
We describe a construction that, given a stratified cell complex $\tstr$, stratified over a poset $P$, and $\I = \{p_0 < \dots < p_n\} \subset P$ a finite increasing sequence, produces a stratified subspace $\PsiStanHoods[\I]{\str}$ of $\str$. This construction relies on a choice of (appropriate topological barycentric) subdivision $\Psi$ of $\str$, and is such that $\PsiStanHoods[\I]{\str} \subset \str$ defines a subcomplex of this subdivision that contains the stratum $X_{p_0}$. It can be seen as a generalization of regular neighborhoods in the PL setting, both to the case of more than two strata, as well as to the case of more general stratified cell complexes. Crucially, the construction has the following property:
\begin{ctheorem}{A}[\cref{prop:models_model_globally,prop:ex_nbhd_cell_cplx,thm:cell_stanhood_is_holink}]\label{thm:hol_main_result}
        Let $\tstr$ be a stratified cell complex, stratified over a poset $P$, and $\I = \{p_0 < \dots < p_n\} \subset P$ a finite increasing sequence. Denote by $\HolIP \tstr$ the $\I$-th generalized homotopy link of $\tstr$ (see \cite{douteauEnTop,douteauwaas2021}).
        There exists a subdivision $\Psi$ of $\str$, such that $\PsiStanHoods[\I]{\str}$ defines a well-defined subcomplex with respect to this subdivision. Given such a $\Psi$, there is a canonical weak homotopy equivalence 
        \[
        \HolIP \tstr \simeq \PsiStanHoods[\I]{\str}_{p_n},
        \]
        between the $\I$-th generalized homotopy link and the $p_n$-stratum of $\PsiStanHoods[\I]{\str}$.
        Furthermore, subdivisions can be chosen such that the construction of $\PsiStanHoods[\I]{\str}$ is compatible with stratum-preserving maps and pushouts along inclusions of subcomplexes.
    \end{ctheorem}
    \begin{figure}[h]
    \centering
    \begin{tikzpicture}[witharrow/.style={postaction={decorate}}, scale = 1.5, 
dot/.style = {circle, fill, minimum size=#1, inner sep=0pt, outer sep=0pt},
dot/.default = 2pt  
]
\newcommand{\doubLine}[3]{\draw[color = black, line width = 2.5pt] (#2) -- (#3);
                 \draw[color = #1, line width = 1.5pt] (#2) -- (#3)}
                \coordinate (lm) at (0,0);
                \coordinate (mm) at (1.5,0);
                \coordinate (rm) at (3,0);
                \coordinate (um) at (1.5,1);
                \coordinate (dm) at (1.5,-1);
                \coordinate (blmum) at ($0.5*(lm)+0.5*(um)$);

                \draw[fill = NewBlue]{}(lm) -- (um) -- (rm) -- (dm) -- (lm);
                \doubLine{NewBlue}{lm}{um};
                  \doubLine{NewBlue}{um}{rm};
                   \doubLine{NewBlue}{rm}{dm};
                \doubLine{NewBlue}{dm}{lm};
                \doubLine{NewBlue}{um}{dm};

                 \draw[draw = none, decoration={markings,mark= at position 0.5 with {\arrow[scale = 1.5]{latex}}},witharrow] (lm) -- (um);
                \draw[draw = none, decoration={markings,mark= at position 0.5 with {\arrow[scale = 1.5]{latex}}},witharrow] (lm) -- (dm);
                 \draw[color = black, line width = 2.5pt] (0,0) -- (3,0);
                \draw[color = green, line width = 1.5pt] (0,0) -- (3,0);
                \draw[draw = none, decoration={markings,mark=between positions 0.5 and 0.6 step 0.1 with {\arrow[scale = 1.5]{latex}}},witharrow] (3,0) -- (1.5,-1);
                \draw[draw = none, decoration={markings,mark=between positions 0.5 and 0.6 step 0.1 with {\arrow[scale = 1.5]{latex}}},witharrow] (3,0) -- (1.5,1);
                \draw[color = black, fill = red, ](lm) circle (2pt);
                \node (x) at (0,0){ \textcolor{white}{x}}   ; 
                \draw[color = black, fill = red, ](rm) circle (2pt);
                \node (x) at (3,0){ \textcolor{white}{x}}   ;
                \draw[color = black, fill = green] (mm) circle (2pt);
                \draw[color = black, fill = NewBlue] (um) circle (2pt);
                \node (y) at (um){ \textcolor{white}{y}}   ; 
                \draw[color = black, fill = NewBlue] (dm) circle (2pt);
                \node (y) at (dm){ \textcolor{white}{y}}   ; 
            \end{tikzpicture}    
            \begin{tikzpicture}[witharrow/.style={postaction={decorate}}, scale = 1.5, 
dot/.style = {circle, fill, minimum size=#1, inner sep=0pt, outer sep=0pt},
dot/.default = 2pt  
]
\newcommand{\doubLine}[3]{\draw[color = black, line width = 2.5pt] (#2) -- (#3);
                 \draw[color = #1, line width = 1.5pt] (#2) -- (#3)}
                \coordinate (l) at (0,0);
                \coordinate (m) at (1.5,0);
                \coordinate (r) at (3,0);
                \coordinate (u) at (1.5,1);
                \coordinate (d) at (1.5,-1);
                \coordinate (blu) at ($0.5*(l)+0.5*(u)$);
                \coordinate (bur) at ($0.5*(r)+0.5*(u)$);
                \coordinate (brd) at ($0.5*(r)+0.5*(d)$);
                \coordinate (bld) at ($0.5*(l)+0.5*(d)$);
                \coordinate (blm) at ($0.5*(l)+0.5*(m)$);
                \coordinate (brm) at ($0.5*(r)+0.5*(m)$);
                 \coordinate (bdm) at ($0.5*(d)+0.5*(m)$);
                 \coordinate (bum) at ($0.5*(u)+0.5*(m)$);

                \coordinate(blum) at (${1/3}*(l)+{1/3}*(u)+{1/3}*(m)$);
                \coordinate(bldm) at (${1/3}*(l)+{1/3}*(d)+{1/3}*(m)$);
                \coordinate(brdm) at (${1/3}*(r)+{1/3}*(d)+{1/3}*(m)$);
                \coordinate(brum) at (${1/3}*(r)+{1/3}*(u)+{1/3}*(m)$);
                \draw[fill = NewBlue]{}(l) -- (u) -- (r) -- (d) -- (l);
                \doubLine{NewBlue}{l}{u};
                  \doubLine{NewBlue}{u}{r};
                   \doubLine{NewBlue}{r}{d};
                \doubLine{NewBlue}{d}{l};
                \doubLine{NewBlue}{u}{d};
                \doubLine{green}{l}{r};
                
                \doubLine{NewBlue}{l}{blum};
                \doubLine{NewBlue}{u}{blum};
                \doubLine{NewBlue}{m}{blum};
                 \doubLine{NewBlue}{blu}{blum};
                \doubLine{NewBlue}{bum}{blum};
                \doubLine{NewBlue}{blm}{blum};
                \doubLine{NewBlue}{l}{bldm};
                \doubLine{NewBlue}{d}{bldm};
                \doubLine{NewBlue}{m}{bldm};
                   \doubLine{NewBlue}{bld}{bldm};
                \doubLine{NewBlue}{blm}{bldm};
                \doubLine{NewBlue}{bdm}{bldm};
                \doubLine{NewBlue}{r}{brdm};
                \doubLine{NewBlue}{d}{brdm};
                \doubLine{NewBlue}{m}{brdm};
                \doubLine{NewBlue}{brd}{brdm};
                \doubLine{NewBlue}{bdm}{brdm};
                \doubLine{NewBlue}{brm}{brdm};
                \doubLine{NewBlue}{r}{brum};
                \doubLine{NewBlue}{u}{brum};
                \doubLine{NewBlue}{m}{brum};
                    \doubLine{NewBlue}{bur}{brum};
                \doubLine{NewBlue}{bum}{brum};
                \doubLine{NewBlue}{brm}{brum};
                 \draw[draw = none, decoration={markings,mark=at position 0.6 with {\arrow[scale = 1.5]{latex}}},witharrow] (l) -- (blu);
                \draw[draw = none, decoration={markings,mark=between positions 0.5 and 0.6 step 0.1 with {\arrow[scale = 1.5]{latex}}},witharrow] (blu) -- (u);
                \draw[draw = none, decoration={markings,mark=at position 0.6 with {\arrow[scale = 1.5]{latex}}},witharrow] (l) -- (bld);
                 \draw[draw = none, decoration={markings,mark= between positions 0.5 and 0.6 step 0.1 with {\arrow[scale = 1.5]{latex}}},witharrow] (bld) -- (d);
                 \draw[draw = none, decoration={markings,mark= between positions 0.4 and 0.6 step 0.1 with {\arrow[scale = 1.5]{latex}}}, witharrow] (bur) -- (u);
                \draw[draw = none, decoration={markings,mark= between positions 0.4 and 0.6 step 0.1 with {\arrow[scale = 1.5]{latex}}}, witharrow] (brd) -- (d);
                \draw[draw = none, decoration={markings,mark= between positions 0.4 and 0.8 step 0.1 with {\arrow[scale = 1.5]{latex}}}, witharrow] (r) -- (brd);
                 \draw[draw = none, decoration={markings,mark= between positions 0.4 and 0.8 step 0.1 with {\arrow[scale = 1.5]{latex}}}, witharrow] (r) -- (bur);
                \draw[color = black, fill = red, ](l) circle (2pt);
                \node (x) at (0,0){ \textcolor{white}{x}}   ; 
                \draw[color = black, fill = red, ](r) circle (2pt);
                \node (x) at (3,0){ \textcolor{white}{x}}   ;
                \draw[color = black, fill = green] (m) circle (2pt);
                \draw[color = black, fill = NewBlue] (u) circle (2pt);
                \node (y) at (u){ \textcolor{white}{y}}   ; 
                \draw[color = black, fill = NewBlue] (d) circle (2pt);
                \node (y) at (d){ \textcolor{white}{y}}   ; 
                \draw[color = black, fill = NewBlue] (blu) circle (2pt);
                 \node (a) at (blu){ \textcolor{white}{a}}   ; 
                \draw[color = black, fill = NewBlue] (bur) circle (2pt);
                \node (b) at (bur) { \textcolor{white}{b}}   ; 
                \draw[color = black, fill = NewBlue] (brd) circle (2pt);
                 \node (b) at (brd){ \textcolor{white}{b}}   ; 
                \draw[color = black, fill = NewBlue] (bld) circle (2pt);
                \node (a) at (bld){ \textcolor{white}{a}}   ; 

                \draw[color = black, fill = green] (blm) circle (2pt);
                \draw[color = black, fill = green] (brm) circle (2pt);
                \draw[color = black, fill = NewBlue] (bdm) circle (2pt);
                \draw[color = black, fill = NewBlue] (bum) circle (2pt);
                \draw[color = black, fill = NewBlue] (blum) circle (2pt);
                \draw[color = black, fill = NewBlue] (bldm) circle (2pt);
                \draw[color = black, fill = NewBlue] (brdm) circle (2pt);
                \draw[color = black, fill = NewBlue] (brum) circle (2pt);
            \end{tikzpicture}   
            \begin{center}
\begin{tikzpicture}[witharrow/.style={postaction={decorate}}, scale = 1.3, 
dot/.style = {circle, fill, minimum size=#1, inner sep=0pt, outer sep=0pt},
dot/.default = 2pt  
]
\newcommand{\doubLine}[3]{\draw[color = black, line width = 2.5pt] (#2) -- (#3);
                 \draw[color = #1, line width = 1.5pt] (#2) -- (#3)}
                \begin{scope}[xshift = -3.5cm]
               \coordinate (l) at (0,0);
                \coordinate (m) at (1.5,0);
                \coordinate (r) at (3,0);
                \coordinate (u) at (1.5,1);
                \coordinate (d) at (1.5,-1);
                \coordinate (blu) at ($0.5*(l)+0.5*(u)$);
                \coordinate (bur) at ($0.5*(r)+0.5*(u)$);
                \coordinate (brd) at ($0.5*(r)+0.5*(d)$);
                \coordinate (bld) at ($0.5*(l)+0.5*(d)$);
                \coordinate (blm) at ($0.5*(l)+0.5*(m)$);
                \coordinate (brm) at ($0.5*(r)+0.5*(m)$);
                 \coordinate (bdm) at ($0.5*(d)+0.5*(m)$);
                 \coordinate (bum) at ($0.5*(u)+0.5*(m)$);

                \coordinate(blum) at (${1/3}*(l)+{1/3}*(u)+{1/3}*(m)$);
                \coordinate(bldm) at (${1/3}*(l)+{1/3}*(d)+{1/3}*(m)$);
                \coordinate(brdm) at (${1/3}*(r)+{1/3}*(d)+{1/3}*(m)$);
                \coordinate(brum) at (${1/3}*(r)+{1/3}*(u)+{1/3}*(m)$);
                
                \draw[fill = NewBlue]{}(l) -- (blu) -- (blum) -- (blm) -- (bldm) -- (bld) -- (l);
                 \draw[fill = NewBlue]{}(r) -- (bur) -- (brum) -- (brm) -- (brdm) -- (brd) -- (r);
                \doubLine{NewBlue}{l}{blu};
                \doubLine{NewBlue}{l}{bld};
                \doubLine{NewBlue}{l}{bldm};
                \doubLine{NewBlue}{l}{blum};
                \doubLine{NewBlue}{blum}{blu};
                \doubLine{NewBlue}{bldm}{bld};
                
                \doubLine{NewBlue}{brum}{bur};
                \doubLine{NewBlue}{brdm}{brd};
                \doubLine{NewBlue}{r}{bur};
                \doubLine{NewBlue}{r}{brd};
                \doubLine{NewBlue}{r}{brdm};
                \doubLine{NewBlue}{r}{brum};
                \doubLine{green}{l}{blm};
                  \doubLine{green}{r}{brm};
                   \doubLine{NewBlue}{blum}{blm};
                     \doubLine{NewBlue}{bldm}{blm};
                  \doubLine{NewBlue}{brdm}{brm};
                  \doubLine{NewBlue}{brum}{brm};

                 \draw[draw = none, decoration={markings,mark=at position 0.6 with {\arrow[scale = 1.5]{latex}}},witharrow] (l) -- (blu);
\draw[draw = none, decoration={markings,mark=at position 0.6 with {\arrow[scale = 1.5]{latex}}},witharrow] (l) -- (bld);
 \draw[draw = none, decoration={markings,mark= between positions 0.4 and 0.8 step 0.1 with {\arrow[scale = 1.5]{latex}}}, witharrow] (r) -- (brd);
\draw[draw = none, decoration={markings,mark= between positions 0.4 and 0.8 step 0.1 with {\arrow[scale = 1.5]{latex}}}, witharrow] (r) -- (bur);
                 \draw[color = black, fill = red, ](l) circle (2pt);
                \node (x) at (0,0){ \textcolor{white}{x}}   ; 
                \draw[color = black, fill = red, ](r) circle (2pt);
                \node (x) at (3,0){ \textcolor{white}{x}}   ;
                 \draw[color = black, fill = NewBlue] (blu) circle (2pt);
                 \node (a) at (blu){ \textcolor{white}{a}}   ; 
                \draw[color = black, fill = NewBlue] (bur) circle (2pt);
                \node (b) at (bur) { \textcolor{white}{b}}   ; 
                \draw[color = black, fill = NewBlue] (brd) circle (2pt);
                 \node (b) at (brd){ \textcolor{white}{b}}   ; 
                \draw[color = black, fill = NewBlue] (bld) circle (2pt);
                \node (a) at (bld){ \textcolor{white}{a}}   ; 
  \draw[color = black, fill = NewBlue] (blum) circle (2pt);
 \draw[color = black, fill = NewBlue] (bldm) circle (2pt);
                \draw[color = black, fill = NewBlue] (brdm) circle (2pt);
                \draw[color = black, fill = NewBlue] (brum) circle (2pt);
                 \draw[color = black, fill = green] (blm) circle (2pt);
                \draw[color = black, fill = green] (brm) circle (2pt);
                    \end{scope}

                       \begin{scope}[xshift = 3.5cm]
                      \coordinate (l) at (0,0);
                \coordinate (m) at (1.5,0);
                \coordinate (r) at (3,0);
                \coordinate (u) at (1.5,1);
                \coordinate (d) at (1.5,-1);
                \coordinate (blu) at ($0.5*(l)+0.5*(u)$);
                \coordinate (bur) at ($0.5*(r)+0.5*(u)$);
                \coordinate (brd) at ($0.5*(r)+0.5*(d)$);
                \coordinate (bld) at ($0.5*(l)+0.5*(d)$);
                \coordinate (blm) at ($0.5*(l)+0.5*(m)$);
                \coordinate (brm) at ($0.5*(r)+0.5*(m)$);
                 \coordinate (bdm) at ($0.5*(d)+0.5*(m)$);
                 \coordinate (bum) at ($0.5*(u)+0.5*(m)$);

                \coordinate(blum) at (${1/3}*(l)+{1/3}*(u)+{1/3}*(m)$);
                \coordinate(bldm) at (${1/3}*(l)+{1/3}*(d)+{1/3}*(m)$);
                \coordinate(brdm) at (${1/3}*(r)+{1/3}*(d)+{1/3}*(m)$);
                \coordinate(brum) at (${1/3}*(r)+{1/3}*(u)+{1/3}*(m)$);
                \draw[fill = NewBlue]{}(l) -- (bum) -- (r) -- (bdm) -- (l);
                \doubLine{green}{l}{r};
                
                \doubLine{NewBlue}{l}{blum};
                \doubLine{NewBlue}{m}{blum};
                \doubLine{NewBlue}{bum}{blum};
                \doubLine{NewBlue}{blm}{blum};
                \doubLine{NewBlue}{l}{bldm};
                \doubLine{NewBlue}{m}{bldm};
                \doubLine{NewBlue}{blm}{bldm};
                \doubLine{NewBlue}{bdm}{bldm};
                \doubLine{NewBlue}{r}{brdm};
                \doubLine{NewBlue}{m}{brdm};
                \doubLine{NewBlue}{bdm}{brdm};
                \doubLine{NewBlue}{brm}{brdm};
                \doubLine{NewBlue}{r}{brum};
                \doubLine{NewBlue}{m}{brum};
                \doubLine{NewBlue}{bum}{brum};
                \doubLine{NewBlue}{brm}{brum};
                \doubLine{NewBlue}{m}{bum};
                \doubLine{NewBlue}{m}{bdm};
                \draw[color = black, fill = green] (blm) circle (2pt);
                \draw[color = black, fill = green] (brm) circle (2pt);
                \draw[color = black, fill = NewBlue] (bdm) circle (2pt);
                \draw[color = black, fill = NewBlue] (bum) circle (2pt);
                \draw[color = black, fill = NewBlue] (blum) circle (2pt);
                \draw[color = black, fill = NewBlue] (bldm) circle (2pt);
                \draw[color = black, fill = NewBlue] (brdm) circle (2pt);
                \draw[color = black, fill = NewBlue] (brum) circle (2pt);
                \draw[color = black, fill = green] (m) circle (2pt);
                 \draw[color = black, fill = red, ](l) circle (2pt);
                \node (x) at (0,0){ \textcolor{white}{x}}   ; 
                \draw[color = black, fill = red, ](r) circle (2pt);
                \node (x) at (3,0){ \textcolor{white}{x}}   ;
                \end{scope}

              \begin{scope}[xshift = 0cm]
               \coordinate (l) at (0,0);
                \coordinate (m) at (1.5,0);
                \coordinate (r) at (3,0);
                \coordinate (u) at (1.5,1);
                \coordinate (d) at (1.5,-1);
                \coordinate (blu) at ($0.5*(l)+0.5*(u)$);
                \coordinate (bur) at ($0.5*(r)+0.5*(u)$);
                \coordinate (brd) at ($0.5*(r)+0.5*(d)$);
                \coordinate (bld) at ($0.5*(l)+0.5*(d)$);
                \coordinate (blm) at ($0.5*(l)+0.5*(m)$);
                \coordinate (brm) at ($0.5*(r)+0.5*(m)$);
                 \coordinate (bdm) at ($0.5*(d)+0.5*(m)$);
                 \coordinate (bum) at ($0.5*(u)+0.5*(m)$);

                \coordinate(blum) at (${1/3}*(l)+{1/3}*(u)+{1/3}*(m)$);
                \coordinate(bldm) at (${1/3}*(l)+{1/3}*(d)+{1/3}*(m)$);
                \coordinate(brdm) at (${1/3}*(r)+{1/3}*(d)+{1/3}*(m)$);
                \coordinate(brum) at (${1/3}*(r)+{1/3}*(u)+{1/3}*(m)$);
                
                \draw[fill = NewBlue]{}(l) -- (blum) -- (blm) -- (bldm) --  (l);
                 \draw[fill = NewBlue]{}(r)-- (brum) -- (brm) -- (brdm) -- (r);
                \doubLine{NewBlue}{l}{bldm};
                \doubLine{NewBlue}{l}{blum};
                \doubLine{NewBlue}{r}{brdm};
                \doubLine{NewBlue}{r}{brum};
                \doubLine{green}{l}{blm};
                  \doubLine{green}{r}{brm};
                   \doubLine{NewBlue}{blum}{blm};
                     \doubLine{NewBlue}{bldm}{blm};
                  \doubLine{NewBlue}{brdm}{brm};
                  \doubLine{NewBlue}{brum}{brm};

                 \draw[color = black, fill = red, ](l) circle (2pt);
                \node (x) at (0,0){ \textcolor{white}{x}}   ; 
                \draw[color = black, fill = red, ](r) circle (2pt);
                \node (x) at (3,0){ \textcolor{white}{x}}   ;
  \draw[color = black, fill = NewBlue] (blum) circle (2pt);
 \draw[color = black, fill = NewBlue] (bldm) circle (2pt);
                \draw[color = black, fill = NewBlue] (brdm) circle (2pt);
                \draw[color = black, fill = NewBlue] (brum) circle (2pt);
                 \draw[color = black, fill = green] (blm) circle (2pt);
                \draw[color = black, fill = green] (brm) circle (2pt);
                    \end{scope}
            \end{tikzpicture}   
            \end{center}          
    \caption{The upper left corner shows a stratified cell structure for the pinched torus, stratified over the poset $\{ {\color{red} 0} < {\color{green} 1} < {\color{NewBlue} 2} \}$.
    Vertices with the same name, and edges with the same markings are being identified and the stratification is indicated by the coloring. To its right, a barycentric subdivision $\Psi$ of this cell structure is shown. In the following row there are illustrations of the subcomplexes $\PsiStanHoods[\I]{\str}$ for $\I = \{{\color{red} 0} < {\color{NewBlue} 2}\}, \{ {\color{red} 0} < {\color{green} 1} < {\color{NewBlue} 2} \}, \{{\color{green} 1} < {\color{NewBlue} 2} \}$.}
    \label{fig:enter-label}
\end{figure}
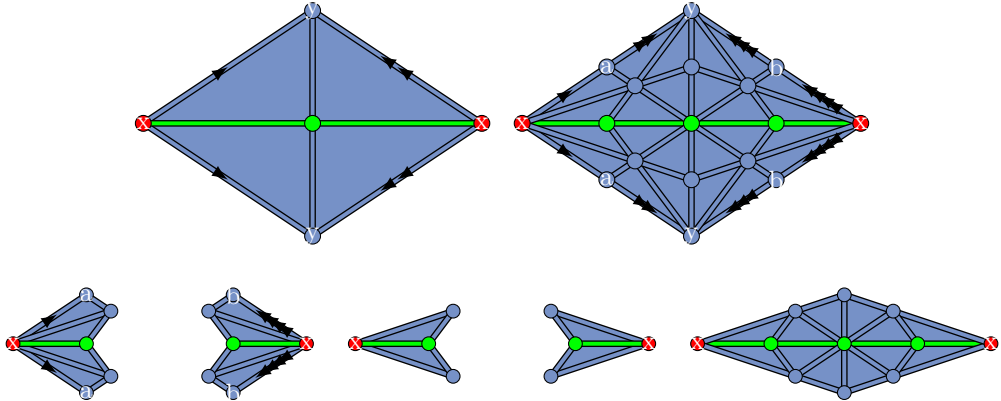
In fact, we show that the subcomplexes $\PsiStanHoods[\I]{\str}$ can even be used to model the whole homotopy link diagram of \cite{douteauEnTop}.
\cref{thm:hol_main_result} has the following corollary, which is central to the construction of semimodel categories of stratified spaces in \cite{TSHHWa}. 
\begin{ctheorem}{B}[\cref{cor:hol_diag_is_ho_pushout}]\label{thm:hol_main_res_B}
    Let $P$ be a poset and $\I = \{p_0 < \dots < p_n\} \subset P$ a finite increasing sequence. Consider a pushout diagram of $P$-stratified cell complexes
        \begin{diagram}\label{diag:lem_mod_for_ho_push_half_intro}
        {\tstr[A]} \arrow[r,"c", hook] \arrow[d, "f"]& {\tstr[B]} \arrow[d] \\
    {\tstr[X]} \arrow[r, hook]  & {{\tstr[Y]}},
    \end{diagram}
    with  $c$ an inclusion of a stratified subcomplex. Then the induced diagram of spaces 
    \begin{diagram}
        \HolIP {\tstr[A]} \arrow[r] \arrow[d]& \HolIP {\tstr[B]} \arrow[d] \\
    \HolIP {\tstr[X]} \arrow[r]  & \HolIP {{\tstr[Y]}},
    \end{diagram}
    obtained by taking generalized homotopy links, is homotopy cocartesian.
\end{ctheorem}
In \cite{TSHHWa} we use the latter result to construct new (semi-)model structures for stratified homotopy theory in which classical examples of stratified spaces such as Whitney stratified spaces are bifibrant. Furthermore, we derive from this semimodel structure a version of the stratified homotopy hypothesis pertaining to a conjecture of \cite{AFRStratifiedHomotopyHypothesis} (\cite[Thm. B]{TSHHWa}).
\subsection{Overview of the article}
It is a well-known classical result that the (pairwise) homotopy link of a piecewise linear two strata stratified space may be computed in terms of the boundary of a regular neighborhood of the lower stratum. Equivalently, one may take the homotopy type of the complement of the lower stratum in the regular neighborhood. In \cite{quinn1988homotopically} the author generalized this result to more general topological notions of regular neighborhood $X_p \subset N \subset X$ of a stratum $X_p$, which admit a so-called \define{nearly stratum-preserving deformation retraction} (this nomenclature was first used in \cite{miller2013}, Quinn speaks of \textit{tame inclusions of strata}). These are given by homotopies 
\[
R \colon N \times [0,1] \to X
\]
such that $R \times (0,1]$ is stratum-preserving, $R_1$ is the inclusion $N \hookrightarrow X$, $R$ is constant on $X_p$ and such that $R_0$ has image entirely in $X_p$. 
In the case of the realization of a stratified-simplicial set, such regular neighborhoods can be obtained by first taking a barycentric subdivision, and then taking the union of closed simplices intersecting the lower stratum (see \cite{quinn1988homotopically}).
The goal here is a two-fold generalization. Firstly, we aim to replace pairwise homotopy links with generalized homotopy links, obtained by replacing the stratified interval with stratified simplices. Secondly, we generalize from stratified simplicial sets to stratified cell complexes.
The case of generalized homotopy links of stratified simplicial sets was studied in detail in \cite{douteauwaas2021}.
To further generalize these results to stratified cell complexes, we need to generalize the notions of neighborhood and nearly stratum-preserving deformation retraction to the $n$-strata case. We proceed to do so in the following steps.
\begin{enumerate}
    \item In \cref{subsec:strata_neighborhoods}, we introduce the notion of a \define{system of strata-neighborhoods} of a stratified space ${\tstr}$. These are defined in a way that they allow for a computation of homotopy links close to a singularity. The ultimate goal of the theory is to identify a class of such neighborhood systems for which the topology of the neighborhoods themselves may be used to compute the homotopy type of homotopy links, and there is no need to pass to path-spaces. Such neighborhood systems are called \define{homotopy link models}. Following from this, our strategy of proof is then to show that every stratified cell complex may be equipped with a homotopy link model, and furthermore that this can be done in a way that is compatible with pushout diagrams. 
    \item A first step lies in constructing the so-called \define{standard neighborhood-systems} for realizations of stratified simplicial sets (\cref{subsec:aspire_of_simplicial}). We show that these standard neighborhoods turn out to be universal in some sense (\cref{prop:universality_of_standard}). This result is crucial to obtain homotopy link models that are compatible with given maps of stratified cell complexes. Furthermore, it provides a new proof of a locality principle for homotopy links of strata-neighborhoods, which ultimately provides a more conceptual proof of \cite[Thm. 4.8]{douteauwaas2021}.
    \item In \cref{subsec:nbhd_for_cell}, we then use the results on strata-neighborhoods of stratified simplicial sets to generalize the construction of standard neighborhood systems to stratified cell complexes.
    \item Having strengthened our understanding of strata-neighborhood systems, we return to the notion of a homotopy link model. At this point, we have all the results necessary available to prove that the \define{regular complement diagram} (\cref{con:nbhd_comp_diag}) associated to a homotopy link model is weakly equivalent to the diagram of homotopy links of the associated stratified space (\cref{prop:models_model_globally}).  
    \item Finally, it remains to prove that the standard neighborhood systems we have constructed for stratified cell complexes are homotopy link models. To accomplish this, we first define an adaptation to the $n$-strata scenario of the notion of nearly stratum-preserving neighborhood retracts (as they were defined in \cite{quinn1988homotopically}) a so-called \aspire{} (see \cref{def:aspire}).
    In \cref{subsec:def_aspire}, we show that the existence of these \aspires{} provides a way to guarantee that a neighborhood system is a homotopy link model (\cref{prop:computing_holinks_via_aspire}).
    \item In \cref{subsec:aspire_of_simplicial}, we then return to the standard neighborhood systems of realization of stratified simplicial sets and show that these can be equipped with \aspires{} (\cref{prop:stan_aspire_is_aspire}). In particular, this result has \cite[Thm. 4.8]{douteauwaas2021} as a corollary.
    \item Finally, in \cref{subsec:aspire_of_cell}, we generalize the results of the previous section to standard neighborhood systems of stratified cell complexes.
    To do so, we first provide a technical gluing lemma that ultimately allows a cell-by-cell construction of \aspires{} (\cref{lem:gluing_aspires}). It then remains to investigate the case of a single cell (\cref{lemma:extending_aspire_on_simplex}), to finish the proof that standard-neighborhood systems provide homotopy link models (\cref{prop:aspire_for_finite_cell}). 
\end{enumerate}
\section{Basic notions: From homotopy links to their models}
The goal of this section is to introduce the basic objects and notions under investigation. We begin by recalling the necessary language and notation from stratified homotopy theory \cref{subsec:language_and_notation} in particular the notion of homotopy link \cref{subsec:generalized_homotopy_links}. We then introduce the central objects of study to this paper: Stratified cell complexes (\cref{subsec:strat_cell_complex}). Our goal is to find convenient models for the homotopy links of such stratified cell complexes. We make this idea rigorous with the notions of strata-neighborhood systems and homotopy link models in \cref{subsec:strata_neighborhoods}.
\subsection{Language and notation}\label{subsec:language_and_notation}
Let us first fix some language and notation pertaining to stratified homotopy theory, mostly lifted from \cite{douSimp,douteauEnTop,douteauwaas2021,haine2018homotopy}.
\begin{notation}
    We will denote simplicial categories in the form $\underline{\cat[S]}$ and their underlying categories in the form $\cat[S]$. Given two categories $\cat[C]$ and $\cat[D]$, we denote by $\FunC (\cat[C], \cat[D])$ the category of functors between the two categories.
\end{notation}
\begin{notation} We use the following notation for categories of simplicial sets.
\begin{itemize}
  \item  We denote by $\Delta$ the category of finite, linearly ordered posets of the form $[n] := \{ 0, \cdots ,n \}$, for $n \in \mathbb N$.
     \item We denote by $\sSet$ the simplicial category of simplicial sets, i.e., the category of set-valued presheaves on $\Delta^{\op}$, equipped with the canonical simplicial structure induced by the product (see \cite{HigherTopos} for all of the standard notation used for simplicial sets).
    \item When we treat $\sSet$ as a (simplicial) model category, this will generally be with respect to the Kan-Quillen model structure (see \cite{Quillen}), unless otherwise noted. When we use Joyal's model structure for quasi-categories (\cite{joyalNotes}) instead, we will denote this model category by $\sSetJoy$. 
\end{itemize}
\end{notation}
\begin{notation}\label{not:notation_top_spaces}
    $\TopN$ is going to denote either of the following three categories of topological spaces. \begin{enumerate}
        \item The category of \define{all topological spaces}, which we will also refer to as \define{general} topological spaces.
        \item The category of compactly generated topological spaces, i.e., such spaces that have the final topology with respect to compact Hausdorff spaces (see, for example, \cite{Rezk2017COMPACTLYGS}).
        \item The category of $\Delta$-generated topological spaces, i.e., such spaces which have the final topology with respect to realizations of simplices, or equivalently just with respect to the unit interval (compare \cite{duggerDelta,GaucherDelta}).
    \end{enumerate}
    We denote by $\real{-} \colon \sSetN \to \TopN$ the realization functor of simplicial sets and by $\Sing \colon \TopN \to \sSetN$ its right adjoint.
    $\TopN$ naturally carries the structure of a simplicial category, tensored and powered over $\sSet$, induced by left Kan extension of the construction
    \[
    T \otimes \Delta^n := T \times \real{\Delta^n}.
    \]
    \end{notation}
    We denote the resulting simplicial category by $\Top$.
    Furthermore, we will always consider $\Top$ to be equipped with the Quillen-model structure \cite{Quillen}, which makes $\real{-} \dashv \Sing$ a simplicial Quillen equivalence, between $\Top$ and $\sSet$, which creates weak equivalences in both directions.
    \begin{remark}
    Note that one commonly only defines the simplicial structure for compactly or $\Delta$-generated spaces. This is, however, mostly due to the fact that, for general topological spaces $T$ and infinite simplicial sets $K$, the tensoring $T\otimes K$ does not agree with the inner product $T \times \real{K}$. Instead, it is given by a colimit of products of $T$ with the simplices of $K$. Similarly, the power $T^{K}$ is not given by an internal mapping space (which does not necessarily exist for arbitrary $K$) but by the limit of the mapping spaces of simplices of $K$, equipped with the compact open topology. The reason we do not take the regular approach of restricting to one fixed convenient category of topological spaces is that much of the literature has been formulated for the $\Delta$-generated case, while we will make several arguments in the category of general topological spaces later on, which seem to lack an internal analogue in the category of $\Delta$-generated spaces. Note that from a homotopy-theoretic perspective these choices in set-theoretic-topological foundations are usually not relevant, as any space is canonically weakly equivalent to its $\Delta$-fication (compactly generated replacement). Furthermore, all our results concern spaces in the $\Delta$-generated category (which is included in the other two categories) and the choice of larger framework is thus mostly inessential.
    \end{remark}
    For the remainder of this section, we fix some category of topological spaces $\TopN$ as in \cref{not:notation_top_spaces}.
    \begin{notation} We are going to use the following terminology and notation for partially ordered sets, drawn partially from \cite{douSimp} and \cite{haine2018homotopy}:
\begin{itemize}
    \item  We denote by $\Pos$ the category of partially ordered sets, with morphisms given by order-preserving maps.
    \item  We consider $\Delta$ as a subcategory of $\Pos$ in the obvious fashion.
    Given $P \in \Pos$, we denote by $\catFlag$ the slice category $\Delta_{/ \pos}$. 
    That is, objects are given by arrows $[n] \to \pos$ in $\Pos$, $n \in \mathbb N$, and morphisms are given by commutative triangles.
    \item We denote by $\catRFlag$ the \define{subdivision of } $\pos$, given by the full subcategory of $\catFlag$ of such arrows $[n] \to \pos$, which are injective.
    \item The objects of $\catFlag$ are called \define{flags of} $\pos$. We represent them by strings $[p_0 \leq \cdots \leq p_n]$, of $p_i \in \pos$. 
    \item Objects of $\catRFlag$ are called regular \define{flags of} $\pos$. We represent them by strings $\standardFlag$, of $p_i \in \pos$.
\end{itemize}
\end{notation}
    \begin{notation}\label{not:top_strat_spaces}
    Having fixed a category of topological spaces $\TopN$, we then use the following notation for stratified topological spaces (all of these constructions can be found in \cite{douteauEnTop} among other places).
    \begin{itemize}
    \item We think of the $1$-category $\Pos$ as naturally embedded in $\TopN$, via the Alexandrov topology functor, equipping a poset $P$ with the topology where the closed sets are given by the downward closed sets. By abuse of notation, we just write $\pos$, for the Alexandrov space corresponding to it (compare \cite[Def. 2.2]{douteauwaas2021}). 
    \item For $\pos \in \Pos$, we denote by $\TopPN$ the slice category $\TopN_{/\pos}$.
    \item Objects of $\TopPN$ are called \define{$\pos$-stratified spaces}. They are given by a tuple $(T,s\colon T \to \pos)$. We will usually use calligraphic letters for stratified spaces and stick to the notational convention $\tstr =(\utstr, \ststr)$ to refer to the underlying space and the stratification.
    \item  Morphisms in $\TopPN$ are called \define{stratum-preserving maps.} 
   \item Given a map of posets $f\colon Q \to \pos$ and $\tstr \in \TopPN$, we denote by $f^*\tstr \in \TopPN[Q]$ the stratified space $\utstr \times_{\pos} Q \to Q$. We are mostly concerned with the case where $f$ is given by the inclusion of a singleton $\{p \}$, of a subset $\{q \sim p \mid q \in \pos \}$, for $p \in \pos$ and $\sim$ some relation on the partially ordered set $\pos$ (such as $\leq$), or more generally a subposet $Q \subset \pos$. We then write $\tstr_{p}$ (or, respectively, $\tstr_{\sim p}$, $\tstr_Q$) instead of $f^*\tstr$. The spaces $\tstr_{p}$, for $p \in \pos$ are called the \define{strata of} $\tstr$.
    \end{itemize}
\end{notation}
\begin{notation}
     Throughout this article, we will consider a series of subspaces of stratified spaces in the category $\TopN$, that is, use maps on them that are not stratum-preserving. We keep following the convention (\cref{not:notation_top_spaces}), which is that calligraphic letters indicate the stratified context, and regular letters the non-stratified one. For example, for $\tstr \in \TopPN$ and $p \in P$, $\tstr_{\leq p}$ is the stratified space over $\{q \in P \mid q \leq p\}$, given by restricting $\tstr$ and $\utstr_{\leq p}$ is its underlying topological space. Note that in the case of the strata there is no notational conflict with using both $\tstr_p$ and $\utstr_p$, if we identify stratified spaces over a poset with one element with topological spaces. This type of notational convention reaches its syntactic limits when applied to expressions such as $\sReal{\Delta^\I}$, which do not have calligraphic counterparts. In this case, we will simply write $\real{\Delta^\I}$ to indicate the underlying topological space.
\end{notation}
\begin{notation}
    We use the following terminology and notation for (stratified) simplicial sets, drawn partially from \cite{douSimp} and \cite{haine2018homotopy}:
    \begin{itemize}
    \item We think of $\Pos$ as being naturally embedded in $\sSetN$, via the nerve functor (compare \cite{haine2018homotopy}). By abuse of notation, we just write $\pos$, for the simplicial set given by the nerve of $\pos \in \Pos$. 
    \item For $\pos \in \Pos$, we denote by $\sSetPN$ the slice category $\sSetN_{/\pos}$, which is equivalently given by the category of set valued presheaves on $\catFlag$. 
    \item Objects of $\sSetPN$ are called \define{$\pos$-stratified simplicial sets}. They are given by a tuple $\str =(\ustr,\sstr \colon \ustr \to \pos)$.
    \item  Morphisms in $\sSetPN$ are called \define{stratum-preserving simplicial maps.} 
    Simplicial homotopies in $\sSetPN$ are called \define{stratified simplicial homotopies.}
    Simplicial homotopy equivalences in $\sSetPN$ are called \define{stratum-preserving simplicial homotopy equivalences}.
    \item Given a map of posets $f\colon Q \to \pos$ and $\str \in \sSetPN$, we denote by $f^*\str \in \sSetPN[Q]$ the stratified simplicial set $\ustr \times_{\pos} Q \to Q$. We are mostly concerned with the case where $f$ is given by the inclusion of a singleton $\{p \}$, of a subset $\{q \sim p \mid q \in \pos \}$, for $p \in \pos$ and $\sim$ some relation on the partially ordered set $\pos$ (such as $\leq$), or more generally, a subposet $Q \subset \pos$. We then write $\str_{p}$ (or, respectively, $\str_{\sim p}$, $\str_Q$) instead of $f^*\str$. The simplicial sets $\str_{p}$, for $p \in \pos$ are called the \define{strata of} $\str$. 
    \item For a flag $\J = [p_0 \leq \cdots \leq p_n] \in \Delta_P$, we write $\Delta^\J$ for the image of $\J$ in $\sSetPN$ under the Yoneda embedding $\catFlag \hookrightarrow \sSetPN$. Equivalently, $\Delta^\J$ is given by the unique simplicial map $\Delta^{n} \to \pos$ mapping $i \mapsto p_i$. $\Delta^\J$ is called the \define{stratified simplex associated to} $\J$.
    \item \label{holCC:not:flag_operations} Using the fully faithful (and continuous) embedding $\catFlag \hookrightarrow \sSetPN$, we extend the base change notation for stratified simplicial sets to flags.  That is, for $f \colon Q \to \pos$ we write $f^*\J$ for the unique flag of $Q$ corresponding to $f^*(\Delta^\J)$. We use the same shorthand notation for subsets $Q \subset \pos$. For example $\J_{\leq p}$ is the flag obtained from $\J$ by removing all entries not lesser equal to $p$.
     \item Given a stratified simplex $\Delta^\J$, for $\J = [p_0 \leq \cdots \leq p_n]$, we write $\partial \Delta^\J$ for its \define{stratified boundary}, given by the composition $\partial \Delta^n \to \Delta^n \to \pos$.
    \end{itemize}
    \end{notation}
     \begin{recollection}[\cite{douSimp}]
        For fixed $P \in \Pos$, the two categories $\TopPN$ and $\sSetPN$ are connected through a singular simplicial set, realization style adjunction, denoted 
        \begin{align*}
            \sReal{-} \colon \sSetPN \rightleftharpoons \TopPN \colon \SingS .
        \end{align*}
        The left adjoint is constructed by mapping a stratified simplex $\Delta^\J \to \pos$, with $\J = [ p_0 \leq \dots \leq p_n]$, to the stratified space
        \begin{align*}
            \real{\Delta^n} &\to \pos \\
            x &\mapsto \sup\{p_i \in \J \mid x_i > 0\},
        \end{align*}
        where we consider $\real{\Delta^n}$ as embedded in $\mathbb R^{n+1} \cong \mathbb R^{\J}$. 
        If we consider $\sSetPN$ as the category of set-valued presheaves on $\Delta_P$, then by the logic of a nerve and realization functor, $\SingS \tstr$ is hence given the stratified simplicial set
        \[
        \SingS\tstr(\J) = \TopPN(\sReal{\Delta^\J}, \tstr)
        \]
        with the obvious structure morphisms.
    \end{recollection}
    \subsection{Generalized homotopy links}\label{subsec:generalized_homotopy_links}
    Homotopy links were originally introduced in \cite{quinn1988homotopically} order to obtain a homotopy-theoretic replacement for the boundary of a regular neighborhood in the piecewise-linear scenario. As functors, they may be understood as the right adjoint to taking products with stratified simplices.
    \begin{notation} 
    In the following subsections, we will make frequent use of the action of $\TopN$ on $\TopPN$, given by 
    \begin{align*}
        \TopN \times \TopPN &\to \TopPN \\
     (T, \str) &\mapsto T \times \str := (T \times X \xrightarrow{\pi_X} X \xrightarrow{\sstr[X]} P).
    \end{align*}
    In case there is any possibility of confusion, the stratification always arises from the second component.
\end{notation}
\begin{recollection}[See \cite{douteauEnTop}]
    Given a (locally compact in the case of general topological spaces) stratified space $\tstr[S]$, the functor 
    \[
    - \times \tstr[S] \colon \TopN \to \TopPN
    \]
    admits a right adjoint. It is given by equipping $\TopPN(\tstr[S], \tstr)$ with the respective subspace topology (depending on the choice of category $\TopN$) of the space of all continuous maps, equipped with the compact-open topology. We are particularly interested in the case $\tstr[S] = \sReal{\Delta^\I}$, for $\I \in \sd(\pos)$ a regular flag. 
    For $\tstr \in \TopPN$, the image under the right adjoint to $- \times \sReal{\Delta^\I}$ is called the \define{$\I$-th (generalized) homotopy link of $\tstr$}.
    Explicitly, it is given by topologizing the set of stratum-preserving maps 
    \[
    \{ \sReal{\Delta^\I} \to \tstr \}
    \]
    as described above. 
    We then can summarize the homotopy links in a global functor 
    \[
        \HolIP[] \colon \TopPN \to \Fun ( \sd(\pos) ^{\op}, \TopN),
    \]
    with structure maps of the diagram $\HolIP (\str)$, $\HolIP[\I](\tstr) \to \HolIP[\I'](\tstr)$ given by restricting along the inclusion $\sReal{\Delta^{\I'}} \subset \sReal{\Delta^\I}$, for $\I' \subset \I$.
\end{recollection}
\begin{example}
    If $\I = [p]$ is a singleton, then $\HolIP{\tstr}$ is naturally homeomorphic to the stratum $\tstr_p$. For $\I = [p_0 < p_1]$ a pair, the homotopy link $\HolIP{\tstr}$ is the space of paths starting in $\tstr_p$ and immediately exiting into $\tstr_q$, so-called \define{exit paths} defined in \cite{quinn1988homotopically}.
\end{example}
\begin{example}
    Let $\I \in \sd(\pos)$ and $\str \in \sSetPN$ be a stratified simplicial set. We can consider the first barycentric subdivision of the underlying simplicial set $\sd \ustr$.
    The vertices of $\sd(\pos)$ correspond to pairs $( \Delta^\J \to \str, \J')$ with $\J' \subset \J \in \Delta_P$ and $\Delta^\J \to \str$ non-degenerate.
    We may then consider the full subsimplicial set of $\sd \ustr$ spanned by such vertices for which $\J'$ degenerates from $\I$.
    The latter is called the \define{simplicial} link, denoted $\LinkI{\str}$.
    In the case of two strata, for $\I = [p < q]$, $\real{\LinkI}$ is precisely the boundary of a regular neighborhood of $(\sReal{\str})_p$.
    In particular, in this case, it is weakly equivalent to the homotopy link of $\sReal{\str}$. In \cite{douteauwaas2021}, we proved the case of general $\I$ of this result obtaining weak homotopy equivalences
    \[
    \HolIP[\I]{\sReal{\tstr}} \simeq \real{ \LinkI{\str}}.
    \]
\end{example}
\subsection{Stratified cell complexes}\label{subsec:strat_cell_complex}
Let us now move on to the case of more general stratified cell complexes
\begin{definition}\label{def:cell_struct}
Let $\str \in \TopPN$. 
A \define{stratified cell structure on $\tstr$} is a family of stratum-preserving maps $(\sigma_i \colon \sReal{\Delta^{\J_{i}}} \to \tstr)_{i \in I}$ such that the following properties hold.
\begin{enumerate}
    \item $\utstr$ has the final topology with respect to the maps $\sigma_i$.
    \item For each $i \in I$, $\sigma_i$ induces a homeomorphism from $\sReal{\Delta^{\J_i}} \setminus \sReal{\partial \Delta^{\J_i}}$ onto its image. Denote these images by $e_i$. Furthermore, denote the image of $\sReal{\partial \Delta^\J_i}$ under $\sigma$ by $\partial e_i$.
    \item $\utstr$ is given by the (set-theoretic) disjoint union of the cells $e_i$.
    \item Denote by $\prec$ the relation on $I$, which is generated under transitivity by
    \[
        i \prec j \iff  e_i \cap \partial e_j \neq \emptyset.
        \]
    Then $\prec$ is irreflexive, and every element of $I$ only has finitely many precursors with respect to $\prec$.
\end{enumerate}
A stratified Hausdorff space $\str$ together with a stratified cell structure $(\sigma_i)_{i \in I}$ on it is called a \define{ stratified cell complex}. A \define{stratified subcomplex} of $(\tstr,(\sigma_i)_{i \in I})$, is a stratified subspace $\str[A] \subset \str$, together with a subset $I' \subset I$ such that $I'$, $\tstr[A] = \bigcup_{i \in I'} e_i$, and such that $I'$ is closed below under $\prec$.
\end{definition}
\begin{remark}
    In many respects, stratified cell complexes behave much like their unstratified counterparts. In particular, it is not hard to see that every stratified subcomplex is closed, and itself a stratified cell complex, with the induced cell structure.
    We will usually refer to a stratified cell complex just by its underlying stratified space, and keep the cell structure implicit. At times, we will say $\str$ is a stratified cell complex, to refer to the fact that it is Hausdorff and admits a stratified cell structure.
\end{remark}
\begin{example}
    Every realization of a stratified simplicial set $\str \in \sSetPN$ naturally inherits the structure of a stratified cell complexes, with cells given by the realizations of non-degenerate simplices $\Delta^\J \to \str$.
\end{example}
If we forget about the cell structure, stratified cell complexes are simply the spaces that arise as absolute cell complexes (in the sense of \cite{hirschhornModel}) from the set of stratified boundary inclusions $\{ \sReal{\partial \Delta^\J} \hookrightarrow \sReal{\Delta^\J} \mid \J \in \Delta_P \}$. 
\begin{proposition}
    Let $\str \in \TopPN$. Then the following are equivalent:
    \begin{enumerate}
        \item \label{item:prop_char_abs_cell_cplx_1} $\tstr$ is an absolute cell complex with respect to the set $\{ \sReal{ \partial \Delta^\J \hookrightarrow \Delta^\J} \mid \J \in \Delta_P \}$.
        \item \label{item:prop_char_abs_cell_cplx_2} $\tstr$ is Hausdorff and admits a stratified cell structure.
    \end{enumerate}
    Furthermore, the following relative version of this result holds. Suppose that $\tstr[A]$ is a stratified cell complex. Then, for a stratified map $i \colon \tstr[A] \to \tstr$, the following are equivalent.
     \begin{enumerate}
        \item \label{item:prop_char_abs_cell_cplx_3} $i$ is a relative cell complex with respect to the set $\{ \sReal{ \partial \Delta^\J \hookrightarrow \Delta^\J} \mid \J \in \Delta_P \}$.
        \item \label{item:prop_char_abs_cell_cplx_4} $\tstr$ is Hausdorff and admits a stratified cell structure, which makes $\tstr[A]$ a stratified subcomplex.
    \end{enumerate}
\end{proposition}
\begin{proof}
    Essentially, this argument is identical with the non-stratified case. For some reason, we were unable to find a reference for this in the literature. 
    First, note that being an absolute cell complex as above implies that $\str$ is Hausdorff. Indeed, this already holds for absolute cell complexes in the Quillen model structure, and may be seen by extending disjoint opens cell by cell, via a transfinite inductive argument. The remaining translation of structures is handled by \cref{con:equ_versions_of_strat_cell_cplx} below.
\end{proof}
\begin{construction}\label{con:equ_versions_of_strat_cell_cplx}
    Let $\str \in \TopPN$.
    If we write $\emptyset \to \tstr$ as a transfinite compositions 
    \[
        \tstr^0 \to \tstr^1 \to \dots \to \tstr^{\alpha} = \tstr
    \]
    with pushout diagrams
    \begin{diagram}
        \sReal{\partial \Delta^{\J_{\beta}}} \arrow[r, hook] \arrow[d]& \sReal{\Delta^{\J_{\beta}}} \arrow[d]\\
        \tstr^{\beta} \arrow[r, hook] & \tstr^{\beta +1},
    \end{diagram}
    then the compositions
    \[
    \sReal{\Delta^{\J_{\beta}}} \to \tstr^{\beta} \to \tstr
    \]
    define a cell structure on $\tstr$. The finite precursor conditions follows from the fact that every compactum in an absolute cell complex is contained in a finite subcomplex (see  \cite[Prop. 10.7.4]{hirschhornModel}, for the topological case). Conversely, if $\tstr$ admits a stratified cell structure, $(\sigma_i)_{i \in I}$, then we may extend the precursor order on $I$ to a well-order as follows: By well-founded induction, we obtain an order-preserving map $\nu \colon I \mapsto \mathbb N$, inductively defined via $\tau \mapsto \sup \{\nu(\tau') \mid \tau' \prec \tau \}$. Then, for every fiber $\nu^{-1}(n)$, choose a well order $\prec_n$, and finally equip
    \[
    I = \bigsqcup_{n \in \mathbb N} \nu^{-1}(n)
    \]
    with the lexicographic order \[
   i < p \iff  i \prec j \lor ( \nu(i)= \nu(j) =n \land i \prec_n j).\]
   Under this construction, we can identify $I$ as an ordinal $\alpha_I$ . Then, for $\beta \leq \alpha_I$ denote $\tstr^{\beta} = \bigcup_{j < i} e_\beta \subset \tstr $. By construction, $\sReal{\partial \Delta^{\J_\beta}} \hookrightarrow \sReal{\Delta^{\J_\beta}} \xrightarrow{\sigma_\beta} \tstr$ factors through $\tstr^\beta$, and one may check that the diagrams 
    \begin{diagram}
        \sReal{\partial \Delta^{\J_{\beta}}} \arrow[r, hook] \arrow[d]& \sReal{\Delta^{\J_{\beta}}} \arrow[d]\\
        \tstr^{\beta} \arrow[r, hook] & \tstr^{\beta +1},
    \end{diagram}
    are pushout and that whenever $\beta$ is a limit element, then $\tstr^{\beta} = \colim_{\beta' < \beta} \tstr^{\beta'}$. The relative case is essentially analogous.
\end{construction}
As a consequence of \cite[Prop. 10.7.6]{hirschhornModel}, one obtains the following. 
\begin{lemma}\label{lem:finite_subcomplex}
    If $\tstr \in \TopPN$ is a stratified cell complex, then every compactum $K \subset \tstr$ is contained in a finite subcomplex $\tstr[A]$ of $\tstr$.
\end{lemma}
\subsection{Systems of strata-neighborhoods}\label{subsec:strata_neighborhoods}
One of the central properties of pairwise homotopy links is that they may be computed locally using neighborhoods of a stratum (compare \cite{quinn1988homotopically}). This type of argument was also central in the proof of \cite[Thm 4.8]{douteauwaas2021} .
In this subsection, we introduce the notion of a strata-neighborhood system of a stratified space $\tstr$ which provide a general framework for these type of local computations. We then move on to the question of when the homotopy links may instead be computed entirely in terms of these strata-neighborhood systems, in which case we speak of a homotopy link model.
\begin{definition}\label{def:DeltaP_neighborhood}
    Let $\tstr \in \TopPN$ and $S \subset \tstr$. A \define{$\Delta_P$-neighborhood of $S$ in $\tstr$} is a subset $U \subset \utstr$ such that for any flag $\J$ of $P$ and any stratum-preserving map $\sigma \colon \sReal{\Delta^\J} \to \tstr$ the set $\sigma^{-1}(U)$ is a neighborhood of $\sigma^{-1}(S)$ in $\sReal{\Delta^\J}$.
\end{definition}
The following elementary property follows immediately from the definition of a $\Delta_P$-neighborhood.
\begin{lemma}
    Let $f \colon \tstr \to {\tstr[Y]}$ in $\TopPN$. If $U \subset {\tstr[Y]}$ is a $\Delta_P$-neighborhood of $S \subset {\tstr[Y]}$, then $f^{-1}(U)$ is a $\Delta_P$-neighborhood of $f^{-1}(S)$ in ${\tstr[Y]}$.
\end{lemma}
Roughly speaking, we should think of $\Delta_P$-neighborhoods as subsets of $\tstr$ that look like a neighborhood if one takes the perspective of a finite stratified cell complex. This heuristic is made rigorous by the following lemma.
\begin{lemma}\label{lem:neighborhoods_of_finite_complexes}
    If $\tstr \in \TopPN$ is equipped with the structure of a finite stratified cell complex and  $S \subset U \subset \utstr$, then the following are equivalent:
    \begin{enumerate}
        \item $U$ is a $\Delta_P$-neighborhood of $S$.
        \item For every cell $\sigma \colon \sReal{\Delta^\J} \to \tstr$, the set $\sigma^{-1}(U)$ is a neighborhood of $\sigma^{-1}(S)$.
        \item $U$ is a neighborhood of $S$.
    \end{enumerate}
\end{lemma}
\begin{proof}
    That the first condition implies the second is trivial. Clearly, also the third implies the first. It remains to show that the second condition implies the third. This is the content of \cref{lemma:neighborhoods_of_pushouts}.
\end{proof}
For infinite stratified cell complexes we may still state the following lemma.
\begin{lemma}\label{lem:char_of_nbhd_in_cell_complex}
    If $\tstr$ admits the structure of a cell complex, then $U \subset \tstr$ is a $\Delta_P$-neighborhood of $S \subset U \subset \tstr$, if and only if the inverse image under every cell $\sigma \colon \sReal{\Delta^\J} \to \tstr$ of $U$ is a neighborhood of $\sigma^{-1}(S)$ in $\real{\Delta^\J}$.
\end{lemma}
\begin{proof}
    The only if case is immediate by the definition of a $\Delta_P$-neighborhood. Now, for the converse, note that any continuous map $\sReal{\Delta^\J} \to \tstr$ factors through a finite subcomplex $\tstr[A]$ of $\tstr$. It follows from this that it suffices to show that $U \cap A$ is a $\Delta_P$-neighborhood of $S \cap A$, for any finite subcomplex $\tstr[A]$ of $\tstr$. Therefore, the result follows from \cref{lem:neighborhoods_of_finite_complexes}.
\end{proof}


Next, let us define the notion of a strata-neighborhood system. 
\begin{definition}
    A \define{strata-neighborhood system} of a stratified space $\tstr \in \TopPN$ is a family of subspaces $\snSys =(\snVal(p))_{p \in P}$ such that $\snVal(p)$ is a $\Delta_{P}$-neighborhood of $\utstr_p$ in $\tstr$ that further fulfills $\utstr_{\leq p} \subset \snVal(p)$.
\end{definition}
\begin{remark}
The additional requirement that $\utstr_{\leq p} \subset \snVal(p)$ is only there for the sake of notational convenience, when passing to flags in \cref{not:flag_nbhd}. Aside from this, this is immaterial for the constructions in this section. We decided to add this condition, as it is automatically fulfilled for the strata-neighborhood systems we construct in this article, and can always be achieved by replacing $\snVal(p)$ with $\snVal(p) \cup \utstr_{<p}$.
\end{remark}
The goal of a strata-neighborhood system is ultimately to provide a geometric model for the homotopy link diagram of a stratified space. However, we do not only need to model the homotopy links, but also the functoriality of the latter. This requires the following definition.
\begin{definition}
Let $\tstr,{\tstr[Y]} \in \TopPN$ be equipped with strata-neighborhood systems $\snSys$ and $\snSys[{\tstr[Y]}]$, respectively.  
We say $f \colon {\tstr} \to {{\tstr[Y]}}$ \define{lifts to a map of strata-neighborhood systems $\snSys \to \snSys[{\tstr[Y]}]$}, if $f(\snVal(p)) \subset U_{{\tstr[Y]}}(p)$, for all $p \in \pos$. 
\end{definition}
\begin{notation}
    Denote by $\NCat$ the category which is defined as follows. Objects are given by pairs $({\tstr},\snSys)$ with ${\tstr} \in \TopPN$ and $\snSys$ a strata-neighborhood system of ${\tstr}$. The set of morphisms from $({\tstr},\snSys)$ to $({{\tstr[Y]}},\snSys[{\tstr[Y]}])$ is given by such stratum-preserving maps $f \colon {\tstr} \to {{\tstr[Y]}}$ which lift to a map of neighborhood systems $\snSys \to \snSys[{\tstr[Y]}]$. We will usually just refer to a pair $({\tstr},\snSys)$ just by $\snSys$.
\end{notation}
To extract an object of $\DiagTop$ from a strata-neighborhood system, we need the following construction.
\begin{notation}\label{not:flag_nbhd}
    For $\snSys$ a strata-neighborhood system of ${\tstr} \in \TopPN$ and $\I  \in \sd(\pos)$ a regular flag, we write $\snVal(\I):= \bigcap_{p \in \I} \snVal(p) $.
\end{notation}
\begin{notation}
    Given a strata-neighborhood system $\snSys$ a strata-neighborhood system of ${\tstr} \in \TopPN$, and $\I \in \sd(\pos)$, we are going to follow our conventions on stratified and non-stratified objects (see \cref{not:top_strat_spaces}), and write $\snVals(\I)$, for the $P$-stratified space obtained by equipping $\snVal(\I)$  with the stratification inherited from $\tstr$. In particular, for $p \in P$, it makes sense to use expressions such as $\snVal(\I)_{\leq p}$, which in this case refers to the subspace of $\snVal(\I)$, given by the strata of index lesser or equal to $p$.
\end{notation}
\begin{construction}\label{con:nbhd_comp_diag}
    Let $\snSys$ be a strata-neighborhood system for ${\tstr} \in \TopPN$. We denote by $\NDiagT(\snSys) \in \DiagTop$ the diagram
    \[
    \I = \standardFlag \mapsto \snVal(\I)_{\geq p_n}
    \]
    with structure maps given by inclusions. The resulting object $\NDiagT( \snSys  )\in \DiagTop$ is called the \define{regular complement diagram} of ${\tstr} \in \TopPN$.
    This construction defines a functor 
    \[
    \NDiagT \colon \NCat \to \DiagTop.
    \]
    \end{construction}
    We may now ask the question under which conditions on a neighborhood system $\snSys$, on a stratified space ${\tstr}$, there is a canonical weak equivalence of diagrams between $\HolIP[]({\tstr})$ and $\NDiagT(\snSys)$. 
\begin{definition}
    Let ${\tstr} \in \TopPN$.
    We say that a strata-neighborhood system $\snSys$ of ${\tstr}$ \define{is a model for the homotopy links of ${\tstr}$} (is a homotopy link model) if the natural maps 
    \[
    \HolIP( \snVals(\I)) \xrightarrow{\ev} \HolIP[{p_n}]( \snVals(\I)) = \snVal(\I)_{p_n} ,
    \]
    and
    \[
    \snVal(\I)_{p_n} \hookrightarrow \snVal(\I)_{\geq p_n}
    \]
    are weak equivalences of topological spaces, for each regular flag $\I = \standardFlag \subset P$. \\
    We denote by $\HoModCat$ the full subcategory of $\NCat$ given by pairs $({\tstr}, \snSys)$ with $\snSys$ a homotopy link model for ${\tstr}$.
\end{definition}
\begin{remark}
We should note that the second condition in the definition of a homotopy link model is necessary since (using the notation of \cref{con:nbhd_comp_diag}) $\NDiagT(\snSys)( \I)$ is defined as $\snVal(\I)_{\geq p_n}$ rather than just using $\snVal(\I)_{p_n}$. This in turn is necessary if we want $\NDiagT(\snSys)$ to provide a model for the whole diagram $\HolIP[]({\tstr})$, not just its pointwise values. To obtain the structure maps of $\NDiagT(\snSys)$ we need to have inclusions $\NDiagT(\snSys)(\I_1) \subset \NDiagT(\snSys)(\I_0)$, whenever $\I_0 \subset \I_1$. In particular, if we also want this inclusion to hold when $\I_0$ and $\I_1$ do not share a maximal element, then we need to define $\NDiagT$ as in \cref{con:nbhd_comp_diag}. In any case, the second condition for being a homotopy link model will usually turn out to be the easy one to verify, as it can be verified entirely in the language of classical homotopy theory, unlike the first one which requires stratified considerations.
\end{remark}
Let us now finish this subsection by stating the result that legitimizes the nomenclature of homotopy link models.
\begin{theorem}\label{prop:models_model_globally}
    The diagram
\begin{diagram}
    \HoModCat \arrow[d] \arrow[r, hook] & \NCat \arrow[d, "\NDiagT"] \\
    \TopPN \arrow[r, "{\HolIP[]}"] & \DiagTop
\end{diagram}
commutes up to weak equivalence.
\end{theorem}
 At this point, we do not yet have the tools necessary for a proof of \cref{prop:models_model_globally}. The proof follows in \cref{subsec:models_are_models}. 
\section[SNSs for stratified simplicial sets and cell complexes]{Strata-neighborhood systems for stratified simplicial sets and cell complexes}
From the perspective of \cref{prop:models_model_globally}, the goal of this paper is to construct homotopy link models for stratified cell complexes. In this section, we provide a general construction for strata-neighborhood systems, first for the simplicial case (\cref{subsec:nbhd_for_simplicial}) and later the case of stratified cell complexes (\cref{subsec:nbhd_for_cell}). Importantly, we prove that this construction can be made compatible with stratum-preserving maps, which ultimately leads to a proof of \cref{prop:models_model_globally}. 
\subsection{Standard strata-neighborhood systems for stratified simplicial sets}\label{subsec:nbhd_for_simplicial}
Before we move on to investigating strata-neighborhood systems on stratified cell complexes, let us first consider the simpler case of the realization of a stratified simplicial set. To do so, we are going to make heavy use of the following coordinates.
\begin{construction}\label{con:strata_coordinates}
    Let $\J$ be a flag in $P$.
    For $p \in P$, we denote $\J_p$ the unique maximal subflag of $\J$ that degenerates from the regular flag $[p]$ (see \cref{holCC:not:flag_operations} for an overview). 
    The inclusion $\J_{p} \subset \J$ induces a natural projection 
    \[
    \mathbb R^\J \to \mathbb R^{\J_p},
    \]
    where $\mathbb R^{\J}$ denotes the vector space spanned by the elements of $\J$ (counted with repetition). Next, consider the canonical embedding $\sReal \Delta^\J \hookrightarrow \mathbb R^{\J}$, which embeds $\sReal{\Delta^\J}$ as the affine hull of the standard basis vectors. For $x \in \sReal{\Delta^\J}$, we write $x_p$ for the image of $x$ under the composition
    \[
    \sReal{\Delta^\J} \hookrightarrow \mathbb R^\J \to \mathbb R^{\J_p}.
    \]
    Now, if $\I = [p_0 < \cdots < p_n ] $ is the regular flag such that all elements of $\J$ are contained in $\I$, then $\J$ is equivalently given by the concatenation
    \[
    \J_{p_0} \cup \cdots \cup \J_{p_n}.
    \]
    (Note that we may indeed allow $p_i$ that are not in $\J$, as then $\J_{p_i}$ is the empty flag.) In particular, we have a canonical isomorphism
    \[
    \mathbb R^{\J} \cong \Pi_{i \in [n]} \mathbb R^{\J_{p_i}},
    \]
    which allows us to make sense of the expression
    \[
    x = (x_{p_0}, \cdots, x_{p_n}).
    \]
    Further, recall from \cref{holCC:not:flag_operations} that $\J_{\leq p}$ denotes the maximal subflag given by such entries of $\J$, of value lesser or equal to $p$. Analogously, we denote by $\J_{\not \leq p}$ the maximal subflag given by such entries of $\J$, of value not lesser equal to $p$, and so on. Just as in the case of a singleton, we denote by $x_{\leq p}$ the image of $x \in \sReal{\Delta^\J}$ under
    \[
    \sReal{\Delta^\J} \hookrightarrow \mathbb R^\J \to \mathbb R^{\J_{\leq p}},
    \]
    and similarly proceed with $x_{\not \leq p}$, $x_{< p}$ and so on.
    \end{construction}
    Next, we need a normalized version of the coordinates defined in \cref{con:strata_coordinates}, so-called join coordinates. The notation here will be somewhat sloppy, in the sense that we are often going to write expressions like $y_p$, when we formally should be writing $y_p(x)$. 
    \begin{construction}\label{con:join_coordinates}
    We again use the setup of \cref{con:strata_coordinates}.
    At the level of underlying simplicial sets, we may then identify $\Delta^\J$ as the join
    \[
    \Delta^\J = \Delta^{\J_{p_0}} * \cdots * \Delta^{\J_{p_n}}. 
    \]
    Doing so, we can equip $\sReal{\Delta^{\J}} \subset \mathbb R^{\J}$ with $n$-fold join coordinates. Explicitly, the topological join 
     \[
    \real{\Delta^{\J_{p_0}}} * \cdots * \real{\Delta^{\J_{p_n}}}
    \]
    can be described as follows. 
    For $I \subset [n]$, denote by $\real{\Delta^{I}} \subset \real{\Delta^n}$ the face corresponding to $I$. Furthermore, for $I' \subset I$, denote by $\pi_{I,I'}$ the obvious projection $\Pi_{i \in I} \real{\Delta^{\J_{p_i}}} \to \Pi_{i \in I'} \real{\Delta^{\J_{p_i}}}$. Then, the $n$-fold join can be described as the quotient space of 
    \[
    \bigsqcup_{I \subset [n]} (\Pi_{i \in I} \real{\Delta^{\J_{p_i}}}) \times \real{\Delta^{I}}
    \]
    where we mod out by the equivalence relation generated by 
    \[
    (y,s) \sim (\pi_{I,I'}(y), s)   
    \]
    whenever $y \in \Pi_{i \in I} \real{\Delta^{\J_{p_i}}}$ and $s \in \real{\Delta^{I'}} \subset \real{\Delta^{I}}$. The homeomorphism to $\real{\Delta^{\J}}$ is then given by mapping 
    \[
    \real{\Delta^{\J_{p_0}}} * \cdots * \real{\Delta^{\J_{p_n}}}  \ni [(y,s)] \mapsto \sum_{i \in  [n]} s_i y_{i} \in \real{\Delta^\J}.
    \]
    Note that if $y_i$ is not defined, then $s_i = 0$ and this expression makes sense. Conversely, an inverse is obtained by 
    \[
    \real{\Delta^{\J}}\ni x \mapsto \big ((\frac{x_{p_0}}{\real{x_{p_0}}}, \cdots, \frac{x_{p_n}}{\real{x_{p_n}}}), (\real{x_{p_0}}, \cdots, \real{x_{p_n}} ) \big ) \in \real{\Delta^{\J_{p_0}}} * \cdots * \real{\Delta^{\J_{p_n}}}.
    \]
    Again, note that this expression makes sense, even if $\real{x_p} = 0$. This leads us to the following change to join coordinates, which we are going to use frequently in this section. For $x \in \sReal{\Delta^\J}$ we denote
    \begin{align*}
        y_p &:= \frac{x_p}{ \real{x_p}}; \\
        s_p &:= \real{x_p}.
    \end{align*}

    \end{construction}
    Finally, we will need another set of coordinates using the decomposition $x = (x_{\leq p} , x_{\not \leq p})$, for $x \in \sReal{\Delta^\J}$ and $p \in P$.
\begin{construction}\label{con:new_coordinates}
     Again, in the setup of \cref{con:strata_coordinates}, we may - just as in \cref{con:join_coordinates} -  identify
    \begin{align*}
        \Delta^{\I} &= \Delta^{\I_{< p}} * \Delta^{\I_{\not < p}}, \\
        \Delta^{\I} &= \Delta^{\I_{\leq p}} * \Delta^{\I_{\not \leq p}}, \\
        \cdots &= \cdots
\end{align*} etc.
    This leads to a change of coordinates
    \begin{align*}
         s_{< p } := \real{x_{< p }} ; \quad 
        s_{ \not <  p} := \real{x_{\not < p}}; \quad 
        \cdots = \cdots \\
        y_{< p } := \frac{x_{< p }}{\real{x_{< p }}}; \quad 
        y_{\not < p } := \frac{x_{\not < p }}{\real{x_{\not < p}}}; \quad
        \cdots = \cdots 
    \end{align*}
\end{construction}
\begin{remark}\label{remark:role_of_coordinates}
    Let us remark on some immediate facts on the $s$-coordinates of \cref{con:join_coordinates} and \cref{con:new_coordinates}.
    First, note that they are indeed invariant under stratified face and degeneracy maps and therefore extend to any realization of a stratified simplicial set $\str \in \sSetPN$. Then, the $s$-coordinates interact with the stratification of $\sReal{\str}$ as follows. In the following, as the notation $s_{\sReal{\str}}(x)$, for $x \in \sReal{X}$, is rather cumbersome, we will use the shortened notation $s(x):=s_{\sReal{\str}}(x) \in \pos$ to refer to the stratum of $x$.
    Let $x \in \sReal{\str}$. Then, we have equivalences
    \begin{itemize}
        \item $s(x) < p \iff s_{< p} =1 \iff s_{\not <p} =0 $.
        \item $s(x) \leq p \iff s_{\leq p} =1 \iff s_{\not \leq p} =0 $.
        \item $\cdots$.
    \end{itemize} 
    It immediately follows that
    \begin{itemize}
        \item $s(x) = p \iff s_{\leq p} =1 \land s_p >0 \iff  s_{\not \leq p} = 0  \land s_p > 0 \iff \cdots $.  
    \end{itemize}
\end{remark}
We are now going to use these coordinates to define strata-neighborhoods for realizations of stratified simplicial sets.
\begin{construction}\label{con:standard_neighborhoods}
    Suppose $p \in P$ and $\varphi_p: [0,1] \to [0,1]$ is a continuous function such that $\varphi_p(s) >0$, whenever $s >0$. Let $\str \in \sSetPN$. We may then consider the following subspaces of $
    \sReal{\str}$. We set 
    \[
    \phiStanHood{\str} := \{ x \in \sReal{\str} \mid s_{\not \leq p} \leq \varphi_p(\s) s_p  \} 
    \]
    and call it the $\varphi_p$-\define{standard neighborhood associated to $\str$}. 
    Note that, since all $s$-coordinates  are invariant under realizations of maps of stratified simplicial sets, this construction extends to a functor
    \begin{align*}
        \phiStanHoods{-} \colon \sSetPN \to \TopPN.
    \end{align*}
\end{construction}
\begin{example}\label{example:standard_hood_of_strat_simplex}Consider the stratified simplex $\sReal{\Delta^{[ {\color{red}0} < {\color{green}1} < {\color{NewBlue}2}]}}$, pictured below with the strata colored in red, green and blue, in ascending order. If we set $\varphi_p = 1$ for $p= 0,1$ we obtain the following standard neighborhoods shaded in red and green respectively for $p=0,1$. 
   
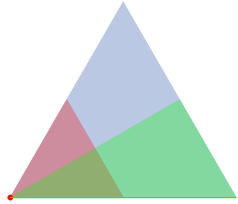
\begin{center}
        \begin{tikzpicture}[scale = 3]
        \newcommand{\bartwo}[2]{($(#1)!0.5!(#2)$)}
       \coordinate (l) at (0,0);
       \coordinate (r) at (1,0);
       \coordinate (u) at (0.5,{0.5*sqrt(3)});
       \draw[fill, color = NewBlue, opacity = 0.5] (l) -- (r) -- (u) -- (l);
         \draw[color = green] (l) to (r);
       \draw[fill, color = red, opacity = 0.3] (l) -- \bartwo{l}{r} -- \bartwo{l}{u} -- (l);
               \draw[color = red, fill = red] (l) circle (0.3pt);
        \draw[fill, color = green, opacity = 0.3] (l) -- \bartwo{r}{u} -- (r) -- (l);
        \draw[color = green] (l) to (r);
   \end{tikzpicture}
   \end{center}
   For a smaller choice of $\varphi_1$, here with $\varphi_1(0) = 0$, we obtain a $\varphi_1$-standard neighborhood whose boundary is tangential to the $1$-stratum at the $0$-stratum:
    
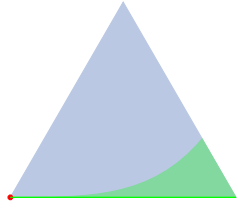
\begin{center}
        \begin{tikzpicture}[scale = 3]
       \def\bartwo(#1)(#2){($(#1)!0.5!(#2)$)}
       \coordinate (l) at (0,0);
       \coordinate (r) at (1,0);
       \coordinate (u) at (0.5,{0.5*sqrt(3)});
       \draw[fill, color = NewBlue, opacity = 0.5] (l) -- (r) -- (u) -- (l);
         \draw[color = green] (l) to (r);
           \draw[color = red, fill = red] (l) circle (0.3pt);
        \draw[fill, color = green, opacity = 0.3] (l) to[out = 0, in=-130] ($(u)!0.7!(r)$) -- (r) -- (l);
        \draw[color = green] (l) to (r);
        \end{tikzpicture}
   \end{center}
\end{example}
\begin{proposition}\label{prop:stanhood_is_hood}
    Let $\varphi_p$ be as in \cref{con:standard_neighborhoods}.
    For any stratified simplicial set $\str \in \sSetPN$ the space $\phiStanHood{\str} \subset \sReal{\str}$ defines a $p$-stratum neighborhood.
\end{proposition}
\begin{proof}
    Consider the open subset of $O \subset \phiStanHood{\str} \subset (\sReal{\str})_{\not < p}$ defined by the condition 
    \[
    s_{\not \leq p} < \varphi_p(\s) s_p.
    \]
     By \cref{remark:role_of_coordinates}, for any $x \in (\sReal{\str})_{p}$ the value of $s_{\not \leq p}$ is $0$, $s_p >0$ and $\s \geq s_p >0$. As $\varphi_p(s) >0$ for $s>0$, it follows that 
    \[
     s_{\not \leq p} = 0 < \varphi_p(\s) s_p.
    \]
    Consequently, $(\sReal{\str})_p \subset O \subset \phiStanHood{\str}$ and $\phiStanHood{\str}$ is even a neighborhood of the $p$-stratum in the strong sense.
\end{proof}
In particular, standard neighborhoods allow us to factor stratified realization through the category of strata-neighborhood systems:
\begin{corollary}\label{cor:factorization_through_system}
    Let $\varphi = (\varphi_p)_{p \in P}$ be a family of functions as in \cref{con:standard_neighborhoods}. Then, $\str \mapsto \phiStanHoodSys = (\sReal{\str}, (\phiStanHood{\str})_{p \in P})$ defines a factorization \begin{diagram}
        \sSetPN \arrow[rd, "\sReal{-}"'] \arrow[r,"{\phiStanHoodSys[-]}" ]& \NCat \arrow[d] \\
        & \TopPN \spaceperiod
    \end{diagram}
\end{corollary}
\begin{proof}
    By \cref{prop:stanhood_is_hood}, it suffices to verify that $(\sReal{\str})_{\leq p} \subset \phiStanHood{\str}$, for $\str \in \sSetPN$ and $p \in P$. This is immediate from \cref{remark:role_of_coordinates}.
\end{proof}
\begin{example}
    The most important case to consider is the case where all functions $\varphi_p$ are given by the constant function with value $1$. In this case, the condition for a point $x$ to lie in $\phiStanHood{\str}$ is simply that 
    \[
    s_{\not \leq p} \leq s_p .
    \]
    In this case, we denote the resulting neighborhood system by $\stanHoodSys$ and call it the \define{standard neighborhood system}.
\end{example}
Let us also remark on some of the more degenerate examples of standard neighborhoods: 
\begin{lemma}\label{prop:desc_neighb_notin}
Let $\varphi_p$ be as in \cref{con:standard_neighborhoods}.
    If $p \notin \J$, then $\phiStanHoods{\Delta^{\J}} = \sReal{\Delta^{\J_{<p}}}$.
\end{lemma}
\begin{proof}
    Indeed, when $p \notin \J$, then for any $x \in \sReal{\Delta^\J}$ it holds that $s_p = 0$. Hence, the defining condition for $\phiStanHood{\Delta^{\J}}$ is fulfilled if and only if $s_{\not \leq p} = 0$, that is, when $x \in (\sReal{\Delta^\J})_{\leq p} = \sReal{\Delta^{\J_{<p}}}$.
\end{proof}
Next, we give a purely combinatorial description of the standard neighborhood in the special case where $\str$ is a stratified simplicial complex, by making use of barycentric subdivisions. By a \define{stratified simplicial complex}, we mean a stratified simplicial set, such that its underlying simplicial set has the property that every non-degenerate simplex is uniquely determined by its set of vertices. In particular, this means that every face of every non-degenerate simplex is non-degenerate.
\begin{construction}\label{con:simplicial_model_of_standard_prec}
    For $\str \in \sSetPN$, consider the first barycentric subdivision of $\ustr[\str]$, equipped with the stratification induced by the last vertex map $\sd \ustr \to \ustr$, and denote it $\sd \str$ (see \cite[Def. 3.7]{douteauwaas2021}).
    If $\str = \Delta^{\J}$, for some flag $\J$ of $P$, then we denote by $\simStanHood{\Delta^\J} \subset \sd \Delta^\J$ the subcomplex, spanned by those vertices that correspond to the subflags $\J'$ of $\J$ such that \[
     p \in \J' \quad \lor \quad \forall q \in \J': q < p.   
    \]
    We denote by $\simStanHoods{\Delta^\J}$ the stratified simplicial set obtained by equipping $\simStanHood{\Delta^\J}$, with the stratification inherited from $\sd \Delta^\J$.
    This construction is functorial with respect to maps of stratified simplices and hence extends to a functor 
    \[
    \simStanHoods{-} \colon \sSetPN \to \sSetPN,
    \]
    via left Kan extension, together with a natural transformation $\simStanHoods{-} \hookrightarrow \sd$, identifying $\simStanHoods{\str}$ with a stratified subsimplicial set of $\sd \str$, for $\str \in \sSetPN$.
\end{construction}
\begin{construction}\label{con:simplicial_model_of_standard}
    If $\str \in \sSetPN$ is a stratified simplicial complex (i.e., every non-degenerate simplex in $\str$ is uniquely determined by its set of vertices), then we may identify $\stanHoods{\str}$ with the realization of $\simStanHoods{\str}$\footnote{This also works in the general simplicial set case. However, the constructions become significantly more involved, and we have no need for this case here.}.
    A stratum-preserving homeomorphism $\sReal{\simStanHoods{\str}} \to \stanHoods{\str}$ is constructed as follows.
    On each stratified simplex $\sReal{\Delta^\J}$, where $\J$ degenerates from a regular flag $\standardFlag$, consider the weighted barycenter $b_{\J}$, given in join coordinates by 
    \[ 
    b_{\J} := [ (b_{p_0} , \cdots, b_{p_n}), (\frac{1}{2}, \frac{1}{4}, \cdots, \frac{1}{2^n}, \frac{1}{2^n} )],
    \] 
    where $b_{p_i}$ is the barycenter of $\sReal{\Delta^{\J_{p_i}}}$. 
    We denote \[\Psi: \sReal{\sd \str} \to \sReal{\str} \] the stratum-preserving homeomorphism which is affinely extended from the map on vertices
     \[
     v \mapsto \sReal{\sigma_v}(b_{\J}),
     \]
    where $v$ corresponds to a non-degenerate simplex $\sigma_{v} \colon \Delta^\J \to \str$ with flag $\J$. Then, the content of \cref{prop:simplicial_standard_model} is that $\Psi$ restricts to a homeomorphism from $\sReal{\simStanHoods{\str}}$ to $\stanHoods{\str}$. 
\end{construction}
\begin{example}
In the context of \cref{ex:standardhood_of_regular_simplex}, the following figure shows the weighted barycentric subdivision given by \cref{con:simplicial_model_of_standard}. As indicated in the following picture, the standard neighborhoods of the $0$ and $1$-stratum are precisely spanned by such vertices in the subdivision fulfilling the condition of \cref{con:simplicial_model_of_standard_prec}.
  \begin{center}
        \begin{tikzpicture}[scale = 3]
        \newcommand{\bartwo}[2]{($(#1)!0.5!(#2)$)}
       \coordinate (l) at (0,0);
       \coordinate (r) at (1,0);
       \coordinate (u) at (0.5,{0.5*sqrt(3)});
        \coordinate (lr) at ($(l)!0.5!(r)$); 
    \coordinate (lu) at ($(l)!0.5!(u)$); 
    \coordinate (ru) at($(r)!0.5!(u)$); 
    \coordinate (lru) at ($(l)!0.5!(ru)$); 
       \draw[fill, color = NewBlue, opacity = 0.5] (l) -- (r) -- (u) -- (l);
         \draw[color = green] (l) to (r);
         \draw[color = NewBlue] (l) -- (lru);
        \draw[color = NewBlue] (r) -- (lru);
         \draw[color = NewBlue] (u) -- (lru);
          \draw[color = NewBlue] (lr) -- (lru);
           \draw[color = NewBlue] (lu) -- (lru);
             \draw[color = NewBlue] (ru) -- (lru);
       \draw[fill, color = red, opacity = 0.3] (l) -- \bartwo{l}{r} -- \bartwo{l}{u} -- (l);
               \draw[color = red, fill = red] (l) circle (0.3pt);
        \draw[fill, color = green, opacity = 0.3] (l) -- \bartwo{r}{u} -- (r) -- (l);
        \draw[color = green] (l) to (r);
   \end{tikzpicture}
   \end{center}
\end{example}
\begin{example}\label{ex:standardhood_of_regular_simplex}
    For a non-degenerate flag $\I = \standardFlag$ and $\str \in \sSetPN$, we denote by $\simStanHoods[\I]{\str}$ the intersection $\bigcap_{p \in \I} \simStanHoods{\str}$.
    If $\str= \Delta^\I$, then $\simStanHoods[\I]{\str}$ is given by the image of the unique embedding $\Delta^\I \hookrightarrow \sd \Delta^\I$.
\end{example}
The following proposition then shows that the construction in \cref{con:simplicial_model_of_standard} does indeed provide a combinatorial model for standard neighborhoods.
\begin{proposition}\label{prop:simplicial_standard_model}
    Let $p \in \pos$ and $\str \in \sSetPN$ be a stratified simplicial complex. Then $\Psi \colon  \sReal{\sd \str} \to \sReal{\str} $ - as defined in \cref{con:simplicial_model_of_standard} - restricts to a stratum-preserving homeomorphism $\sReal{\simStanHoods{\str}} \xrightarrow{\sim} \stanHoods{\str}$.
\end{proposition}
\begin{proof}
    The statement is easily reduced to the case where $\str = \Delta^\J$, for $\J$ a flag of $\pos$ that degenerates from a non-degenerate flag $\standardFlag$.
    Let us begin by computing the value of $s_{\not \leq p}(b_{\J})$ and $s_p( b_{\J})$ : If $p = p_n$, then
    \begin{equation}\label{equ:proof_con_simplicial_nbhd_0}
        s_{\not \leq p}(b_{\J}) = 0 < 2^{-n} = s_p( b_{\J}).
    \end{equation}
    If $p = p_k$, for some $k \in [n-1]$, then $b_\J$ fulfills 
    \begin{equation}\label{equ:proof_con_simplicial_nbhd_1}
        s_{\not \leq p}(b_{\J}) = 2^{-(k+1)} = s_p( b_{\J}).
    \end{equation}
    If $p_k < p$, for all $k \in [n]$, then
    \begin{equation}\label{equ:proof_con_simplicial_nbhd_2}
        s_{\not \leq p}( b_{\J} ) = 0 \leq 0 = s_p( b_{\J}).
    \end{equation}
    Finally, if $p \notin \J$ and $k \in [n]$ is minimal with the property that $p_k \not \leq p$, then
     \begin{equation}\label{equ:proof_con_simplicial_nbhd_3}
        s_{\not \leq p}( b_{\J} ) = 2^{-k} > 0 = s_p( b_{\J}).
    \end{equation}
    It follows, from the inequalities \labelcref{equ:proof_con_simplicial_nbhd_0,equ:proof_con_simplicial_nbhd_1,equ:proof_con_simplicial_nbhd_2} that $\Psi$ does indeed restrict to an embedding $\sReal{\simStanHoods{\str}} \to \stanHoods{{\str}}$.
    It remains to show surjectivity of this restriction. So, let $x \in \stanHood{\Delta^\J} \subset \sReal{\Delta^\J}$ be a point in $\sReal{\Delta^\J}$.
    Let $\{ \J_0 \subset \cdots \subset \J_m \}$, be the minimal set of subflags of $\J$ such that $x$ lies in the affine span of $(b_{\J_i})_{i \in m}$, i.e., we have 
    \[
    x = t_0 b_{\J_0} + \cdots + t_m b_{\J_m}
    \]
    with $t_i >0$, for all $i \in [m]$. We need to show that for each $i \in [m]$ either $p \in \J_i$ or $q < p$, for all $q \in \J_i$. In other words, we need to show that the set 
    \[
    S = \{ i \in [m] \mid p \notin \J_i \land \big(\exists q \in \J_i \colon q \not\leq p \big )\}
    \]
    is empty.
    Since $x \in \stanHood{\Delta^\J}$, we have 
    \[
    s_{\not \leq p} (x) \leq s_p(x)
    \]
    and thus 
    \[
    t_0 s_{\not \leq p}(b_{\J_0}) + \cdots + t_m s_{\not \leq p}( b_{\J_m}) =  s_{\not \leq p} (x) \leq s_p(x) =  t_0 s_{p}(b_{\J_0}) + \cdots + t_m s_{p}( b_{\J_m})
    \]
    Using \cref{equ:proof_con_simplicial_nbhd_1,equ:proof_con_simplicial_nbhd_2} it follows that 
    \begin{equation}\label{equ:proof_con_simplicial_nbhd_4}
        \sum_{i \in S} t_i s_{\not \leq p}(b_{\J_i})  \leq \sum_{i \in S} t_is_{p}(b_{\J_i}).
    \end{equation}
    By \cref{equ:proof_con_simplicial_nbhd_3}, the right-hand side of this equality is $0$ and the left-hand side is a (possibly empty) sum of strictly positive numbers. 
    It follows that $S = \emptyset$, as was to be shown.
\end{proof}
We will need the following technical lemma. Roughly speaking, it states that, for a finite simplicial set $\str$, the strata-neighborhood system $\phiStanHoodSys$ of \cref{con:standard_neighborhoods} are universal.
\begin{lemma}\label{prop:universality_of_phi_stan_hood}
    Let $\str \in \sSetPN$ be a finite stratified simplicial set. Let $\snSys[{\sReal{\str}}]$ be any neighborhood system on $\sReal{\str}$. Then there exists a family of functions $\varphi$ as in \cref{cor:factorization_through_system} such that, for any $p \in P$, we have
    \[
    \phiStanHood{\str} \subset  \snVal[\sReal{\str}](p).
    \]
    In other words, the identity on $\sReal{\str}$ lifts to a map of neighborhood systems $\phiStanHoodSys \to  \snSys[\sReal{\str}]$.
\end{lemma}
\begin{proof}
    Note first that it suffices to solve the problem on each simplex and then define $\varphi_p$ by passing to minima. Hence, without loss of generality we may assume that $\str = \Delta^{\J}$, for some flag $\J$. Fix some $p \in P$ and denote $U = U_{\sReal{\Delta^\J}}(p)$. By \cref{lem:neighborhoods_of_finite_complexes}, we may therefore assume that $U$ is actually a neighborhood of the $p$-stratum.
    Next, note that for any $0 < \alpha \leq  1$ the set \[S_{\alpha}:=\{ x \in (\sReal{\Delta^{\J}})_p \mid \s \geq \alpha \}\] is compact. A neighborhood basis for $S_{\alpha}$ in $\sReal{\Delta^\J}$ is given by the sets \[S_{\alpha,\beta}:= \{ x \in \sReal{\Delta^{\J}} \mid \s \geq \alpha - \beta \land s_{\not \leq p} \leq \beta \},\] where $\beta >0$.  
    Hence, for any $n \in \mathbb N$ there exists some $\beta_n >0$ such that for any $x \in \sReal{\str}$, the implication
    \[
    \s \geq \frac{1}{n}  \land s_{\not \leq p} \leq \beta_n \implies x \in U
    \]
    holds. Without loss of generality, we may assume that the sequence $\beta_n$ is decreasing. Choosing a partition of unity $\sigma_n$ on the family $( (\frac{1}{n},1] )_{n \in \mathbb N}$ covering $(0,1]$, we set
    \[ \varphi^{p}(s) = \sum_{n \in \mathbb N} \sigma_n(s) \beta_n.
\]    
    In this fashion, we obtain a continuous function $\varphi^{p} \colon [0,1] \to [0,1]$, which is positive outside of $0$. Now, let $x \in \sReal{\str}$ be such that $\s \in ( \frac{1}{m}, \frac{1}{m-1}]$, for some $m >1$, and suppose that $s_{\not \leq p} \leq \varphi_p(\s)$. Then, since $\sigma_n(s) = 0$ for $s <\frac{1}{n}$, we obtain
    \begin{align*}
         s_{\not \leq p} \leq \varphi_p(\s) &= \sum_{n \in \mathbb N} \sigma_n(\s) \beta_n \\
                                            &= \sum_{n \geq m} \sigma_n(\s) \beta_n  \\
                                            &\leq \beta_m \spacecomma
    \end{align*}
    where the last inequality follows as $\sum_{n \geq m} \sigma_n(\s) \beta_n $ is a convex combination and $(\beta_n)_{n \in \mathbb N}$ is a decreasing sequence.
Thus, it follows that
\[
\s >0 \land s_{\not \leq p} \leq \varphi_p(\s) \implies x \in U,
\]
and in particular also 
\[
\s >0 \land s_{\not \leq p}  \leq \varphi^{p}(\s)s_{p} \implies x \in U,
\]
for any $x \in \sReal{\Delta^{\J}}$. We deduce that 
\[
(\phiStanHood{\Delta^\J})_{\not < p} \subset U.
\]
Since $U$ contains $(\sReal{\str})_{\leq p}$ by assumption, it follows that
\[
\phiStanHood{\Delta^\J}  \subset U
\]
as was to be shown.
\end{proof}
In the next step, we show that (up to a stratum-preserving homeomorphism) we may really replace $\phiStanHoodSys$ by $\stanHoodSys$, making the latter universal among strata-neighborhood systems in this sense.
 Before we do so, let us introduce another set of coordinates, to simplify notation. 
\begin{construction}\label{con:final_coordinates}
    In the framework of \cref{con:new_coordinates} we may repeat the procedure described there with $\Delta^{\J_{\not <p }}$, and decompose 
    \[
    \Delta^{\J_{\not < p}} = \Delta^{\J_p} \ast \Delta^{\J_{\not \leq p}}.
    \]
    We denote the resulting coordinates by 
    \begin{align*}
    t_{p} := \real{y_{p}} ; \quad
    t_{\not \leq  p} := \real{y_{\not \leq  p}} ; \quad
    z_{p} := \frac{y_{p}}{\real{y_{p}}} ; \quad
    z_{\not \leq p} := \frac{y_{\not \leq p}}{\real{y_{\not \leq p}}}. 
    \end{align*}
    Using the affine relations between the several variables (such as $1= s_{<p} + s_{\not < p}$), we may then express $x$ entirely in terms of $y_{<p}$, $z_{p}$, $z_{\not \leq p}$ and $s_{\not < p}$, $t_{ \not \leq p}$. Explicitly, we have
    \[
    x = (1- \s)\yp + \s (\tp \zp + (1-\tp) \zop  ).
    \]
\end{construction}
\begin{remark}
    By construction, whenever $\tp$ is defined, we have an equality
    \[
    \tp = \frac{s_{\not \leq p}}{\s}.
    \]
    Using this, the condition for $x \in \sReal{\str}$ to lie in $\phiStanHood{\str}$ may equivalently be rewritten as
    \[
    \s = 0 \quad \lor  \quad \tp \leq \frac{\varphi_p(\s)}{1+\varphi_p(\s)}.
    \]
\end{remark}
\begin{proposition}\label{prop:universality_of_standard}
    Let $\varphi$ be a system of functions as in \cref{con:standard_neighborhoods}. Then there exists a natural stratum-preserving automorphism 
    \[
    \Phi \colon \sReal{-} \to \sReal{-}
    \]
    that, for each $\str \in \sSetPN$, lifts to a map $\stanHoodSys[\str] \to \phiStanHoodSys[\str]$.
    In particular, $\Phi$ induces a natural transformation
        \[
     \stanHoodSys[-]  \to \phiStanHoodSys[-]
        \]
    in $\NCat$.
    Furthermore, $\Phi$ can be taken naturally stratum-preserving homotopic to the identity, through a family of natural homeomorphisms which lift to maps $\stanHoodSys[\str] \to \stanHoodSys[\str]$. 
\end{proposition}
\begin{proof}
    We use the coordinates as in \cref{con:final_coordinates}. We first define separate homeomorphisms for each $p \in P$, $\Phi_p$.
    Note that by left Kan extension it suffices to construct the natural transformation $\Phi_p$ for stratified simplices. On $\sReal{\Delta^\J}$, we define $\Phi_p$ via 
    \[
    [ (y_{<p}, z_p, z_{\not \leq p}), (\s, \tp) ] \mapsto [ (y_{<p}, z_p, z_{\not \leq p}), (\s, \altp) ] 
    \]
    where 
    \[
     \altp := \begin{cases}
        2\tp - 1 + (2-2\tp)\frac{\varphi_p(\s)}{1+ \varphi_p(\s)} & \textnormal{, for }\tp \geq  \frac{1}{2} \\
        2\tp\frac{\varphi_p(\s)}{1+ \varphi_p(\s)} &\textnormal{, for } \tp \leq  \frac{1}{2}.
    \end{cases}
    \]
    One may easily verify that this assignment is well defined (under the identifications in the join), using the fact that the only coordinate that changes is $\tp$ that $\tp=1 \iff \altp = 1$ and that $\tp=0 \iff \altp = 0$. Similarly, one can easily verify that the resulting map 
    \[
    \Phi_p: \sReal{\Delta^\J} \to \sReal{\Delta^\J}
    \]
    is stratum-preserving. Naturality follows from the fact that both $\s$ and $\tp$ are invariant under stratified face inclusions and degeneracy maps. 
    Let us assume for a second that $p \in \J$. Then, if $\s >0$ and all coordinates except $\tp$ remain fixed, the $\altp$ component of $\Phi_p$ is given by gluing affine homeomorphisms $[0, \frac{1}{2}] \xrightarrow{\sim} [0, \frac{\varphi_p(\s)}{1+ \varphi_p(\s)}] $ and $[\frac{1}{2},1] \xrightarrow{\sim} [\frac{\varphi_p(\s)}{1+ \varphi_p(\s)},1]$. It follows from this fiberwise decomposition that $\Phi_p$ does indeed define a bijection (which is clearly continuous). Since source and target are compact Hausdorff spaces, this already shows that $\Phi_p$ defines a stratum-preserving homeomorphism. Next, let us agglomerate some more observations about $\Phi_p$, the first of which verifies that $\Phi_p$ is a homeomorphism if $p \notin \J$.  
    \begin{PhiProps}
        \item\label{proof:prop_univ_stan_ob_1} Whenever $p \notin \J$, then $\tp = 1$ for all $x \in \sReal{\Delta^{\J}}$. Hence, then $\Phi_p$ is given by the identity. 
        \item\label{proof:prop_univ_stan_ob_2} For any $p \in  \pos$, we have $\Phi_p( \stanHood{\Delta^\J}) \subset \phiStanHood{\Delta^\J}$. 
        \item\label{proof:prop_univ_stan_ob_3} For $q \leq p$, $\Phi_p(\phiStanHood[q]{\Delta^\J})) \subset \Phi_p(\phiStanHood[q]{\Delta^\J}))$.
        \item\label{proof:prop_univ_stan_ob_4} For any $q \in P$, $\Phi_p(\stanHood[q]{\Delta^\J})) \subset \Phi_p(\stanHood[q]{\Delta^\J})$.
    \end{PhiProps}
    \cref{proof:prop_univ_stan_ob_1,proof:prop_univ_stan_ob_2} are immediate from the construction of $\Phi_p$.
    Let us verify \cref{proof:prop_univ_stan_ob_3,proof:prop_univ_stan_ob_4}.
    Note first that by \cref{proof:prop_univ_stan_ob_1}, we may assume that $p \in \J$. Furthermore, by \cref{prop:desc_neighb_notin} we may assume that $q \in \J$. 
    Therefore, since $\J$ is a flag, we may proceed with the remaining cases $q < p$, $q=p$, and $q>p$. Furthermore, since any of the relevant $q$-strata 
    neighborhoods contains $(\sReal{\Delta^{\J}})_{<q}$ and $\Phi_p$ is stratum-preserving, we may always assume $s_{\not < q} >0$. 
    For $q < p$, note that $\Phi_p(x)_{\leq q} = x_{\leq q}$, for all $x \in \sReal{\Delta^\J}$ (this follows from the computation of $x$ in \cref{con:final_coordinates}). From this it follows that \[
    s_{\not < q}(\Phi_p(x)) = 1- s_{<p}(\Phi_p(x)) =1- s_{<p}(x) = s_{\not <q}(x) \]
    and similarly
    \[
    t_{\not \leq q}(\Phi_p(x)) = t_{\not \leq q}(x)
    \]
    (whenever the latter is defined). In particular, this immediately implies \cref{proof:prop_univ_stan_ob_3,proof:prop_univ_stan_ob_4}. If $q = p$, then by \cref{proof:prop_univ_stan_ob_2} and since $\phiStanHood{\Delta^\J} \subset \stanHood{\Delta^\J}$, we have
    \[
    \Phi_p( \phiStanHood{\Delta^\J}) \subset \Phi_p( \stanHood{\Delta^\J}) \subset \phiStanHood{\Delta^\J} \subset \stanHood{\Delta^\J}.
    \]
    It remains to consider the case where $q>p$ for \cref{proof:prop_univ_stan_ob_4}. In this case, one can compute from the description of $x \in \sReal{\Delta^\J}$ in \cref{con:final_coordinates} the equalities
    \begin{align}
        s_{\not < q}( \Phi_p(x) ) &= \frac{\tp( \Phi_p(x))}{\tp (x)} s_{\not <q} (x) \\
        s_{\not \leq q}( \Phi_p(x) ) &= \frac{\tp( \Phi_p(x))}{\tp (x)} s_{\not \leq q} (x)
    \end{align}
     whenever these expressions are defined. An elementary verification shows that this is indeed the case whenever $t_{\not \leq q} ( \Phi_p(x)) >0$.
     It follows that in this case
     \[
     t_{\not \leq q}( \Phi_p(x) ) = \frac{ s_{\not \leq q}( \Phi_p(x) )}{s_{\not < q}( \Phi_p(x) )}=  \frac{s_{\not \leq q} (x)}{s_{\not < q} (x)} = t_{\not \leq q} (x).
     \]
     In particular, \cref{proof:prop_univ_stan_ob_4} holds for the remaining case $q >p$ .
     This finishes the verification of the properties of the natural transformation $\Phi_p$.
     Next, for a flag $\J$ degenerating from a regular flag $\standardFlag$, we set 
     \[
     \Phi = \Phi_{p_n } \circ \cdots \circ \Phi_{p_0} \colon \sReal{\Delta^\J} \to \sReal{\Delta^\J}.
     \]
     
     Note that this still defines a natural transformation of the realization functor (on the stratified simplex category). Indeed, whenever there is a stratum-preserving simplicial map 
     \[
     \Delta^\J \to \Delta^{\J'}
     \]
     and $\J'$ degenerates from $[q_0 < \cdots < q_m ]$, then $\standardFlag \subset [q_0 < \cdots < q_m ]$. Thus, it follows from \cref{proof:prop_univ_stan_ob_1} that on $\sReal{\Delta^\J}$
     \[
     \Phi_{p_n } \circ \cdots \circ \Phi_{p_0} = \Phi_{q_m } \circ \cdots \circ \Phi_{q_0}.
     \]
      Using this equation, the naturality of $\Phi$ follows from the naturality of the $\Phi_{q_n}$. By left Kan extension $\Phi$ extends to a natural automorphism of $\sReal{-}$. It remains to verify that $\Phi_{\str}$ lifts to a map $\stanHoodSys \to \phiStanHoodSys$. By construction of these neighborhood systems, it suffices to verify this on stratified simplices $\Delta^\J$, where $\J$ degenerates from the regular flag $\standardFlag$. As before, we may assume that $p \in \J$, i.e., $p= p_k$, for some $k \in [n]$. Then, we have
     \begin{align*}
         \Phi( \stanHood{\Delta^\J} ) &=  \Phi_{p_n} \circ \cdots \circ \Phi_{p_0}( \stanHood{\Delta^\J} ) \\
         &\subset \Phi_{p_n} \circ \cdots \circ \Phi_{p_k}( \stanHood{\Delta^\J} ) \\
         &\subset \Phi_{p_n} \circ \cdots \circ \Phi_{p_{k+1}}( \phiStanHood{\Delta^\J} ) \\
         & \subset \phiStanHood{\Delta^\J},
     \end{align*}
     where the first inclusion follows by \cref{proof:prop_univ_stan_ob_4}, the second inclusions follows by \cref{proof:prop_univ_stan_ob_2}, and the final inclusion follows by \cref{proof:prop_univ_stan_ob_3}. \\
     It remains to see that $\Phi$ is stratum-preserving homotopic to the identity.  
     Given any family of functions $\varphi$ as in the assumption, we write $\Phi(\varphi)$ for the corresponding $\Phi$, constructed as above. Note that $\Phi$ varies continuously in each $\varphi_p$ (with respect to the supremum distance on $C^0([0,1],[0,1])$ and that $\Phi(\varphi) =1$, if $\varphi$ is given by the constant functions of value $1$. For $t \in [0,1]$, define $\varphi^t$ via $\varphi_p^t(s) = (1-t) + t\varphi_p$, for $p \in P$. Then, $t \mapsto \Phi(\varphi^t)$ defines the required natural homotopy. 
\end{proof}
We may combine \cref{prop:universality_of_phi_stan_hood,prop:universality_of_standard} as the following result, which will be central to generalizing our construction of strata-neighborhood systems to stratified cell complexes. 
\begin{proposition}\label{prop:maps_from_finite_approx}
    Let $\str \in \sSetPN$ be a finite stratified simplicial set. Let $\tstr[U]_{{\str[Y]}}$ be a strata-neighborhood system on ${\str[Y]} \in \TopPN$ and $f\colon \sReal{\str} \to {\str[Y]}$ be any stratum-preserving map. Then, there exists a natural stratum-preserving automorphism $\Phi$ of $\sReal{-}$, naturally stratified homotopic to the identity through a family of automorphisms which lift to maps $\stanHoodSys \to \stanHoodSys$ such that $f \circ \Phi_{\str}: \sReal{\str} \to {\tstr[Y]}$ lifts to a map $\stanHoodSys \to \snVals[{\str[Y]}]$.
\end{proposition}
 A first consequence of \cref{prop:maps_from_finite_approx} is that strata-neighborhood systems may be used to compute homotopy links using only data close to a stratum. Such an argument was essentially the decisive step in the proof of \cite[Theorem 4.8]{douteauwaas2021}, where we gave an elementary proof of a special case of the following more general statement.
 \begin{proposition} \label{prop:computing_links_using_nbhds}
    Let $(\tstr,\snSys) \in \NCat$ and $\I\subset \pos$ be a regular flag. Then the inclusion $ \snVals(\I) \hookrightarrow \tstr$ induces a weak equivalence
    \[
    \HolIP(\snVals(\I)) \to \HolIP(\tstr).
    \]
\end{proposition}
\begin{proof}
    We use \cref{lem:criterion_for_weak_equivalence}. Under the adjunction
    \[- \times \sReal{\Delta^\I} \colon \TopN \rightleftharpoons \TopPN \colon \HolIP,\] 
    we may thus equivalently show that every stratum-preserving map $g_0:D^{n+1} \times \sReal{\Delta^\I} \to {\tstr}$ which maps $S^n \times \sReal{\Delta^\I}$ to $\snVals(\I)$ is stratum-preserving homotopic to a map $g_1:D^{n+1} \times \sReal{\Delta^\I} \to {\tstr}$ with image in $\snVals(\I)$ through a homotopy mapping $S^n \times \sReal{\Delta^\I}$ to $\snVals(\I)$. 
    Now, fix some identification $\real{\Delta^{n+1}} \cong D^{n+1}$. Under this identification, we obtain a canonical isomorphism
        \[
        \sReal{\Delta^{n+1} \times \Delta^\I} \cong \real{\Delta^{n+1}} \times \sReal{\Delta^\I} \cong D^{n+1} \times \sReal{\Delta^\I},
        \]
    identifying $\sReal{\partial \Delta^{n+1} \times \Delta^\I}$ with $S^n \times \sReal{\Delta^\I}$. 
    We can now apply \cref{prop:maps_from_finite_approx} to
    \[
    g_0 \colon \sReal{\Delta^{n+1} \times \Delta^\I} \cong D^{n+1} \times \sReal{\Delta^\I} \to \tstr,
    \]
    from which it follows that $g_0$ is stratum-preserving homotopic to a stratum-preserving map $g'_0 \colon \sReal{\Delta^{n+1} \times \Delta^\I} \cong D^{n+1} \times \sReal{\Delta^\I} \to \tstr$ that maps $\stanHood[\I]{\Delta^{n+1} \times \Delta^\I}$ into $\snVals(\I)$.
    Furthermore, by naturality of the homotopy in \cref{prop:maps_from_finite_approx}, it follows that the homotopy between $g_0$ and $g'_0$ maps $S^n \times \sReal{\Delta^\I}$ into $\snVals(\I)$. Next, note that the identification $\sReal{\Delta^{n+1} \times \Delta^\I} \cong D^{n+1} \times \sReal{\Delta^\I}$ restricts to an identification \[\stanHood[\I]{\Delta^{n+1} \times \Delta^\I}  \cong D^{n+1} \times \stanHood[\I]{\Delta^\I}. \] Hence, we may assume without loss of generality that $g_0$ maps $D^{n+1} \times \stanHood[\I]{\Delta^\I}$ into $\snVals(\I)$.
    Finally, note that, by \cref{ex:standardhood_of_regular_simplex},  the inclusion $ \stanHood[\I]{\Delta^\I} \hookrightarrow  \sReal{\Delta^\I}$ , equivalently given by  \[
    \sReal{\Delta^\I} \hookrightarrow  \sReal{\sd \Delta^\I} \cong \sReal{\Delta^\I}.
    \]
    Choose any strong (stratum-preserving) deformation retraction $R: \sReal{\Delta^\I} \times [0,1] \to \sReal{\Delta^\I}$ of the inclusion $ \sReal{\Delta^\I} \hookrightarrow \sReal{\sd \Delta^\I} \cong \sReal{\Delta^\I}$ (given, for example, by affine interpolation between the identity and the last vertex map). Then the homotopy \[ g_0 \circ (1_{D^{n+1}} \times R ) \colon D^{n+1} \times \sReal{\Delta^{\I}} \times [0,1] \to D^{n+1} \times \sReal{\Delta^{\I}} \] has the required properties.    
\end{proof}
\subsection{Strata-neighborhood systems for stratified cell complexes}\label{subsec:nbhd_for_cell}
Next, let us generalize the construction of standard neighborhood systems to stratified cell complexes. The obvious issue at hand is that we may generally not expect the choices of standard neighborhood on cells to be compatible with attaching maps. To amend this difficulty, we first need a notion of subdivision of a stratified cell complex. For the remainder of this section, by a stratified cell complex we will always mean a $P$-stratified space, together with a fixed choice of cell structure $(\sigma_i \colon \sReal{\Delta^{\J_i}} \to {\tstr})_{i \in I}$. By a slight abuse of notation, we will often just refer to the underlying space.
\begin{definition}\label{def:subdivision}
    Let ${\tstr}$ be a stratified cell complex, defined by cells $(\sigma_i \colon \sReal{\Delta^{\J_i}} \to \tstr)_{i \in I}$. By a barycentric subdivision of $\tstr$ we mean a family of stratum-preserving homeomorphisms $\Psi_i \colon \sReal{\sd \Delta^{\J_i}}  \xrightarrow{\sim} \sReal{\Delta^{\J_i}}$, for $i \in I $, which fulfill $ \Psi_i(\sReal{\sd \Delta^{\J}}) \subset \sReal{\Delta^{\J}}$, for $\J \subset \J_i$. 
\end{definition}
\begin{remark}
    Note that any choice of barycentric subdivision $(\Psi_i)_{i \in I}$ on a stratified cell complex $\tstr$ naturally induces a new cell structure on $\tstr$ which is indexed over \[
\{ (i, \tau) \mid i \in I, \tau \textnormal{ simplex of } \sd \Delta^{\J_i} \textnormal{ s.t. }\J_i \in \tau \}.
    \] 
    We write $\sd_{\Psi} \tstr$, for the underlying space of $\tstr$ equipped with this new cell structure.
\end{remark}
\begin{example}
    If $\tstr \in \TopPN$ is the realization of a stratified simplicial complex $\tstr[K]$, equipped with the induced cell structure, and we take $\Psi_{i}\colon \sReal{\sd \Delta^{\J_i}} \to \sReal{\Delta^{\J_i}}$ to be the barycentric subdivision homeomorphism, then the cell structure induced by the subdivision $\Psi$ is the one coming from the barycentric subdivision homeomorphism $\sReal{\sd \str[K]} \cong \sReal{\str[K]} = \str$.
\end{example}
\begin{definition}\label{def:defines_stratum_nbhd_system}
    Let $\tstr$ be a stratified cell complex, and $\Psi$ a choice of subdivision of $\tstr$. We say that $\Psi$ \define{defines a standard neighborhood system on $\tstr$} if for every $p \in P$ and every $i \in I$ the inclusion
    \[
    \sigma_i \circ \Psi_i(\sReal{\simStanHoods{\partial \Delta^{\J_i}}}) \subset \bigcup_{j \prec i} \sigma_j \circ \Psi_j(\sReal{\simStanHoods{ \Delta^{\J_j}}}),
    \]
    holds (here $\prec$ denotes the precursor order of \cref{def:cell_struct}).
\end{definition}
\begin{example}
Consider the pinched torus $S^1 \times S^1 /(S^1 \times {x_0})$, stratified over $ \{ {\color{red}0}< {\color{NewBlue}2} \}$ taking the equivalence class $S^1 \times {x_0}$ as the $0$-stratum. To the left, a stratified cell structure induced by a simplicial model is shown, with the stratification indicated by the coloring. To the right, we show a subdivision of this cell structure, which defines a standard neighborhood system on the pinched torus. The standard-neighborhood of the $p$-stratum induced by this subdivision is shown shaded in red.
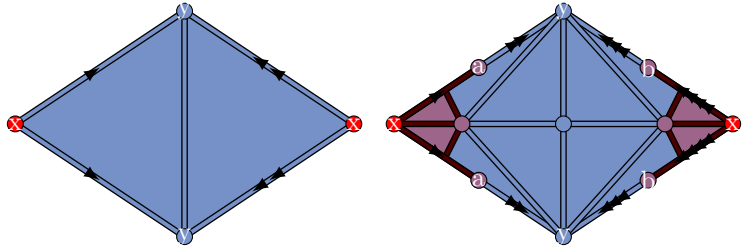
\begin{center}

\begin{tikzpicture}[witharrow/.style={postaction={decorate}}, scale = 1.5, 
dot/.style = {circle, fill, minimum size=#1, inner sep=0pt, outer sep=0pt},
dot/.default = 2pt  
]
    \newcommand{\doubLine}[3]{\draw[color = black, line width = 2.5pt] (#2) -- (#3);
                 \draw[color = #1, line width = 1.5pt] (#2) -- (#3)}
                \coordinate (l) at (0,0);
                \coordinate (r) at (3,0);
                \coordinate (u) at (1.5,1);
                \coordinate (d) at (1.5,-1);
                \coordinate (blu) at ($0.5*(l)+0.5*(u)$);
                \draw[fill = NewBlue]{}(l) -- (u) -- (r) -- (d) -- (l);
                \doubLine{NewBlue}{l}{u};
                  \doubLine{NewBlue}{u}{r};
                   \doubLine{NewBlue}{r}{d};
                \doubLine{NewBlue}{d}{l};
                \doubLine{NewBlue}{u}{d};
                 \draw[draw = none, decoration={markings,mark= at position 0.5 with {\arrow[scale = 1.5]{latex}}},witharrow] (l) -- (u);
                \draw[draw = none, decoration={markings,mark= at position 0.5 with {\arrow[scale = 1.5]{latex}}},witharrow] (l) -- (d);
                \draw[draw = none, decoration={markings,mark=between positions 0.5 and 0.6 step 0.1 with {\arrow[scale = 1.5]{latex}}},witharrow] (3,0) -- (1.5,-1);
                \draw[draw = none, decoration={markings,mark=between positions 0.5 and 0.6 step 0.1 with {\arrow[scale = 1.5]{latex}}},witharrow] (3,0) -- (1.5,1);
                \draw[color = black, fill = red, ](l) circle (2pt);
                \node (x) at (0,0){ \textcolor{white}{x}}   ; 
                \draw[color = black, fill = red, ](r) circle (2pt);
                \node (x) at (3,0){ \textcolor{white}{x}}   ;
                \draw[color = black, fill = NewBlue] (u) circle (2pt);
                \node (y) at (u){ \textcolor{white}{y}}   ; 
                \draw[color = black, fill = NewBlue] (d) circle (2pt);
                \node (y) at (d){ \textcolor{white}{y}}   ; 
            \end{tikzpicture}  
    \begin{tikzpicture}[witharrow/.style={postaction={decorate}}, scale = 1.5, 
dot/.style = {circle, fill, minimum size=#1, inner sep=0pt, outer sep=0pt},
dot/.default = 2pt  
]
    \newcommand{\doubLine}[3]{\draw[color = black, line width = 2.5pt] (#2) -- (#3);
                 \draw[color = #1, line width = 1.5pt] (#2) -- (#3)}
                \coordinate (l) at (0,0);
                \coordinate (r) at (3,0);
                \coordinate (u) at (1.5,1);
                \coordinate (d) at (1.5,-1);
                \coordinate (blu) at ($0.5*(l)+0.5*(u)$);
                \coordinate (altblu) at ($0.7*(l)+0.3*(u)$);
                \coordinate (bld) at ($0.5*(l)+0.5*(d)$);
                \coordinate (altbld)  at ($0.7*(l)+0.3*(d)$);
                \coordinate (brd) at ($0.5*(r)+0.5*(d)$);
                \coordinate (altbrd)  at ($0.7*(r)+0.3*(d)$);
                \coordinate (bru) at ($0.5*(r)+0.5*(u)$);
                \coordinate (altbru)  at ($0.7*(r)+0.3*(u)$);
                 \coordinate (bud) at ($0.5*(d)+0.5*(u)$);
                \coordinate (altbud)  at ($0.7*(d)+0.3*(u)$);
                \coordinate(blud) at ($0.6*(l)+0.4*(bud)$);
                \coordinate(brud) at ($0.6*(r)+0.4*(bud)$);
                \draw[fill = NewBlue]{}(l) -- (u) -- (r) -- (d) -- (l);
                 \draw[fill = red, opacity = 0.3] (l) -- (altblu) -- (blud) -- (altbld) --(l);
                 \draw[fill = red, opacity = 0.3] (r) -- (altbru) -- (brud) -- (altbrd) --(r);
                \doubLine{NewBlue}{l}{u};
                  \doubLine{NewBlue}{u}{r};
                   \doubLine{NewBlue}{r}{d};
                \doubLine{NewBlue}{d}{l};
                \doubLine{NewBlue}{u}{d};
            \doubLine{red, opacity = 0.3}{l}{blud};
            \doubLine{red, opacity = 0.3}{altblu}{blud};
            \doubLine{red, opacity = 0.3}{altbld}{blud};
            \doubLine{red, opacity = 0.3}{l}{blu};
            \doubLine{red, opacity = 0.3}{l}{bld};
            \doubLine{NewBlue}{bud}{blud};
            \doubLine{NewBlue}{u}{blud};
            \doubLine{NewBlue}{d}{blud};
            %
            \doubLine{red, opacity = 0.3}{r}{brud};
            \doubLine{red, opacity = 0.3}{altbru}{brud};
            \doubLine{red, opacity = 0.3}{altbrd}{brud};
            \doubLine{red, opacity = 0.3}{r}{bru};
            \doubLine{red, opacity = 0.3}{r}{brd};
            \doubLine{NewBlue}{bud}{brud};
            \doubLine{NewBlue}{u}{brud};
            \doubLine{NewBlue}{d}{brud};
               \draw[draw = none, decoration={markings,mark=at position 0.6 with {\arrow[scale = 1.5]{latex}}},witharrow] (l) -- (blu);
                 \draw[draw = none, decoration={markings,mark=between positions 0.5 and 0.6 step 0.1 with {\arrow[scale = 1.5]{latex}}},witharrow] (blu) -- (u);
                 \draw[draw = none, decoration={markings,mark=at position 0.6 with {\arrow[scale = 1.5]{latex}}},witharrow] (l) -- (bld);
                  \draw[draw = none, decoration={markings,mark= between positions 0.5 and 0.6 step 0.1 with {\arrow[scale = 1.5]{latex}}},witharrow] (bld) -- (d);
                  \draw[draw = none, decoration={markings,mark= between positions 0.4 and 0.6 step 0.1 with {\arrow[scale = 1.5]{latex}}}, witharrow] (bru) -- (u);
                 \draw[draw = none, decoration={markings,mark= between positions 0.4 and 0.6 step 0.1 with {\arrow[scale = 1.5]{latex}}}, witharrow] (brd) -- (d);
                 \draw[draw = none, decoration={markings,mark= between positions 0.4 and 0.8 step 0.1 with {\arrow[scale = 1.5]{latex}}}, witharrow] (r) -- (brd);
                  \draw[draw = none, decoration={markings,mark= between positions 0.4 and 0.8 step 0.1 with {\arrow[scale = 1.5]{latex}}}, witharrow] (r) -- (bru);
                 \draw[color = black, fill = red, ](l) circle (2pt);
                 \node (x) at (l){ \textcolor{white}{x}}   ; 
                 \draw[color = black, fill = red, ](r) circle (2pt);
                 \node (x) at (r){ \textcolor{white}{x}}   ;
                 \draw[color = black, fill = NewBlue] (u) circle (2pt);
                 \node (y) at (u){ \textcolor{white}{y}}   ; 
                 \draw[color = black, fill = NewBlue] (d) circle (2pt);
                 \node (y) at (d){ \textcolor{white}{y}}   ; 
                \draw[color = black, fill = NewBlue] (bud) circle (2pt);
                \draw[color = black, fill = NewBlue] (blud) circle (2pt);
                 \draw[color = black, fill = red, opacity = 0.3] (blud) circle (2pt);
                \draw[color = black, fill = NewBlue] (brud) circle (2pt);
                 \draw[color = black, fill = red, opacity = 0.3] (brud) circle (2pt);
                \draw[color = black, fill = NewBlue] (blu) circle (2pt);
                \draw[color = black, fill = red, opacity = 0.3] (blu) circle (2pt);
                  \node (a) at (blu){ \textcolor{white}{a}}   ; 
                 \draw[color = black, fill = NewBlue] (bru) circle (2pt);
                \draw[color = black, fill = red, opacity = 0.3] (bru) circle (2pt);
                 \node (b) at (bru) { \textcolor{white}{b}}   ; 
                 \draw[color = black, fill = NewBlue] (brd) circle (2pt);
                 \draw[color = black, fill = red, opacity = 0.3] (brd) circle (2pt);
                  \node (b) at (brd){ \textcolor{white}{b}}   ; 
                 \draw[color = black, fill = NewBlue] (bld) circle (2pt);
                \draw[color = black, fill = red, opacity = 0.3] (bld) circle (2pt);
                 \node (a) at (bld){ \textcolor{white}{a}}   ; 
            \end{tikzpicture}      
\end{center}
Note that the definition of a stratified cell complex allows for the attachment of $n$-cells to $n$-cells, as long as this does not lead to cycles in attachment. 
\end{example}
\begin{remark}\label{rem:alternative_v_prec}
   Note that the condition in \cref{def:defines_stratum_nbhd_system} can be equivalently defined by replacing the precursor order by any ordinal structure on $I$, which exposes $\tstr$ as an absolute cell complex with respect to stratified boundary inclusions of simplices (see \cref{con:equ_versions_of_strat_cell_cplx}). 
\end{remark}
\begin{construction}\label{con:standard_nbhd_system_of_cell_complex}
     The condition in \cref{def:defines_stratum_nbhd_system} precisely guarantees that when a subdivision $\Psi$ of a cell complex $\str$ defines a standard neighborhood system on $\str$, then the indexing set
    \[ \{ (i, \tau) \mid i \in I,\J_i \in \tau ; \tau \subset \simStanHoods{ \Delta^{\J_i}} \} \]
    corresponding to such cells in $\sd \str$ that lie in standard neighborhoods in the stratified cells define a subcomplex of $\sd_{\Psi} \str$.
    As a stratified space, it is given by the union \[
    \bigcup_{i \in I} \sigma_i \circ \Psi_i (\sReal{\simStanHoods{\Delta^{\J_i}}}) \subset \tstr.\] 
    We denote this subcomplex by $\PsiStanHood{\tstr}$. We call $\PsiStanHood{\tstr}$ the \define{$p$-th standard neighborhood associated to the subdivision $\Psi$}. Furthermore, we denote \[ \PsiStanHoodSys:= ( \PsiStanHood[p]{\tstr})_{p \in P} \] and call this family the \define{standard neighborhood system of $\tstr$ associated to the subdivision $\Psi$}. 
\end{construction}
Let us verify that the nomenclature of \cref{con:standard_nbhd_system_of_cell_complex} does make sense, that is, that we have indeed defined a strata-neighborhood system. Before we do so, note the following remark.
\begin{remark}\label{rem:Psi_stan_hood_for_cplx}
    For any stratified simplicial complex $\tstr \in \sSetPN$ we may use \cref{prop:simplicial_standard_model} to  identify $\sReal{\simStanHoods{\tstr}}$ with $\PsiStanHoods{\sReal{\tstr}}$, where $\Psi$ is the barycentric subdivision of \cref{con:simplicial_model_of_standard}.
\end{remark}
\begin{proposition}\label{prop:stan_hood_is_hood_cell}
    In the setting of \cref{con:standard_nbhd_system_of_cell_complex}, the family $\PsiStanHoodSys$ defines a strata-neighborhood system on $\tstr$. Furthermore, $\PsiStanHoodSys$ has the property that $\PsiStanHoods[\I]{{\tstr}} \subset {\tstr}$ is a subcomplex of $\sd_{\Psi} {\tstr}$, for every regular flag $\I \subset P$.
\end{proposition}
\begin{proof}
    That, for any $p \in P$, it holds that ${\utstr}_{\leq p} \subset \PsiStanHood[p]{{\tstr}}$, is immediate from \cref{rem:Psi_stan_hood_for_cplx}, \cref{prop:simplicial_standard_model} and \cref{cor:factorization_through_system}. Next, let us verify that $\PsiStanHood{\tstr}$ does indeed define a $p$-stratum neighborhood.
    By \cref{lem:char_of_nbhd_in_cell_complex}, it suffices to show that, for every $p\in P$ and every cell $\sigma \colon \sReal{\Delta^{\J_i}} \to {\tstr}$ of ${\tstr}$, the set $\sigma^{-1}(\PsiStanHood{{\tstr}})$ defines a neighborhood of the $p$-stratum. Since $\Psi_i$ is a stratum-preserving homeomorphism, we may equivalently show that $(  \sigma \circ \Psi_i)^{-1}(\PsiStanHood{{\tstr}})$ has this property. Note that by construction we have 
    \[\real{\simStanHoods{\Delta^{\J_i}}} \subset (\sigma \circ \Psi_i)^{-1}(\PsiStanHood{{\tstr}}).\]
    By \cref{prop:simplicial_standard_model}, up to a stratum-preserving homeomorphism of $\sReal{\Delta^\J_i}$ we have \[\sReal{\simStanHoods{\Delta^{\J_i}}} = \stanHoods{\Delta^{\J_i}},\] which shows that both $\real{\simStanHoods{\Delta^{\J_i}}}$ and thus also $(\Psi \circ \sigma)^{-1}(\PsiStanHood{{\tstr}})$ is a neighborhood of the $p$-stratum. The statement on subcomplexes is immediate from the fact that the intersection of subcomplexes is again a subcomplex.
\end{proof}
\begin{example}
    If $\tstr \in \TopPN$ is the realization of a stratified simplicial complex $\str[K]$,
    equipped with the induced cell structure, and we take $\Psi_{i}\colon \sReal{\sd \Delta^{\J_i}} \to \sReal{\Delta^{\J_i}}$ to be the subdivision homeomorphism of \cref{con:simplicial_model_of_standard}, then $\PsiStanHood{\str}= \stanHood[\I]{\str[K]}$.
\end{example}
Next, let us show that there always exists a subdivision which defines a standard neighborhood system on a stratified cell complex ${\tstr}$. This is ultimately a consequence of \cref{prop:universality_of_standard}, and provides a first step towards \cref{thm:hol_main_result}.
\begin{proposition}\label{prop:ex_nbhd_cell_cplx}
    For every stratified cell complex $\tstr$, there exists a subdivision $\Psi$ of ${\tstr}$ such that $\Psi$ defines a standard neighborhood system on ${\tstr}$. Additionally, for any $P$-stratified space ${{\tstr[Y]}}$ equipped with a strata-neighborhood system $\snSys[{\tstr[Y]}]$ and any stratum-preserving map $f \colon {\tstr} \to {{\tstr[Y]}}$, $\Psi$ may be chosen such that $f$ lifts to a map $\PsiStanHoodSys \to \snSys[{\tstr[Y]}]$.
    Furthermore, if such a subdivision $\Psi_{\tstr[A]}$ has already been chosen on a subcomplex of ${\tstr[A]} \subset {\tstr}$, then $\Psi$ may be taken such that 
    \[
    \Psi_{{\tstr[A]},i} = \Psi_i
    \]
    whenever $i \in I$ defines a cell of ${\tstr[A]}$. 
\end{proposition}
\begin{proof}
    Via transfinite induction, it suffices to show the following. For any commutative diagram
    \begin{diagram}
        \sReal{\partial \Delta^{\J_i}} \arrow[d] \arrow[r]& {\tstr[A]} \arrow[d] \\
        \sReal{\Delta^{\J_i}} \arrow[r] & {{\tstr[Y]}} \spacecomma
    \end{diagram}
    where ${\tstr[A]} \to {{\tstr[Y]}}$ lifts to a map of neighborhood systems $\snSys[{\tstr[A]}]\to\snSys[{\tstr[Y]}]$, there exists a stratum-preserving homeomorphism $\Psi_i \colon \sReal{\sd \Delta^{\J_i}} \to \sReal{\Delta^{\J_i}}$ (which is compatible with faces) with the following properties. 
    \begin{enumerate}
        \item The composition  $\sReal{ \sd \partial \Delta^{\J_i}} \xrightarrow{\Psi_i \mid_{\sReal{ \sd \partial \Delta^{\J_i}}}} \sReal{\partial \Delta^{\J_i}} \to {\tstr[A]}$ lifts to a map of neighborhood systems  \[ (\sReal{ \simStanHoods{\partial \Delta^{\J_i}} })_{p \in P} \to \snSys[{\tstr[A]}] .\]
        \item    The composition  $\sReal{ \sd  \Delta^{\J_i}} \xrightarrow{\Psi_i} \sReal{ \Delta^{\J_i}} \to {{\tstr[Y]}}$ lifts to a map of neighborhood systems  \[ (\sReal{ \simStanHoods{\Delta^{\J_i}} })_{p \in P} \to \snSys[{\tstr[Y]}] .\]
    \end{enumerate}
    We may first apply \cref{prop:simplicial_standard_model} and instead show the analogous statement with $\sd \Delta^{\J_i}$ replaced by $\Delta^{\J_i}$ and $(\sReal{ \simStanHoods{\Delta^{\J_i}} })_{p \in P}$ replaced by $\stanHoodSys[\Delta^{\J_i}]$.
    Next, apply \cref{prop:maps_from_finite_approx} twice, first to $\sReal{\Delta^{\J_i}} \to {{\tstr[Y]}}$, obtaining a natural stratum-preserving automorphism $\Phi_{{\tstr[Y]}}$ of $\sReal{-}$, and then to the composition $\sReal{\partial \Delta^{\J_i}} \xrightarrow{\Phi_{{\tstr[Y]}}} \sReal{\partial \Delta^{\J_i}} \to {\tstr[A]}$, obtaining a natural stratum-preserving automorphism $\Phi_{\tstr[A]}$ of $\sReal{-}$. By construction, these have the following properties:
    \begin{enumerate}
        \item The composition $\sReal{\partial \Delta^{\J_i}} \xrightarrow{\Phi_{{\tstr[Y]}} \circ \Phi_{\tstr[A]}} \sReal{\partial \Delta^{\J_i}} \to {\tstr[A]} $ lifts to a map $\stanHoodSys[\partial \Delta^{\J_i}] \to \snSys[{\tstr[A]}]$.
            \item $\sReal{\Delta^{\J_i}} \xrightarrow{\Phi_{\tstr[A]}} \sReal{ \Delta^{\J_i}}$ lifts to a map $\stanHoodSys[ \Delta^{\J_i}] \to \stanHoodSys[\Delta^{\J_i}]$.
        \item The composition $\sReal{ \Delta^{\J_i}} \xrightarrow{\Phi_{{\tstr[Y]}}} \sReal{ \Delta^{\J_i}} \to {{\tstr[Y]}} $ lifts to a map $\stanHoodSys[ \Delta^{\J_i}] \to \snSys[{\tstr[Y]}]$.
    \end{enumerate}
    In particular, by the composability of morphisms of strata-neighborhood systems, it follows that the composition $\Psi_i := \Phi_{{\tstr[Y]}} \circ \Phi_{\tstr[A]}$ also has the property that $ \sReal{ \Delta^{\J_i}} \xrightarrow{\Psi_i} \sReal{ \Delta^{\J_i}} \to {{\tstr[Y]}}$ also lifts to a map $\stanHoodSys[ {\Delta^{\J_i}}] \to \snSys[{\tstr[Y]}]$.
\end{proof}
Furthermore, we can now prove a first result towards \cref{thm:hol_main_res_B}.
\begin{corollary}
 \label{cor:models_for_homotopy_pushouts_half}
    Suppose, we are given a pushout square in $\TopPN$ \begin{diagram}\label{diag:lem_mod_for_ho_push_half}
        {\tstr[A]} \arrow[r,"c", hook] \arrow[d, "f"]& {\tstr[B]} \arrow[d] \\
    {\tstr[X]} \arrow[r, hook]  & {{\tstr[Y]}},
    \end{diagram}
    with ${\tstr[A]},{\tstr[X]}$ stratified cell complexes and $c$ an inclusion of a stratified subcomplex. Let $(\sigma_i)_{i \in I}$ be the cell structure on ${\tstr[X]}$, $(\sigma_{j})_{j \in J}$ be the cell structure on ${\tstr[A]}$ and $(\sigma_{j})_{j \in J \sqcup J'}$ be the cell structure on $\tstr[B]$, extending the one on ${\tstr[A]}$ along $c$. Then, there exist barycentric subdivisions $\Psi$ of ${\tstr[X]}$ and $\hat \Phi$ of ${\tstr[B]}$ such that the following holds:
    \\ Denote by $\Phi$ the restriction of $\hat \Phi$ to ${\tstr[A]}$.
    Denote by $\hat \Psi$ the subdivision of the induced cell structure on ${{\tstr[Y]}}$, given by $(\Psi_i)_{i \in I} \cup (\Phi_j)_{j \in J'}$. Then, \cref{diag:lem_mod_for_ho_push_half} lifts to a diagram of strata-neighborhood systems
    \begin{diagram}\label{diag:lift_of_pushout_diag}
        \mathfrak U^{\Phi}_{\tstr[A]} \arrow[r, "\tilde c"] \arrow[d] & \mathfrak U^{\hat \Phi}_{\tstr[B]} \arrow[d] \\
        \mathfrak U^{\Psi}_{\tstr[X]} \arrow[r] & \mathfrak U^{\hat \Psi} _Y \spaceperiod
    \end{diagram}
\end{corollary}
\begin{proof}
    By \cref{prop:ex_nbhd_cell_cplx}, we obtain subdivisions $\hat \Phi$ and $\Psi$ as in the claim such that ${\tstr[A]} \to {\tstr[X]}$ lift to maps of strata-neighborhood systems $\mathfrak U^{\Phi}_{\tstr[A]}  \to \mathfrak U^{\Psi}_{\tstr[X]}$. Now, define $\hat \Psi$ as above. Let us show that $\hat \Psi$ does define a strata-neighborhood system on the induced cell structure on ${{\tstr[Y]}}$. Via transfinite induction, we may without loss of generality assume that ${\tstr[B]}$ is given by gluing a single cell $\sigma_m \colon \sReal{\Delta^{\J}} \to {\tstr[B]}$ to ${\tstr[A]}$. Then, ${{\tstr[Y]}}$ is given by gluing $\sReal{\Delta^\J}$ to ${\tstr[X]}$ along $\sReal{\partial \Delta^\J} \to {\tstr[A]} \to {\tstr[X]}$. Denote the resulting cell of ${{\tstr[Y]}}$, by $\hat \sigma_m\colon \sReal{\Delta^{\J}} \to {{\tstr[Y]}}$. Then, by construction, we have 
    \begin{align*}
    \hat \sigma_m \circ \hat \Psi_m (\sReal{ \simStanHoods{\partial \Delta^J}})  &= f \circ \sigma_m \circ \Phi_m (\sReal{ \simStanHoods{\partial \Delta^J}}) \\
                                                                                &\subset f (U_{\tstr[A]}^{\Phi}(p)) \\
                                                                                &\subset U^\Psi_{\tstr[X]}(p) \\
                                                                                & = \bigcup_{i \in I} \sigma_j \circ \Psi_j(\sReal{\simStanHoods{ \Delta^{\J_j}}})
    \end{align*}        
    for all $p\in P$, which (by \cref{rem:alternative_v_prec}) was to be shown. It is then immediate by construction that ${\tstr[B]} \to {{\tstr[Y]}}$ also lifts to a map of strata-neighborhood systems. 
\end{proof}
Next, let us verify that the functor $\NDiagT \colon \NCat \to \DiagTop$ sends \cref{diag:lift_of_pushout_diag} to a homotopy pushout square.
\begin{lemma}\label{lem:comp_nbhd_preserves_cell_cplx}
    In the situation of \cref{cor:models_for_homotopy_pushouts_half}, the image of \cref{diag:lift_of_pushout_diag} under $\NDiagT$ has the following property. For each $\I \in \sd(\pos)$, the resulting square
        \begin{diagram}\label{diag:NDiagTModel}
          \NDiagT (  \mathfrak U^{\Phi}_{\tstr[A]})(\I) \arrow[r, hook] \arrow[d] &  \NDiagT (\mathfrak U^{\hat \Phi}_{\tstr[B]}(\I)) \arrow[d] \\
        \NDiagT (\mathfrak U^{\Psi}_{\tstr[X]}(\I)) \arrow[r, hook] &  \NDiagT (\mathfrak U^{\hat \Psi}_{{\tstr[Y]}}(\I))
        \end{diagram}
        is such that:
    \begin{enumerate}
        \item All objects of the square are cell complexes in $\TopN$;
        \item The square is a pushout in $\TopN$;
        \item The horizontals are relative cell complexes in $\TopN$.
    \end{enumerate}
    In particular, \cref{diag:NDiagTModel} is a homotopy pushout diagram in $\TopN$.
\end{lemma}
\begin{proof}
    It is immediate from the construction of the standard neighborhoods of a stratified cell complex that the diagrams
    \begin{diagram}
        \mathcal U^{\Phi}_{\tstr[A]}(\I) \arrow[r, hook] \arrow[d] &  \mathcal 
 U^{\hat \Phi}_{\tstr[B]}(\I) \arrow[d] \\
         \mathcal  U^{\Psi}_{\tstr[X]}(\I) \arrow[r, hook] & \mathcal  U^{\hat \Psi}_{{\tstr[Y]}}(\I)
    \end{diagram}
    are pushout diagrams of stratified cell complexes (see also \cref{prop:stan_hood_is_hood_cell}), with the upper vertical a relative (stratified) cell complex, for every $\I \in \sd(\pos)$. Indeed, note that the cells missing in $\mathcal U^{\Psi}_{\tstr[X]}(\I)$ from $\mathcal U^{\hat \Psi}_{{\tstr[Y]}}(\I)$ correspond precisely to the respective cells missing in $\mathcal U^{\Phi_0}_{\tstr[A]}(\I)$ from $\mathcal U^{\Phi}_{\tstr[B]}(\I)$. What remains to be shown is that these properties are preserved under applying the functor $(-)_{\geq p} \colon \TopPN \to \TopN$. This is an immediate consequence of \cref{lem:cell_cplxs_are_preserved} and \cref{lem:colim_of_upperset}.
    We can thus summarize that \cref{diag:NDiagTModel} is a pushout diagram of cell complexes where the upper horizontal is given by a relative cell complex, in particular a cofibration. It follows from \cite[A.2.4.4]{HigherTopos} that the diagram is homotopy cocartesian.
\end{proof}
\subsection{The proof \texorpdfstring{of \cref{prop:models_model_globally}}{that standard neighborhood models are homotopy link models.}}\label{subsec:models_are_models}
As a consequence of \cref{prop:computing_links_using_nbhds} we are now ready to give a proof of \cref{prop:models_model_globally}, which tells us that we may indeed use homotopy link models to compute homotopy links. 
Precisely, \cref{prop:computing_links_using_nbhds} guarantees us that for $\tstr[X]\in \TopPN$ the diagram $\HolIP[](\tstr[X])$ may equivalently computed via the diagram given by $\I \mapsto \HolIP(\snVals(\I))$, where $\snSys$ is any strata-neighborhood system of $\tstr[X]$.
\begin{notation}
    Let $({\tstr[X]}, \snSys) \in \NCat$. We denote by $\NDiagH(\snSys)$ the element of $\DiagTop$ given by
    \[
    \I \mapsto \HolIP(\snVals(\I))
    \]
    with the obvious structure maps induced by the ones on $\HolIP(\tstr[X])$. We denote 
    \[
    \NDiagH \colon \NCat \to \DiagTop
    \]
    the functor induced by this construction.
\end{notation}
We may then rephrase \cref{prop:computing_links_using_nbhds} as follows.
\begin{corollary}\label{cor:computing_link_using_nbhds}
    The inclusions $\snVals(\I) \hookrightarrow {\tstr[X]}$, for $({\tstr[X]},\snSys) \in \NCat$ and $\I \in \sd(\pos)$, induce a natural weak equivalence of functors
    \[
    \NDiagH \xrightarrow{\simeq} \HolIP[].
    \]
\end{corollary}
The obvious next step to prove \cref{prop:models_model_globally} is to show that $\NDiagH$ is in turn weakly equivalent to $\NDiagT$. The definition of a homotopy link model suggests to use the maximal vertex evaluation maps 
\[
\HolIP(\snVals(\I)) \xrightarrow{\ev_{p_n}} \snVal(\I)_{p_n} \hookrightarrow \snVal(\I)_{\geq p_n}
\]
However, there is a technical difficulty to overcome first. In fact, these maps do not induce a morphisms of diagrams. Already in the case where $\I = [p_0 < p_1]$ the diagram 
\begin{diagram}
    \HolIP[p_0 < p_1](\snVals(p_0 < p_1)) \arrow[r, "\ev_{p_1}"] \arrow[d] & \snVal(p_0 < p_1)_{\geq p_1} \arrow[d, hook]\\
    \HolIP[p_0](\snVals(p_0)) = (U_{\tstr[X]}(p_0))_{p_0} \arrow[r, hook] & \snVal(p_0)_{\geq p_0} 
\end{diagram}
is only commutative up to homotopy. What we may do instead is to construct a natural transformation $\NDiagH \to \NDiagT$ only up to homotopy coherence. We may then use rigidification results such as \cite[Prop. A.3.4.12]{HigherTopos} to obtain a weak equivalence of functors.
\begin{remark}
    There will occur a slight set-theoretical difficulty when using \cite[Prop. A.3.4.12]{HigherTopos}. Namely, we will want to consider the homotopy coherent nerve of $\Top$ as an element of $\sSet$. Size issues require us to pass to a larger Grothendieck universe. To make this rigorous, we need to assume large cardinals $\kappa < \kappa'$, and denote by $\sSet$ the category of simplicial sets of size smaller than $\kappa$ some fixed large cardinal, and by $\widetilde{\sSet}$ the category of simplicial sets of cardinality smaller than $\kappa'$. 
\end{remark}
\begin{definition}
       In the case where $\TopN$ denotes either $\Delta$-generated or topologically generated spaces (i.e., $\TopPN$ is cartesian closed). We denote by $X^Y$ the internal mapping space of $Y,X \in \TopN$. For any $X \in \Top$, this constructions defines a simplicial functor 
    \[
    X^{-} \colon \Top^{\op} \to \Top,
    \]
    by mapping an $n$-simplex
    \[\sigma \colon \real{\Delta^n}  \times Z \to Y\]
    of $\Top(Z,Y)$ to the adjoint map of
    \[
    X^Y \times \real{\Delta^n} \times Z \xrightarrow{1 \times \sigma } X^Y \times Y \xrightarrow{\textnormal{ev}} X, 
    \]
    which indeed defines an $n$-simplex of $\Top(X^Y, X^Z)$.
\end{definition}

\begin{construction}\label{con:coherent_evalutation}
    We only construct the weak equivalence for a fixed $({\tstr[X]},\snSys) \in \NDiagT$. Generalizing to the case of a whole natural transformation essentially just comes down to an increase in notation. Furthermore, we only prove the case where $\Top$ is cartesian closed. The general case follows from this using that every space is weakly equivalent to its kelleyfication with respect to $\Delta$-generated spaecs. We denote by $\mathcal{N}$ the homotopy coherent nerve functor from the category of ($\kappa'$) small simplicial categories $\widetilde{\sCat}$ to $\widetilde \sSetN$. We denote its leftadjoint by $\mathcal{S}$.  \\
    Denote by $\underline{\Pos}$ the category of all ($\kappa$ small) posets, with its simplicial structure inherited from $\sSet$. 
    Furthermore, consider the posets $Q=\sd(\pos) \times[1] ^{\op}$ as a category. Now, consider the assigment
    $
    E \colon Q \to \Pos
    $
    by mapping 
    \begin{align*}
         (\I, 1) &\mapsto [0]  \\
         (\I,0) &\mapsto \I
    \end{align*}
    and 
    \begin{align*}
        (\I, 1) \leq (\I',1) &\mapsto ( [0] \to [0]])  \\
           (\I, 1) \leq (\I',0) &\mapsto \{ 0 \mapsto \max \I ' \} \\
            (\I, 0) \leq (\I',0) &\mapsto ( \I \hookrightarrow \I').
    \end{align*}
    This assignment does not define a functor! However, we can turn it into a homotopy coherent functor. This is due to the fact that $E$ has the property \begin{align}{}\label{equ:E_almost_functor}
        E(f \circ g)(p) \geq E(f) \circ E(g)(p),
    \end{align} for composable $f,g \in Q$ and $p$ in the source of $E(g)$.
    For $\alpha_1, \alpha_0 \in Q$, denote by $Q_{\alpha_1, \alpha_0} \subset  Q$ the poset of all regular flags $S \subset Q$, with $\min S= \alpha_1$ and $\max S = \alpha_0$ ordered by reverse inclusion. Next, consider map
    \begin{align*}
        \mathcal{E} \colon Q_{\alpha_1, \alpha_0} \times E(\alpha_1) &\to E(\alpha_0) \\
    ([S_0 < \cdots < S_n],p) &\mapsto E(S_{n-1} \leq S_n) \circ \cdots \circ E(S_{0} \leq S_1) (p).
    \end{align*}
    It follows by \cref{equ:E_almost_functor} that $\mathcal{E}$ defines a map of posets.
    Thus, equivalently $\mathcal{E}$ specifies a simplicial map
    \[
    \nerve (Q_{\alpha_1, \alpha_0}) \to \sSet( \nerve (E(\alpha_1)), \nerve (E(\alpha_0))).
    \]
    In this manner, we have defined a simplicial functor
    \[
    \mathcal E \colon \mathcal{S}( \sd(\pos) \times [1]^{\op} ) \to \underline{\Pos},
    \]
    where $\mathcal{S}$ is the left adjoint to the homotopy coherent nerve (see for example \cite[Sec. 1.1.5]{HigherTopos}) and
    where the simplicial structure on the right hand side is inherited from the one on $\sSet$. 
    Next, consider the composition of simplicial functors
    \[
     \mathcal{S}( \sd(\pos)^{\op} \times [1]) \xrightarrow{\mathcal E} \underline{\Pos}^{\op} \xrightarrow{\nerve} \sSet^{\op} \xrightarrow{\real{-}} \Top^{\op} \xrightarrow{\utstr[X]^{-}} \Top.
    \]

    It specifies a homotopy coherent diagram $D$ in $\Top$, indexed over $\sd(\pos)^{\op} \times [1]$, which restricts to the constant diagram of value ${\utstr[X]}$ at $1$ and to the diagram $D_0$ given by $\I \mapsto \utstr[X]^{\real{\Delta^{\I}}}$ at $0$. 
    We may then consider $\NDiagH(\snSys)$ as a subdiagram of $D_0$ and $\NDiagT(\snSys)$ as a subdiagram of $D_1$. For $\alpha_0= (\I_0, 0)$ and $\alpha_1=(\I_1,1)$, and $p \in \I_1$, $\mathcal{E}_{\alpha_1, \alpha_0}$ has the property that $\mathcal{E}(-,p)$ has image in $\{q \in \I_0 \mid q \geq \max \I_0 \}$.
    It follows from this that restricting to $\NDiagH(\snSys)$ at $0$ and $\NDiagT(\snSys)$ at $1$ defines a homotopy coherent subdiagram of $D$.
    To summarize, we have constructed a simplicial functor
    \begin{align*}
          \ev \in \widetilde{{\sCat}}( \mathcal{S}( \sd(\pos)^{\op} \times [1]), \Top) &\cong \widetilde{\sSetN} ( \sd(\pos)^{\op} \times [1], \mathcal{N} (\Top))
    \end{align*}
    which restricts to $\NDiagH$ at $0$ and $\NDiagT$ at $1$, or in other words by the identity
    \begin{align*} 
    \widetilde{ \sSetN} ( \sd(\pos)^{\op} \times [1], \mathcal N(\Top))\cong \textnormal{Fun}( \sd(\pos)^{\op}, \mathcal N(\Top))_1. 
    \end{align*}
    a natural transformation of functors of quasi categories between
    \[
     \mathcal{N}(\sd(\pos)^{\op} \xrightarrow{\NDiagH(\snSys)} \Top) \]
     and
     \[
         \mathcal{N}(\sd(\pos)^{\op}  \xrightarrow{\NDiagT(\snSys)} \Top ) .
    \]
    For any fixed flag $\I = \standardFlag$ this natural transformation is given by 
    \[
    \HolIP(U_{\tstr[X]}(\I)) \xrightarrow{\ev _{p_n}} U_{\tstr[X]}(\I)_{\geq p_n}.
    \]
    Now, if $\snSys$ is a homotopy link model for ${\tstr[X]}$, then the latter map is a weak equivalence. Hence, if we pass to Kan-complex ($\sSet^o$) by applying singular simplicial sets, then this natural transformation is given pointwise by an isomorphism in the quasi-category $\mathcal{N}(\sSet^o)$. We have thus defined an isomorphism
    between the functors of quasi-categories
    \begin{align*}
        \mathcal{N}(\sd(\pos)^{\op}  \xrightarrow{\NDiagH(\snSys)} \Top \to \sSet^{o} ),  \\
         \mathcal{N}(\sd(\pos)^{\op}  \xrightarrow{\NDiagT(\snSys)} \Top \to \sSet^{o} ).
    \end{align*}
\end{construction}
We may now finish the proof of \cref{prop:models_model_globally}.
\begin{proof}[Proof of \cref{prop:models_model_globally}]
    We only provide a weak equivalence for some fixed homotopy link model $\snSys$. 
    The global case is essentially analogous.
    We consider $\widetilde \sCat$ as equipped with the model structure for simplicial categories (see \cite{BergnerSimCat}) making the adjunction $\mathcal{S} \dashv \mathcal{N}$ a Quillen equivalence between $\widetilde \sCat$ and simplicial sets equipped with the Joyal model structure, $ \sSetJoyL$ (\cite[Thm. 1.21]{QuasiCatAndsSCat}).
    If not indicated otherwise by an superscript $\mathfrak J$, we consider $\sSet$ to be equipped with the Kan-Quillen model structure.
    Let $\snSys$ be a homotopy link model for a stratified space ${\tstr[X]} \in \TopPN$. 
    We need to show that $\HolIP[]{\tstr[X]}$ and $\NDiagT(\snSys)$ are weakly equivalent. 
    By \cref{cor:computing_link_using_nbhds}, we may instead show that $\NDiagT(\snSys)$ and $\NDiagH(\snSys)$ are weakly equivalent. 
    Using the Quillen equivalence between $\Top$ and $\sSet$, we may equivalently show that $\Sing\circ \NDiagT(\snSys) \colon \sd(\pos) ^{\op} \to \sSet$ and $\Sing\circ \NDiagH(\snSys) \colon \sd(\pos)^{\op} \to \sSet$ are weakly equivalent. In other words, we need to show that these two functors present the same path component in $\pi_0(\DiagS)$ (using the notation of \cite[Prop. A.3.4.12]{HigherTopos}.) By \cite[Prop. A.3.4.12]{HigherTopos} there is a canonical bijection:
   \begin{align*}
        \pi_0(\DiagS) = \ho \widetilde{\sCat} ( \sd(\pos)^{\op}, \sSet^o ).
    \end{align*}
    Furthermore, under the Quillen equivalence between simplicial and quasi-categories \cite[Thm. 1.21]{QuasiCatAndsSCat}, this bijection extends to 
    \[
      \pi_0(\DiagS) = \ho \widetilde{\sCat} ( \sd(\pos)^{\op}, \sSet^o ) = \ho \sSetJoyL ( \sd(\pos)^{\op}, \mathcal N ( \sSet ^o) ).
    \]
     $\ho \sSetJoyL ( \sd(\pos)^{\op}, \mathcal N ( \sSet ^o) )$ is the set of isomorphism classes of functors of quasi-categories $\sd(\pos)^{\op} \to \mathcal N ( \sSet ^o)$. We have constructed such an isomorphism between $\mathcal{N}(\Sing\circ \NDiagT(\snSys))$ and $\mathcal{N}(\Sing\circ \NDiagH(\snSys))$ in
     \cref{con:coherent_evalutation}.
\end{proof}
\section{Regular neighborhoods and homotopy link models}\label{subsec:def_aspire}
In the previous section, we have constructed strata-neighborhood systems for stratified simplicial sets and stratified cell complexes. For a proof of \cref{thm:hol_main_result}, in light of \cref{prop:models_model_globally}, it remains to show that these strata-neighborhood systems are homotopy link models. To do so, we develop a generalized notion of regular neighborhoods for stratified spaces, which also applies to flags $\I \in \sd(\pos)$ of length greater equal to two. 
Recall from \cite[A]{friedman2003stratified} the notion of a nearly stratum-preserving deformation retraction (introduced in similar form in \cite{quinn1988homotopically}). 
The following generalizes this notion to the case of more than two strata. In the following sections, we will generally omit the index from the stratification maps $\sstr \colon \str \to \pos$ and just write $s(x)\in \pos$, for $x \in \str$.
\begin{definition}\label{def:aspire}
    Let ${\tstr[X]}\in \TopPN$ be a stratified space and let $\I = [p_0 < \cdots < p_n]$ be a regular flag in $\pos$.
    We say that ${\tstr[X]}$ admits \define{an almost\footnote{The usage of 'almost' instead of 'nearly' is purely for the sake of having a phonetically pleasant acronym.} stratum-preserving $\I$-retraction} - \aspire{}  for short - if the following holds: 
    There exists a stratum-preserving map $R \colon \utstr_{p_n} \times \sReal{\Delta^{\I}} \to \tstr$  such that, for each $p \in \I $, the map of (general topological) spaces
   \begin{align*}
    (\utstr_{p_n} \cup \utstr_{\leq p}) \times \sReal{\Delta^{\I_{\geq p}}}  \to \utstr \\
    (x,u) \mapsto \begin{cases}
        R(x,u) & s(x) = p_n \\
        x & s(x) \leq p
    \end{cases}
   \end{align*}
    is well defined and continuous.
\end{definition}
\begin{remark}\label{rem:aspire_two_strata}
    To get a first intuition for \cref{def:aspire}, let us decode what the requirements in \cref{def:aspire} mean in the case where $\I = [p_0 < p_1] = \pos$. Then, we may identify $\sReal{\Delta^\I}$ with the (stratified) interval $[0,1]$. Suppose a (stratified) neighborhood $\tstr[N] \subset \tstr$ of $\utstr_{p_0}$ admits an \aspire{} $R$.
    Then, equivalently $R$ is a stratum-preserving map
    \[
    R \colon N_{p_1} \times [0,1] \cong N_{p_1} \times \sReal{\Delta^{\I}} \to N
    \]
    which extends to
    \[
    N \times [0,1] \to N
    \]
    by taking the constant homotopy of the inclusion on $\utstr_{p_0} \hookrightarrow N$, and furthermore $R$ restricted to $N_{p_n}\times \{1\}$ is given by the inclusion $N_{p_n} \hookrightarrow N$.\\ We may summarize this information as $R$ defining a strong deformation retraction from $N$ to $\utstr_{p_0}$, which is stratum-preserving, except at time $0$ when all of $N$ is mapped into $\utstr_{p_0}$. Note that this is (up to a slight but inessential variation in target space) the definition of a \textit{nearly stratum-preserving strong deformation retraction} given in \cite[A]{friedman2003stratified} (adapted from \cite{quinn1988homotopically}). It is a consequence of \cite[Prop. A.1]{friedman2003stratified} that (under some additional conditions on $\tstr$) the existence of such a deformation retraction guarantees that the map
    \[
    \HolIP[p_0 < p_1](N) \xrightarrow{\ev_{p_1}} N_{p_1}
    \]
    is a homotopy equivalence. Note that this is half the condition required for $N$ to be part of a homotopy link model for $\tstr$. This already makes it plausible that \aspires{} may be used to verify that certain strata-neighborhood systems are homotopy link models.
\end{remark}
Another technical remark on questions of set theoretic topology is in order.
\begin{remark}
    In \cref{def:aspire}, we required the map $(\utstr_{p_n} \cup \utstr_{\leq p}) \times \sReal{\Delta^{\I_{\geq p}}}  \to X$ to be a continuous map of \textit{general} topological spaces. In particular, we take $(\utstr_{p_n} \cup \utstr_{\leq p}) \subset X$ to have the classical relative topology, not the $\Delta$-generated or compactly generated one. Indeed, since $(\utstr_{p_n} \cup \utstr_{\leq p})$ is not open in $X$, this is generally a stronger requirement. For example, this subtlety will be important in the proof of \cref{prop:properties_of_star_homotopy}.
\end{remark}
\begin{remark}
    We are often going to treat an \aspire{} $R \colon \utstr_{p_n} \times \sReal{\Delta^\I} \to \utstr$ as a (not-necessarily continuous) map 
    \[
     \tstr_{p_n} \times \sReal{\Delta^\I}  \cup \bigcup_{p \in \I} \utstr_{\leq p} \times \sReal{\Delta^{\I_{\geq p}}} \to X. 
    \]
     In this sense, we also write
     \[
     R(x,u):= x
     \]
     for $x \in \utstr_{p}$ and $u \in \sReal{\Delta^{\I_{\geq p}}}$. Furthermore, under the adjunction $- \times \sReal{\Delta^\I} \dashv \HolIP$ it can be useful to treat an \aspire{} $R$ as a map $\ustr_{p_n} \to \HolIP(\str)$, the value of which at $x\in \ustr$ we denote by $R_{x}$.
\end{remark}
Finally, let us give another characterization of \aspires{} in the case where $X$ is a metric space, which may be somewhat more intuitive.
\begin{remark}
    When $X$ is metrizable, we may equivalently require the stratum-preserving map $R$ as in \cref{def:aspire} to have the following property.
    Whenever a sequence $x_m \in \utstr_{p_n}$ converges to $x \in \utstr_{p}$, then the sequence of stratum-preserving simplices
    \[
    R_{x_m}|_{\sReal{\Delta^{\I_{\geq p}}}} \colon \sReal{\Delta^{\I_{\geq p}}} \to \tstr,
    \]
    converges uniformly to the constant map \[
    c_x \colon \real{\Delta^{\I_{\geq p}}} \to \ustr
    \]
    of value $x$.  In particular, if we denote by $v_n$ the maximal vertex of $\sReal{\Delta^\I}$, then $R(x,v_n) = x$. 
\end{remark}
\begin{remark}
    The question may arise, why we have chosen to use the more technical condition, to only require \aspires{} to extend continuously to certain subspaces of $\utstr_{\leq p_n} \times \sReal{\Delta^\I}$, and not to the whole space.
    For realizations of standard neighborhoods of stratified simplicial sets one can indeed produce \aspires{} which extends to the whole space (see \cref{prop:stan_aspire_is_aspire}).
    For stratified cell complexes, however, this is not the case (see \cref{ex:aspire_gen_not_ext}). This is ultimately due to the fact that stratified cell complexes allow for vastly pathological gluing maps, which are generally far from being piecewise linear. Nevertheless, the more general definition of \aspires{} we have chosen here also applies to stratified cell complexes.
\end{remark}
\cref{rem:aspire_two_strata} already suggests the following proposition.
\begin{proposition}\label{prop:computing_holinks_via_aspire}
    Let $\tstr \in \TopPN$ and let $\I =  \standardFlag \subset P$ be a regular flag. If $\tstr$ admits an \aspire, then 
        \[ \HolIP(\tstr) \xrightarrow{\ev} \utstr_{p_n}\]
    is a weak homotopy equivalence in $\TopN$.
\end{proposition}
To prove this proposition, we need the following construction: 
\begin{construction}\label{con:def_star_R}
     Let $\tstr \in \TopPN$ and let $\I =  \standardFlag \subset P$ be a regular flag. Let $R \colon \utstr[A]_{p_n} \times \sReal{\Delta^{\I}} \to \tstr[A]$ define an \aspire{} on a closed subspace $\tstr[A]$ of $\tstr$. Then, for any $p \in P$, it follows from $\tstr[A] \subset \tstr[X]$ being closed that the restriction of $R$ to $\utstr[A]_{p_n} \times \sReal{\Delta^{\I_{\geq p}}}$ extends continuously to a map 
     \[ (\utstr[A]_{p_n} \cup \utstr_{\leq p}) \times \sReal{\Delta^{\I_{\geq p}}} \to \utstr \] by mapping $(x,u)$ to $x$, whenever $s(x) \leq p$. 
     For notational simplicity, we consider $R$ as a (not necessarily continuous) map \[
     R \colon \bigcup_{p \in \I} (\utstr[A]_{p_n} \cup \utstr_{\leq p}) \times \sReal{\Delta^{\I_{\geq p}}} \to \utstr
     \]
     in this fashion. We may identify $\sReal{\Delta^{\I}} \times [0,1]  = \sReal{\Delta^{\I} \times \Delta^{1}}$. Having done so, we can consider the natural embedding 
     \begin{align*}
         \sReal{\Delta^{\I} \times \Delta^{1}} \hookrightarrow \real{\Delta^\I * \Delta^\I} \\
     \end{align*}
     under which $\sReal{\Delta^{\I} \times \Delta^{1}}$ corresponds to the union of joins $\real{\Delta^{\I_{\leq p}}*\Delta^{\I_{\geq p}}}$, $p\in \I$. This embedding induces join coordinates $(u,t) \estimates [y_{0},y_{1},t]$ on $\sReal{\Delta^{\I} \times \Delta^{1}}$. \\
     Now, denote by $v_n$ the maximal vertex of $\real{\Delta^{\I}}$. Let $\sigma \colon \sReal{\Delta^{\I}} \to \tstr$ be a stratum-preserving map, and $t \in [0,1]$ such that $\sigma(u) \in \utstr[A]\cup \utstr_{<p_n}$, if $u_{p_n} = t$.
    We define
     \begin{align*}
       \sigma \star_t R \colon  \sReal{\Delta^\I} &\to \tstr           \\
        u &\mapsto R( \sigma((1-t)y_0 + tv_n),y_1) \spaceperiod
     \end{align*}
     \begin{figure}
         \centering
         \begin{tikzpicture}[scale=3]
        \coordinate (A) at (0,0);
        \coordinate (B) at (2,0);
        \coordinate (C) at (1,{sqrt(3)});
        \coordinate (AC) at (barycentric cs:A=1,C=1);
        \coordinate (AB) at (barycentric cs:A=1,B=1);
        \coordinate (BC) at (barycentric cs:B=1,C=1);
        \coordinate (ABB) at (barycentric cs:AB=1,B=1);
        \coordinate (ACBC) at (barycentric cs:AC=1,BC=1);
        \draw (A) -- (B) -- (C) -- cycle;
        \filldraw[pattern=dots] (AC) -- (BC) -- (C) -- cycle;
        \filldraw[pattern=crosshatch] (A) -- (AB) -- (AC) -- cycle;
        \draw (AB) -- (AC);
        \draw (AC) -- (BC);      
        \draw[dashed] (ABB) -- (ACBC);
        \node[fill = white, rounded corners = 3mm,  inner sep = 2pt] at (barycentric cs:A=1,AB=1,AC=1){$R_{\sigma( \frac{1}{2}, 0, \frac{1}{2})}$};
        \node[fill = white, rounded corners = 3mm,  inner sep = 2pt] at (barycentric cs:AC=1,C=1,BC=1){$\sigma|_{u_2 \geq \frac{1}{2}}$};
        \node[left] at (AC){$\sigma(\frac{1}{2}, 0, \frac{1}{2})$}; 
        \node[right = 4pt, fill = white, rounded corners = 3mm, inner sep = 2pt] at (barycentric cs:ABB=1,ACBC=1){$R_{\sigma(\frac{1}{4}, \frac{1}{4}, \frac{1}{2})}|_{u_0 = 0}$};
        \node[below right = 4pt, fill = white, rounded corners = 3mm,  inner sep = 2pt] at (ACBC){$\sigma(\frac{1}{4},\frac{1}{4},\frac{1}{2})$};
        \node[circle, fill, inner sep=1.5pt] at (ACBC){};
        \node[circle, fill, inner sep=1.5pt] at (C){};
        \node[circle, fill, inner sep=1.5pt] at (AC){};
        \node[above] at (C) {$\sigma(0,0,1)$};
    \end{tikzpicture}
      \begin{tikzpicture}[scale=3]
        \coordinate (A) at (0,0);
        \coordinate (B) at (2,0);
        \coordinate (C) at (1,{sqrt(3)});
        \coordinate (AC) at (barycentric cs:A=1,C=3);
        \coordinate (AB) at (barycentric cs:A=1,B=3);
        \coordinate (BC) at (barycentric cs:B=1,C=3);
        \coordinate (ABB) at (barycentric cs:AB=1,B=1);
        \coordinate (ACBC) at (barycentric cs:AC=1,BC=1);
        \draw (A) -- (B) -- (C) -- cycle;
        \filldraw[pattern=dots] (AC) -- (BC) -- (C) -- cycle;
        \filldraw[pattern=crosshatch] (A) -- (AB) -- (AC) -- cycle;
        \draw (AB) -- (AC);
        \draw (AC) -- (BC);      
        \draw[dashed] (ABB) -- (ACBC);
        \node[fill = white, rounded corners = 3mm,  inner sep = 2pt] at (barycentric cs:A=1,AB=1,AC=1){$R_{\sigma( \frac{1}{4}, 0, \frac{3}{4})}$};
        \node[fill = white, rounded corners = 3mm,  inner sep = 2pt] at (barycentric cs:AC=1,C=1,BC=1){$\sigma|_{u_2 \geq \frac{3}{4}}$};
        \node[left] at (AC){$\sigma(\frac{1}{4}, 0, \frac{3}{4})$}; 
        \node[right = 4pt, fill = white, rounded corners = 3mm, inner sep = 2pt] at (barycentric cs:ABB=1,ACBC=1){$R_{\sigma(\frac{1}{8}, \frac{1}{8}, \frac{3}{4})}|_{u_0 = 0}$};
        \node[below right, fill = white, rounded corners = 3mm,  inner sep = 2pt] at (ACBC){$\sigma(\frac{1}{8},\frac{1}{8},\frac{3}{4})$};
        \node[above] at (C) {$\sigma(0,0,1)$};
        \node[circle, fill, inner sep=1.5pt] at (ACBC){};
        \node[circle, fill, inner sep=1.5pt] at (C){};
        \node[circle, fill, inner sep=1.5pt] at (AC){};
        \node[above] at (C) {$\sigma(0,0,1)$};
    \end{tikzpicture}
         \caption{Illustration in barycentric coordinates of the maps $\sigma \star_t R$ for $t=\frac{1}{2}$, $t=\frac{3}{4}$ and $\I= [0<1<2]$. As $t$ increases from $0$ to $1$, the stratified simplex $\sigma$ is gradually replaced by simplices of the form $R_{\sigma(u)}$, ending in $R_{\sigma(v_n)}$, for $t=1$.}
         \label{fig:illustration_R_star_sigma}
     \end{figure}
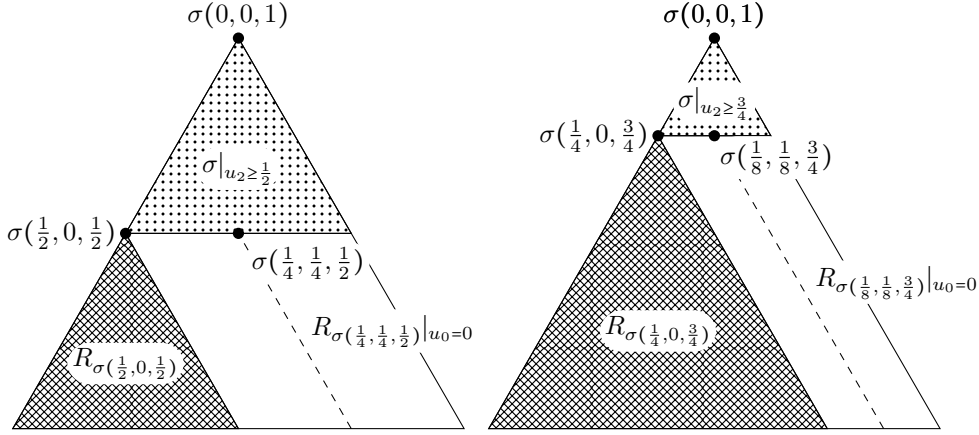
     See also the illustration of  $\sigma \star_t R$ in \cref{fig:illustration_R_star_sigma}.
     If $\sigma(u) \in \utstr[A]$, for all $u \in \sReal{\Delta^\I}$, then this construction extends to a homotopy 
     \begin{align*}
         \sigma \star R \colon \sReal{\Delta^{\I}} \times [0,1] &\to \tstr \\
         (u,t) &\mapsto \sigma \star_{t} R (u).
     \end{align*}
\end{construction}
\begin{proposition}\label{prop:properties_of_star_homotopy}
    Using the notation of \cref{con:def_star_R}, $\sigma \star R$ is well defined and has the following properties:
    \begin{enumerate}
        \item $\sigma \star_t R$ is stratum-preserving.
        \item $(\sigma \star R)_0 = \sigma$ and $(\sigma \star R)_1 = R(\sigma(v_n),-).$  
        \item Consider $\utstr_{\leq p_0}$ as a subspace of the space of continuous maps from $\real{\Delta^{\I}}$ to $\utstr$, $C^0( \real{\Delta^{\I}},\utstr)$, equipped with the compact open topology, by mapping $x$ to the constant map $c_x$ with value $x$. Furthermore, let $S$ denote the union of
         \[
        \{ (\sigma,t) \mid \sigma \in \HolIP(\tstr) \land  \forall u\in \sReal{\Delta^{\I}}: u_{p_n} = t \implies \sigma(u) \in \utstr[A] \}
         \]
         with  $\utstr_{\leq p_0} \times [0,1]$ in $C^0( \real{\Delta^{\I}},\utstr) \times [0,1]$.
         We equip $S$ with the compactly generated topology, that is, the Kelleyfication of the subspace topology in $C^0( \real{\Delta^{\I}},\utstr) \times [0,1]$ with respect to compact Hausdorff spaces. \\
        Then, the map 
        \begin{align*}
            - \star R \colon S  &\to \HolIP(\tstr) \cup \utstr_{\leq p_0} \\
            (\sigma,t) &\mapsto \sigma \star_t R. \\
            (x,t) &\mapsto x
        \end{align*}
        is continuous.
    \end{enumerate}
    In particular, if $\tstr = \tstr[A]$, then we obtain a homotopy
    \begin{align*}
        - \star R \colon \HolIP(\tstr) \times [0,1] &\to \HolIP(\tstr) \\
        (\sigma,t) &\mapsto \sigma \star_t R. 
    \end{align*}
    (with respect to the Kelleyfication topology)
    between the identity and $\sigma \mapsto R_{\sigma(v_n)}$.
\end{proposition}
\begin{proof}
Let us first verify that $\sigma  \star_t R$ is indeed well defined on each join $\real{\Delta^{\I_{\leq p}}*\Delta^{\I_{\geq p}}}$. For $p = p_n$, and $u\in \real{\Delta^{\I_{\leq p}}*\Delta^{\I_{\geq p}}}$ the coordinate $y_1(u)$ is given by $v_n$. It follows that $R(-,y_1)$ is given by the identity on $\tstr$ and there is nothing to show. \\
For any $t > 0$ and $p < p_n$ the point $(1-t)y_0 + tv_n$ satisfies $((1-t)y_0 + tv_n))_{p_n} =t$. As $\sigma$ is stratum-preserving, this also implies $\sigma((1-t)y_0 + tv_n) \in \utstr[A]_{p_n}$, making $R( \sigma((1-t)y_0 + tv_n),y_1)$ a well-defined expression, as long as we show independence from a choice of representatives in join coordinates. 
If $t=0$, then $(1-t)y_0 + tv_n = y_0 = u$, and hence $\sigma((1-t)y_0 + tv_n)) \in \utstr_{p_k}$, for some $p_k \leq p$. 
As $y_1 \in \sReal{\Delta^{\I_{\geq p}}} \subset \sReal{\Delta^{\I_{\geq {p_k}}}}$, it follows
that then the expression $R( \sigma((1-t)y_0 + tv_n),y_1)$ is independent of $y_1$, hence well defined in join coordinates. Precisely, we have 
\[ 
R( \sigma((1-t)y_0 + tv_n),y_1) = \sigma(u).
\]
Conversely, if $t=1$, $R( \sigma((1-t)y_0 + tv_n),y_1)$ is clearly independent of $y_0$ and given by 
\[
R( \sigma((1-t)y_0 + tv_n),y_1) = R(\sigma(v_n),y_1).
\]
Next, note that $\sigma \star_t R$ is stratum-preserving. We only need to check the case $t >0$. Then, $(1-t)x + tv_n\in (\sReal{\Delta^{\I}})_{p_n}$. Hence, as $\sigma$ was assumed to be stratum-preserving, it also follows that $\sigma((1-t)x + tv_n) \in \utstr_{p_n}$. Now, the stratum of $[y_0,y_1,t]$ (in join coordinates) is given by $s(y_1)$, whenever $t >0$. Hence, it follows from the assumption that $R$ is stratum-preserving that we indeed have 
\[
s(R( \sigma((1-t)y_0 + tv_n),y_1)) = s(y_1) = s(y_0,y_1,t),
\]
as was to be shown. 
It remains to verify the continuity of 
\begin{align*}
            - \star R \colon S  &\to \HolIP(\tstr) \cup \utstr_{\leq p_0} \\
            (\sigma,t) &\mapsto \sigma \star_t R\\
            (x,t) &\mapsto x .
        \end{align*}
 Using mapping space adjunctions, it suffices to verify the following statement: 
 Let $D$ be a compact Hausdorff space and let $f \colon D \times \sReal{\Delta^\I} \to \tstr$ and $\tau: D \to [0,1]$ be a pair of maps such that for all $a \in D$ either $f(a,-)$ is stratum-preserving and $f(a,u) \in \utstr[A]$ whenever $u_{p_n}= \tau(a)$, or $f$ is constant with value in $\utstr_{\leq p_0}$. Then the map 
    \begin{align*}
         f' \colon D \times \sReal{\Delta^\I} &\to \tstr\\ 
        (a,u) &\mapsto R\Big ( f\big (a,(1- \tau(a)) y_0(u, \tau(a)) + \tau(a)v_n, \tau(a) \big),y_1(u, \tau(a) \big) \Big)
    \end{align*}
    is continuous. \\
    For $p\in \I$, denote by $T^{p}$ the pushout of general topological spaces
    \[
    (\utstr_{\leq p} \cup \utstr[A]_{p_n}) \times \sReal{\Delta^{\I_{\geq p}}} \cup_{\utstr_{\leq p}  \times \sReal{\Delta^{\I_{\geq p}}}} \utstr_{\leq p}.
    \] Note that $R$ induces a continuous maps
    \[
    R^p \colon T^p \to \utstr \spaceperiod
    \]
    We obtain a closed covering of $D \times \sReal{\Delta^\I}$ by the sets $D^p$, for $p \in \I$, where
    \[
    D^p = \{ (a,u) \in D \times \sReal{\Delta^\I} \mid (u,\tau(a)) \in \real{\Delta^{\I_{\leq p}}*\Delta^{\I_{\geq p}}}\},
    \]
    and verify continuity of $f'$ separately on these pieces.
    Now, on each $D^p$, $f'$ is given by a composition
    \begin{align*}
        D^p \to  D^p \times_{\tau, t} \real{\Delta^{\I_{\leq p}}*\Delta^{\I_{\geq p}}} \to T^p \xrightarrow{R^p} \tstr
        \end{align*}
        with the respective maps defined by
        \begin{align*}
            (a,u) &\mapsto (( a,u), [y_0(a,u), y_1(a,u), \tau(a)]) \\
            ((a,u), [y_0, y_1, t]) & \mapsto [f(a,(1-t)y_0 + tv_n   ),y_1] \\
            [z,y] &\mapsto R^p[z,y].
        \end{align*}
        To verify the continuity of the first of these maps, one needs to treat the set $D^p \times_{\tau, t} \real{\Delta^{\I_{\leq p}}*\Delta^{\I_{\geq p}}}$ as a pullback, while for the second, one needs to use the topology given by taking the pushout of
        \begin{diagram}
        D^p \times_{\tau, \pi_{\{0,1\}}} ( \real{\Delta^{\I_{\leq p}}} \times \real{\Delta^{\I_{\geq p}}} \times \{ 0, 1\} )  \arrow[r] \arrow[d] &  D^p \times_{\tau, \pi_{[0,1]}} ( \real{\Delta^{\I_{\leq p}}} \times \real{\Delta^{\I_{\geq p}}} \times [0, 1] )  \\
            D^p \times_{\tau, \pi_{\{0\}}} \real{\Delta^{\I_{\leq p}}} \times \{0 \} \sqcup D^p \times_{\tau, \pi_{\{1\}}} \real{\Delta^{\I_{\geq p}}} \times \{1 \} &  \spaceperiod
        \end{diagram}
        The latter is, a priori, finer than the former. Since $D^p \times_{\tau, t} \real{\Delta^{\I_{\leq p}}*\Delta^{\I_{\geq p}}}$ is Hausdorff, with respect to the former topology, and compact, with respect to the latter, the two topologies do in fact agree. Summarizing, we have shown continuity of $f'$ on each $D^p$, and hence continuity of $f'$. \\
\end{proof}
We can now prove \cref{prop:computing_holinks_via_aspire}.
\begin{proof}[Proof of \cref{prop:computing_holinks_via_aspire}]
    We are going to show that $\ev$ is a homotopy equivalence, if we pass to the $\Delta$-generated topology. Note that since the $\Delta$-generated topology on $S$, as in \cref{prop:properties_of_star_homotopy}, is finer than the Kelleyfication with respect to compact Hausdorff spaces, $- \star R$ is also continuous with respect to the $\Delta$-generated topology.
    Since any space is naturally weakly equivalent to its $\Delta$-ification, this shows the result also for the case of compactly generated and general topological spaces.
    Let $R \colon \utstr_{p_n} \times \sReal{\Delta^{\I}} \to \str$ define an \aspire{} on $\tstr$.
    Consider the map 
    \begin{align*}
            \iota: \utstr_{p_n} &\to \HolIP{\tstr} \\
            x &\mapsto \{u \mapsto R(x,u)\}.
    \end{align*}
    Since $R$ is stratum-preserving, this map is indeed well defined.
    Furthermore, since $R(x,v_n) =x$, we have 
    \[
    \ev \circ \iota = 1.
    \]
    By \cref{prop:properties_of_star_homotopy}, the map
    \begin{align*}
        - \star R \colon \HolIP(\tstr) \times [0,1] &\to \HolIP(\tstr) \\
        (\sigma,t) &\mapsto \sigma \star_t R. 
    \end{align*}
    defines a homotopy between the identity and $\iota \circ \ev$.
\end{proof}
\subsection{\texorpdfstring{\aspires{}}{Aspires} of standard neighborhoods}\label{subsec:aspire_of_simplicial}
    Now, let us construct \aspires{} for the standard neighborhoods of stratified simplicial sets of \cref{con:standard_neighborhoods}. To accomplish this, let us first describe a class of retracts of the inclusions $\utstr_{\leq p} \hookrightarrow \stanHood[p]{\str}$.
\begin{construction}\label{con:retract_to_strata}
   Let $p \in P$ and $\J \subset P$ be a flag.
   We use coordinates $y_{\leq p}$, $y_{ \not \leq  p}$ and $s_{\leq p}$ (as in \cref{con:final_coordinates}) on $\sReal{\Delta^\J}$. 
   Consider the map 
   \begin{align*}
       \rho^p \colon \stanHood{\Delta^\J} &\to (\sReal{\Delta^{\J}})_{\leq p} \\
       [y_{\leq p}, y_{ \not \leq p}, s_{\leq p}] &\mapsto [y_{\leq p},y_{ \not \leq p} ,1] = [y_{\leq p}].
   \end{align*}
   Note that since $s_{\leq p} \geq \frac{1}{2}$, for $x\in \stanHood{\Delta^{\J}}$, this map is indeed well defined. Under left Kan extension, $\rho^p$ extends to a natural transformation
   \[
   \rho^p \colon \stanHood{\str} \to (\sReal{\str})_{\leq p} 
   \]
   which defines a retract to the natural inclusion
   \[
   (\sReal{\str})_{\leq p}   \hookrightarrow \stanHood{\str}
   \]
   of functors $\sSetPN \to \TopN$. 
   In fact, $\rho^p$ extends to a strong deformation through the natural homotopy defined simplexwise by
   \[
   ([y_{\leq p}, y_{ \not \leq p}, s_{\leq p}], t) \mapsto [y_{\leq p}, y_{ \not \leq p}, (1-t)s_{\leq p} + t].
   \]
   If we consider $\stanHood{\str}$ as stratified over $P_{\leq p}$ via \[
    x \mapsto \begin{cases}
        s(x) & s(x) \leq p \\
        p & s(x) \not \leq p 
    \end{cases}
    \]
    then this construction, in fact, defines a natural stratum-preserving strong deformation retraction of functors
    $\sSetPN \to \TopPN[P_{\leq p}]$.
\end{construction}
Next, we verify that the retractions $\rho^p$ are compatible with intersections of $p$-standard neighborhoods. 
\begin{lemma}\label{lem:retraction_preserve_nbhds}
     Let $\str \in \sSetPN$. Then, for any $q \leq p$, the inclusion \[
     \rho^p ( \stanHood[q]{\str} \cap \stanHood[p]{\str}) \subset \stanHood[q]{\str} 
     \]
    holds. 
\end{lemma}
\begin{proof}
  Similarly to the proof of \cref{prop:universality_of_phi_stan_hood} one may easily verify that 
    \begin{align}\label{equ:proof_lem_retraction_preserve_nbhds_1}
        s_q(\rho^p(x)) &= \frac{s_q(x)}{s_{\leq p}(x)}, \\
        \label{equ:proof_lem_retraction_preserve_nbhds_2} s_{\not \leq q}(\rho^p(x)) = 1-s_{\leq q}(\rho^p(x)) &= 1- \frac{s_{\leq q}(x)}{s_{\leq p}(x)} .         
    \end{align}
    Let $x \in \stanHood[q]{\str} \cap \stanHood[p]{\str}$.
    Then by \cref{equ:proof_lem_retraction_preserve_nbhds_1,equ:proof_lem_retraction_preserve_nbhds_2} 
    \begin{align*}
        s_{\not \leq q}(\rho^p(x)) & = 1- \frac{s_{\leq q}(x)}{s_{\leq p}(x)}  \\
                                    & = \frac{1}{s_{\leq p}(x)} ( s_{\leq p (x)} - s_{\leq q}(x)) \\
                                    & \leq \frac{1}{s_{\leq p}(x)} (1-s_{\leq q}(x))\\
                                    & \leq \frac{s_q(x)}{s_{\leq p}(x)}\\
                                    & = s_q(\rho^p(x)),
    \end{align*}
    that is, $\rho^p(x) \in \stanHood[q]{\str}$, as was to be shown.
\end{proof}
Using the simplex-wise convexity of the standard neighborhoods, we immediately obtain.
\begin{corollary}\label{cor:nbhd_retraction}
    For any regular flag $\I = \standardFlag \subset P$ and any $\str \in \sSetPN$ the natural transformation 
    $\rho^{p_n} \colon \stanHood[{p_n}]{\str} \to (\sReal{\str})_{\leq p_n}$ 
    restricts to a natural transformation
    \[
    \rho^\I \colon \stanHood[\I]{\str} \to \stanHood[\I]{\str}_{\leq p_n}.
    \]
    Even more, $\rho^\I$ is part of a natural strong deformation retraction (over $P_{\leq p_n}$) of the inclusion
    \[
    \stanHood[\I]{\str}_{\leq p_n} \hookrightarrow \stanHood[\I]{\str}.
    \]
\end{corollary}
As a first consequence of \cref{cor:nbhd_retraction} we obtain that the standard neighborhood systems $\stanHoodSys$, for $\str \in \sSetPN$, fulfill the second requirement of being a homotopy link model:
\begin{corollary}\label{cor:stanhood_are_model_half}
    For any $\str \in \sSetPN$ and $\I = \standardFlag \in \sd(\pos)$ the inclusion
    \[
    \stanHood[\I]{\str}_{p_n} \hookrightarrow \stanHood[\I]{\str}_{ \geq {p_n}}
    \]
    is a homotopy equivalence of topological spaces.
\end{corollary}
\begin{proof}
    By \cref{cor:nbhd_retraction}, the inclusion $\stanHood[\I]{\str}_{\leq p_n} \hookrightarrow \stanHood[\I]{\str}$ is a stratum-preserving homotopy equivalence over $P_{\leq p_n}$. Consequently, the restriction of this inclusion to the $p_n$-stratum is a homotopy equivalence, as was to be shown.
\end{proof}
Next, we use the retractions $\rho^p$ to define \aspires{} for standard neighborhoods.
\begin{construction}\label{con:standard_aspire}
    Let $\I = \standardFlag$ be a regular flag in $P$ and let $\J$ be some other flag. It follows from \cref{lem:retraction_preserve_nbhds} and the convexity of standard neighborhoods that the map
    \begin{align*}
        \stanHood[\I]{\Delta^\J} \times \sReal{\Delta^\I} &\to \stanHood[\I]{\Delta^\J} \\
        (x,t) &\mapsto \sum_{p \in \I} t_p \rho^p(x)
    \end{align*}
    is well defined. One may easily verify that this construction is natural in $\J$, and thus induces a natural transformation
    \[
    R_{\I} \colon \stanHood[\I]{-} \times \sReal{\Delta^\I} \to \stanHood[\I]{-},
    \]
   of functors $\sSetPN \to \TopN$.
\end{construction}
\begin{proposition}\label{prop:stan_aspire_is_aspire}
    For any $\str \in \sSetPN$, the natural transformation $R_{\I} \colon \stanHood[\I]{\str} \times \sReal{\Delta^\I} \to \stanHood[\I]{\str}$ restricts to an \aspire{} on $\stanHoods[\I]{\str}$.
\end{proposition}\label{prop:aspire_for_stanhood}
\begin{proof}
    Denote $U:= \stanHood[\I]{\str}$.
    Note that by construction $R_{\I}$ may even be defined continuously on all of $\stanHood[\I]{\str} \times \sReal{\Delta^\I}$.
    Let us first verify that the restriction $R_{\I}|_{U_{p_n} \times \sReal{\Delta^\I}}$ is stratum-preserving. First, note that for $x \in U_{p_n}$, and $p < p_n$
    we have 
    \[
    s_p \geq s_{\not \leq p} \geq s_{p_n} > 0.
    \]
    It follows that $\rho^p(x) \in U_{p}$, for all $p \in \I$. It follows from this that 
    $\sum_{p \in \I} t_p \rho^p(x)$
    lies in the stratum corresponding to the maximal $p$ with $t_p > 0$, as was to be shown.
    Furthermore, whenever $x \in U_{\leq p}$ and $t \in  \sReal{\Delta^{\I_{\geq p}}}$, then 
    \[
    R_\I(x,t) = \sum_{q  \in \I} t_q \rho^q(x) = \sum_{q \in \I_{\geq p}} t_q \rho^q(x) = \sum_{q \in \I_{\geq p}} t_q x = x
    \]
    as required.
\end{proof}
We may now summarize \cref{prop:computing_holinks_via_aspire,prop:aspire_for_stanhood,cor:stanhood_are_model_half} as:
\begin{corollary}
 For any $\str \in \sSetPN$, the standard neighborhood system $\stanHoodSys$ is a homotopy link model for $\sReal{\str}$.   
\end{corollary}
 \subsection{\texorpdfstring{\aspires{}}{Aspires} for stratified cell complexes}\label{subsec:aspire_of_cell}
 The problem with extending the construction of an \aspire{} as in \cref{con:standard_aspire} to stratified cell complexes is of course that \aspires{} may generally not be compatible with gluing. This is circumvented by the following construction.
 \begin{construction}\label{con:gluing_aspires} Let $\I = \standardFlag $ be  a regular flag in $P$ and
    suppose we are given a pushout diagram of finite stratified cell complexes in $\TopPN$
    \begin{diagram}
        {\tstr[A]} \arrow[d, "f"] \arrow[r, hook] & \tstr[B] \arrow[d , "g"] \\
        \tstr[X] \arrow[r, hook, "i"]& \tstr[X] \cup_{\tstr[A]} \tstr[B] ={{\tstr[Y]}} \spaceperiod
    \end{diagram}
    where $\tstr[A] \hookrightarrow \tstr[B]$ is the inclusion of a subcomplex. Furthermore, suppose we are given the following data:
    \begin{enumerate}
        \item A function $\psi \colon \utstr[B]_{p_n} \to [0,1]$ such that $\psi^{-1}(1) = \utstr[A]_{p_n}$, which we consider as extended by $1$ to ${{\utstr[Y]}}_{p_n}$;
        \item An \aspire{} $R_{\tstr} \colon \utstr_{p_n} \times \sReal{\Delta^\I} \to \tstr$;
        \item An \aspire{} $R_{\tstr[B]} \colon \utstr[B]_{p_n} \times \sReal{\Delta^\I} \to \str[B]$ such that $R_{\tstr[B]}(x, u) \in \utstr[A]$, whenever $u_{p_n} = \psi(x)$.
    \end{enumerate}
    Then, we denote by $R_{{\tstr[Y]}}$ the map
        \begin{align*}
            R_{{\tstr[Y]}} \colon \utstr[Y]_{p_n} \times \sReal{\Delta^\I} &\to \utstr[Y] & \\
            ([x],u) &\mapsto \big ( (g \circ R_{\tstr[B],x}) \star_{\psi(x)} R_{\tstr[X]})(u) & \textnormal{, for } x \in \utstr[B] \\
             ([x],u) &\mapsto R_{\tstr[X]}(x,u) & \textnormal{, for } x \in X .
        \end{align*}
\end{construction}
\begin{lemma}\label{lem:gluing_aspires}
$R_{\tstr[Y] }$ as in \cref{con:gluing_aspires} defines an \aspire{} on $\tstr[Y]$, which extends $R_{\tstr[X]}$. 
\end{lemma}
\begin{proof}
    Note first that since all the spaces involved are finite cell complexes, we need not distinguish between the $\Delta$-generated topology and the relative topology on subspaces of the form $ T_{p_n} \cup T_{\leq_{p_i}}$ (see \cref{prop:top_agree_for_finite_cell}). In particular, both of these topologies also agree with the compactly generated topology.
    $R_{\tstr[Y]}$ may then equivalently be constructed as follows. 
    Consider the set $S \subset C^0(\real{\Delta^\I},Y) \times [0,1]$, defined as in \cref{prop:properties_of_star_homotopy} with respect to the closed inclusion $\tstr[X] \hookrightarrow \tstr[Y]$.  Furthermore, consider the map
    \begin{align*}
        R'^0 \colon \utstr[B]_{p_n} 
        &\to S &\xrightarrow{- \star R_{\str}} \HolIP(\tstr[Y]) \cup \utstr[Y]_{\leq p_0} \\
                    b & \mapsto (g \circ R_{\tstr[B],b}, \psi(b)) &\mapsto (g \circ R_{\tstr[B],b}) \star_{\psi(b)} R_{\tstr[X]}
    \end{align*}
    which is continuous by \cref{prop:properties_of_star_homotopy}. Note that, for $a \in A$, this map is given by 
    \[
    a \mapsto (g \circ R_{\tstr[B],a},1) \mapsto R_{\str}(g(a)) = i \circ R_{\tstr[X], f(a)}  .
    \]
    We claim that $R'^0$ extends to a continuous map
    \[
    R^0 \colon \utstr[B]_{p_n} \cup  \utstr[B]_{\leq p_0}
        \to \HolIP(\tstr[Y]) \cup \utstr[Y]_{\leq p_0} 
    \]
    by mapping $b \mapsto g(b)$, for $b \in B_{\leq p_0}$.
    Since $\utstr[B]_{p_n} \cup  \utstr[B]_{\leq p_0}$ is metrizable, it suffices to see that, for any sequence $b_m \in B_{p_n}$ converging to $b \in B_{\leq p_0}$, it also holds that $R'^0(b_m)$ converges to $R^0(b) = g(b)$ (with respect to the topology of uniform convergence on $\HolIP(\tstr[Y]) \cup \utstr[Y]_{\leq p_0}$). 
    Furthermore, we may without loss of generality assume that the sequence $\psi(b_m)$ converges. Indeed, this follows from the standard argument that a sequence $y_n$ converges to $y$, if and only if each of its subsequences has, in turn, a subsequence converging to $y$, together with compactness of $[0,1]$. For ease of notation, denote $\psi(b):= \lim_{m \to \infty}\psi(b_m)$.
    Let $D$ denote the subspace of $ \utstr[B]_{p_n} \cup  \utstr[B]_{\leq p_0}$, given by the elements of the sequence $b_m$ and $b$. Since $b_m$ converges to $b$, $D$ is a compact Hausdorff space. 
    It follows that the map 
    \begin{align*}
        D
        &\to S\\
                    b_m & \mapsto (g \circ R_{\tstr[B],b_m}, \psi(b_m)) \\
                    b   & \mapsto (g \circ R_{\tstr[B],b}, \psi(b)) 
    \end{align*}
    is continuous, both with respect to the subspace topology on $S$ as well as with respect to the compactly generated topology.
    In particular, the composition of the last map with $- \star R_{\str}$ is also continuous.
    It follows that 
    \[
    \lim_{m \to \infty}R'_0(b_m) = g \circ R_{\tstr[B],b} \star_{\psi(b)} R_{\str} = g(b) \star_{\psi(b)} R_{\str} = g(b)\] as was to be shown. \\
    It follows from the assumption on $i$ and $g$ being closed and \cref{lem:preservation_of_closed_pushout} that the square
    \begin{diagram}
        \utstr[A]_{\leq p_0} \cup \utstr[A]_{p_n} \arrow[d] \arrow[r, hook] & \utstr[B]_{\leq p_0} \cup \utstr[B]_{p_n} \arrow[d] \\
        \utstr_{\leq p_0} \cup \utstr_{p_n} \arrow[r, hook]& \utstr[Y]_{\leq p_0} \cup \utstr[Y]_{p_n}    \end{diagram}
        remains a pushout square.
    Hence, together with 
    \begin{align*}
    R^1 \colon \utstr_{p_n} \cup \utstr_{\leq p_0} &\to \HolIP(\tstr[Y]) \cup \utstr[Y]_{\leq p_0} \\
    x &\mapsto  i \circ R_{\tstr[X],x} \spacecomma
    \end{align*}
    $R^0$ glues to a map \[
    R \colon \utstr[Y]_{p_n} \cup \utstr[Y]_{\leq p_0} \to \HolIP(\tstr[Y]) \cup \utstr[Y]_{\leq p_0} 
    \]
    whose adjoint is the extension of $R_{\tstr[Y]}$ to $(\utstr[Y]_{p_n} \cup \utstr[Y]_{\leq p_0}) \times \sReal{\Delta^\I}$ as defined in the proposition. 
    This shows that $R_{\tstr[Y]}$ is indeed well-defined and stratum-preserving. It remains to verify that $R_{\tstr[Y]}$ interacts with lower strata, as required in the definition of an \aspire. We have already covered the case $ p=p_0$.  All other cases can be reduced to this one, by replacing $\I$ by $\I_{\geq p}$ and restricting $R_{\tstr[B]}$ and $R_{\tstr[X]}$ accordingly. That $R_{\tstr[Y]}$ extends $R_{\tstr[X]}$ is immediate by definition.
\end{proof}
Suppose now, for a second, that we have already shown the following lemma.
\begin{lemma}\label{lemma:extending_aspire_on_simplex}
    Let $\I = \standardFlag$ be a regular flag in $P$. Then, for any flag $\J$, there exists a function
    $\psi \colon \stanHood[\I]{\Delta^\J}_{p_n} \to [0,1]$, together with an \aspire{} $\overline{R} \colon \stanHood[\I]{\Delta^\J}_{p_n} \times \sReal{\Delta^\I} \to \stanHoods[\I]{\Delta^\J}$ such that \begin{enumerate}
        \item $\psi^{-1}(1) = \stanHood[\I]{\partial \Delta^\J}_{p_n}$;
        \item $\overline{R}$ restricts to the standard \aspire{} on $\stanHoods[\I]{\partial \Delta^\J}$ (see \cref{con:standard_aspire});
        \item For $u \in \sReal{\Delta^\I}$, and $x \in \stanHood[\I]{\Delta^\J}_{p_n}$ such that $u_{p_n} = \psi(x)$, we have $\overline{R}(x,u) \in \stanHood[\I]{\partial \Delta^\J}$.
    \end{enumerate}
\end{lemma}

Then we may proceed to show the following statement.
\begin{proposition}\label{prop:aspire_for_finite_cell}
    Let $\tstr$ be a finite stratified cell complex and $\Psi$ a barycentric subdivision of $\tstr$ that defines a standard neighborhood system of $\tstr$. Then, for any regular flag $\I \subset P$, $\PsiStanHood[\I]{\tstr}$ admits an \aspire that is compatible with subcomplexes, i.e. whenever $\tstr[B] \subset \tstr$ is a subcomplex of $\tstr$, then the \aspire{} on $\PsiStanHoods[\I]{\tstr}$ restricts to one on $\mathcal U^{\Psi|_{\tstr[B]}}_{\tstr[B]}(\I)$.
    Furthermore, if $\tstr[A] \subset \tstr[X]$ is a subcomplex, then for any such \aspire{} $R_{\tstr[A]}$ on $ \mathcal U^{\Psi|_{\tstr[A]}}_{\tstr[A]}(\I)$, the \aspire{} on $\PsiStanHoods[\I]{\tstr}$ may be taken to extend $R_{\tstr[A]}$.
\end{proposition}
\begin{proof}
    Via induction over the number of cells, it suffices to consider the case where $\tstr[A] \subset \tstr[X]$ differ only in one cell. Then, using \cref{prop:simplicial_standard_model}, we have a pushout diagram
        \begin{diagram}
       {   \stanHoods[\I]{\partial \Delta^{\J_i}} }\arrow[r, hook]\arrow[d] &  {\stanHoods[\I]{ \Delta^{\J_i}} \arrow[d]}\\
            {\mathcal U^{\Psi|_{\tstr[A]}}_{\tstr[A]}(\I)} \arrow[r, hook]& {\PsiStanHoods[\I]{\tstr}} 
        \end{diagram}
    of finite stratified cell complexes.
    We may then use \cref{con:gluing_aspires} together with \cref{lemma:extending_aspire_on_simplex} to extend the \aspire{} on ${\PsiStanHoods[\I]{\tstr[A]}[\Psi|_{\tstr[A]}]}$ to one on $\PsiStanHoods[\I]{X}$. One may verify directly from the construction in \cref{con:gluing_aspires} that the \aspires{} defined inductively in this fashion are compatible with subcomplexes.
\end{proof}
\begin{remark}
    One may hope that the construction in \cref{prop:aspire_for_finite_cell} generalizes to arbitrary cell complexes via transfinite composition. While it is true that the construction goes through, note that \cref{con:gluing_aspires} requires $X$ to be a finite cell complex. This assumption was needed to circumvent the subtle differences between $\Delta$-generated topology and relative topology described in \cref{ex:substrata_not_delta_gen}. Note, however that the main purpose of \aspires{} in this work is to compute homotopy links. For this, existence of \aspires{} on finite subcomplexes is sufficient.
\end{remark}
The analogue of \cref{prop:aspire_for_finite_cell} fails, if one instead changes the definition of \aspires{} such that they are required to extend continuously to $\utstr[U]_{\I} \times \sReal{\Delta^\I} \to \tstr[U]$. Let us give an example to illustrate this:
\begin{example}\label{ex:aspire_gen_not_ext}
    Let $\I = P = \{ p< p_1 <p_2\}$. Consider the flag $\J= [p_0 \leq p_1 \leq p_1 \leq p_2]$. Now, we may glue $\sReal{\Delta^\I}$ to $\sReal{\Delta^\J}$, along any stratum-preserving map $\xi \colon \sReal{\Delta^{[p_0 < p_1]}} \to \sReal{\Delta^\J}$. Denote the resulting stratified cell complex by $\tstr$ and let $\tstr[A]$ be the subcomplex defined by $\sReal{{\Delta^\J}}$. 
    Next, fix any barycentric subdivision $\Psi$ of the stratified cell complex $\tstr$, and denote by $\Phi$ the induced subdivision of $\sReal{\Delta^\I}$ (by treating the latter as a cell of $\tstr$). Furthermore, denote $U:=\snVal[\str]^{\Psi}(\I)$ and $V:=\snVal[{\tstr[A]}]^{\Psi_{\tstr[A]}}(\I) \subset U$, and by $V'$ the image of $\tstr[U]^{\Phi}_{\sReal{\Delta^\I}}$ in $\tstr$. 
    Suppose we are given a map
    \[
    R \colon U \times \sReal{\Delta^\I} \to U,
    \]
    which is stratum-preserving when restricted to $(U \cap \tstr_{p_n}) \times \sReal{\Delta^\I}$, and fulfills $R(x,v_i) =x$, for $i \in [2]$, $v_i$ the vertex of $\sReal{\Delta^\I}$ corresponding to $p_i \in \I$ and $x \in \str_{p_i}$.
    In particular, all of these properties are consequences of the altered definition of \aspires{} we are investigating.
    For connectivity reasons, using the fact that the $p_2$ stratum consists of two disjoint cells, $R$ must also fulfill $R(x,u) \in V$, for any $x \in V$, as well as $R(x,u) \in V'$, for any $x \in V'$ and all $u \in \sReal{\Delta^\I}$.
    Consequently, it follows that $R(V \cap V' \times \sReal{\Delta^\I}) \subset V \cap V'$. In particular, it follows that $R$ restricts to a map 
    \[
    R' \colon (V \cap V' ) \times \sReal{\Delta^{[p_0 < p_1]}} \to  V \cap V' .
    \]
    By identifying $\sReal{\Delta^{[p_0 < p_1]}}$ with the interval $[0,1]$, it follows that $R'$ defines a homotopy between the identity and the constant map with value the unique point $y$ in the $p_0$ stratum of $\tstr$. 
    Note that $V \cap V'$ is of the form $\xi([0,a])$, for some $a>0$. Since $\xi$ was allowed to be arbitrarily complicated,  there is no reason to assume that the image $\xi([0,a])$ is contractible, for any choice of $a$ (think of a spiral converging to $y$ which intersects itself infinitely often, as it does so). Note that if $\xi$ was piecewise linear, we could indeed assume contractibility for sufficiently small $a$.
\end{example}
We may now finally prove the following theorem:
\begin{theorem}\label{thm:cell_stanhood_is_holink}
    Let $\tstr$ be a $P$-stratified cell complex and let $\Psi$ be any subdivision of $\tstr$ that induces a strata-neighborhood system. Then $\PsiStanHoodSys$ defines a homotopy link model for $\tstr$.
\end{theorem}
\begin{proof} 
    By the standard compactness arguments, it suffices to show the case where $\tstr$ is a finite stratified cell complex.
    First, let us show that the maps $\PsiStanHood[\I]{\tstr}_{p_n} \hookrightarrow \PsiStanHood[\I]{\tstr}_{\geq p_n}$ are weak equivalences. Let us first note that the result holds when $\tstr$ is the realization of a stratified simplicial complex $\str[K] = \partial \Delta^\J, \Delta^\J$, for some flag $\J$ in $\pos$ and $\Psi$ is the subdivision given by \cref{con:simplicial_model_of_standard}.
    Indeed, then we have $\PsiStanHoods[\I]{\sReal{\str[K]}} = \stanHoods[\I]{\str[K]}$, for which the result holds by \cref{prop:simplicial_standard_model}. Next, let us proceed to show the result for $\tstr$ a finite cell complex, via induction over the number of cells.
    Suppose $\tstr$ is obtained by gluing a cell $\sigma \colon \sReal{\Delta^{\J}} \to \tstr$ along $\sReal{\partial \Delta^\J} \to \tstr[A]$, for some finite complex $\str[A]$. Using \cref{prop:simplicial_standard_model} it follows that there is a pushout diagram of $P$-stratified spaces
    \begin{diagram}
       { \stanHoods[{\I}]{\partial \Delta^\J}} \arrow[r, hook] \arrow[d]& { \stanHoods[\I]{\Delta^\J}} \arrow[d]\\
       {\PsiStanHoods[\I]{\tstr[A]}[\Psi|_{\tstr[A]}]} \arrow[r] & {\PsiStanHoods[\I]{\tstr}.}
    \end{diagram}
    From this, we obtain the following commutative cube.
    \begin{diagram}
       {} && {\stanHood[{\I}]{\partial \Delta^\J}}_{\geq p_n} \arrow[r, hook] \arrow[dd]&   \stanHood[\I]{\Delta^\J}_{\geq p_n} \arrow[dd] \\
         {\stanHood[{\I}]{\partial \Delta^\J}}_{p_n} \arrow[rru, "\simeq"] \arrow[r, hook] \arrow[dd] & {\stanHood[{\I}]{\Delta^\J}}_{p_n}\arrow[rru, "\simeq"] \arrow[dd]&  \\
        &  &  {\PsiStanHood[\I]{\tstr[A]}[\Psi|_{\tstr[A]}]}_{\geq p_n} \arrow[r, hook] &  \PsiStanHood[\I]{\tstr}_{ \geq p_n}  \\
        {\PsiStanHood[\I]{\tstr[A]}[\Psi|_{\tstr[A]}]}_{p_n} \arrow[rru, "\simeq"] \arrow[r, hook]& \PsiStanHood[\I]{\tstr}_{p_n} \arrow[rru] && \spaceperiod
    \end{diagram}
    By inductive assumption all the diagonal maps but the lower vertical one are known to be weak homotopy equivalences in $\TopN$. By the standard properties of homotopy pushouts (see for example \cite[Prop 13.5.4]{hirschhornModel}) it suffices to show that the front and the back face of this cube are homotopy cocartesian. 
    This follows from \cref{lem:cell_cplxs_are_preserved} together with \cref{lem:preservation_of_closed_pushout} and the characterization of homotopy cocartesian squares in a model category in \cite[Prop. A.2.4.4]{HigherTopos}. \\
    Next, we need to show that for any regular flag $\I = \standardFlag \subset P$, the natural map
    \[
    \HolIP(\PsiStanHoods[\I]{\str[X]}) \xrightarrow{\ev} \PsiStanHood[\I]{\str[X]}_{p_n}
    \]
    is a weak equivalence. This follows directly from \cref{prop:aspire_for_finite_cell} together with \cref{prop:computing_holinks_via_aspire}.
\end{proof}
As a corollary of \cref{thm:cell_stanhood_is_holink}, we obtain the following result, which is central to our investigation of the stratified homotopy hypothesis in \cite{TSHHWa}. 
\begin{corollary}\label{cor:hol_diag_is_ho_pushout}
    Let $\I = \standardFlag $ be  a regular flag in $P$ and
    suppose we are given a pushout diagram of stratified cell complexes in $\TopPN$
    \begin{diagram}\label{diag:cor:cell_pushout}
        {\tstr[A]} \arrow[d, "f"] \arrow[r, hook] & \tstr[B] \arrow[d , "g"] \\
        \tstr[X] \arrow[r, hook, "i"]& \tstr[X] \cup_{\tstr[A]} \tstr[B] ={{\tstr[Y]}} 
    \end{diagram}
    where $\tstr[A] \hookrightarrow \tstr[B]$ is the inclusion of a subcomplex. Then the image of this square under $\HolIP$
    \begin{diagram}\label{diag:holink_diag_proof}
     \HolIP{\tstr[A]} \arrow[r] \arrow[d] & \HolIP{\tstr[B]} \arrow[d] \\
     \HolIP{\tstr[X]} \arrow[r] & \HolIP{\tstr[Y]}
    \end{diagram}
    is homotopy cocartesian in $\TopN$.
\end{corollary}
\begin{proof}
    By \cref{cor:models_for_homotopy_pushouts_half}, \cref{diag:cor:cell_pushout} lifts to a diagram of strata-neighborhood systems 
    \begin{diagram}\label{diag:lift_of_pushout_diag_proof}
        \PsiStanHoodSys[\str [A]][\Phi] \arrow[r] \arrow[d] &  \PsiStanHoodSys[\str [B]][\hat \Phi] \arrow[d] \\
        \PsiStanHoodSys[\str [X]][\Psi]  \arrow[r] & \ \PsiStanHoodSys[\str[Y]][\hat \Psi]  ,
    \end{diagram}
    for appropriate choice of subdivisions $\Phi, \hat \Phi, \Psi, \hat \Psi$. By \cref{thm:cell_stanhood_is_holink}, \cref{diag:lift_of_pushout_diag_proof} is a diagram of homotopy link models. Thus, by \cref{prop:models_model_globally}, \cref{diag:holink_diag_proof} is weakly equivalent to the image of \cref{diag:lift_of_pushout_diag_proof} under $\NDiagT$ at $\I$. That the latter is homotopy cocartesian is the content of \cref{lem:comp_nbhd_preserves_cell_cplx}.
\end{proof}
Finally, to finish this section, we still need to provide a proof of \cref{lemma:extending_aspire_on_simplex}.
\begin{proof}[Proof of \cref{lemma:extending_aspire_on_simplex}]
Using \cref{prop:simplicial_standard_model} we may identify $\stanHood{\Delta^\J}$ with the realization of $\simStanHood[\I]{\Delta^{\J}} = \bigcap_{p \in \I} \simStanHood{\Delta^{\J}}$ as defined in \cref{con:simplicial_model_of_standard} and proceed analogously with $\partial \Delta ^\J$. As full subcomplexes of $\sd \Delta^\J$ the two complexes $\simStanHood[\I]{\Delta^{\J}}$ and $\simStanHood[\I]{\partial \Delta^{\J}}$ only differ in the vertex corresponding to the maximal simplex of $\Delta^{\J}$, $x_\J$. 
 If $x_\J \in S_{\Delta^\J}(\I)$, then either $\I \subset \J$ or $\max \J < p_k$, for some $k \in [n]$. If $\max \J < p_k < p_n$, then $\stanHood[\I]{\Delta^\J}_{p_n} = \emptyset$ and there is nothing to show. Hence, we may assume that $\max \J = p_n$ and $\I \subset \J $. \\
 \textbf{Step 1: }Let $x_0 \in \sReal{\Delta^\J} \setminus \bigcup_{p \in \I, p < p_n}\stanHood[p]{\Delta^\J} \cup \sReal{\partial \Delta^\J}$. That such a point exists is a consequence of the inclusion $\I \subset \J$. Indeed, any point $x$ in the interior of $\sReal{\Delta^\J}$, with $s_{p_n}(x) > \frac{1}{2}$ will do. Next, consider the straight line projection through $x_0$
 \begin{align*}
     r \colon \sReal{\Delta^\J} \setminus \{x_0\} \to \sReal{\partial \Delta^\J}. 
 \end{align*}
Let us show that $r$ maps $\stanHood[\I]{\Delta^\J}$ to $\stanHood[\I]{\partial \Delta^\J}$.
By definition of $\stanHood[\I]{-}$, and using that $\max \J = p_n$, we may instead show that $r$ maps $\stanHood{\Delta^\J}$ to $\stanHood{\partial \Delta^\J}$ for all $p \in \I_{< p_n}$.
The map $r$ maps $x$ to the intersection point of the ray \[
 \{ x + \alpha(x - x_0) \mid \alpha \geq 0
\}
\]
with $\stanHood{\partial \Delta^\J}$.
Let $\alpha_x$ be the unique value in $[0, 1]$, specifying this intersection point.
In particular, for any $p \in \I_{< p_n}$, we may compute
\begin{align*}
    s_{ \not \leq p}(r(x))  &= s_{ \not \leq p}(  x + \alpha_x(x - x_0) ) &= (1 + \alpha_x) s_{\not \leq p}(x) - \alpha_x s_{ \not \leq p}(x_0) \\
     s_{ p}(r(x))  &=  \cdots &= (1 + \alpha_x) s_{p}(x) - \alpha_x s_{p}(x_0).
\end{align*}
By assumption, $s_{\not \leq p}(x_0) > s_{p}(x_0)$   and $s_{ \not \leq p}(x) \leq s_p(x)$. 
Since, $\alpha_x \geq 0$, it follows that
\[ s_{\not \leq p} (r(x)) = (1+\alpha_x) s_{ \not \leq p}(x) - \alpha_x s_{\not \leq p}(x_0) \leq (1+\alpha_x) s_{ p}(x) -\alpha_x  s_{p}(x_0) = s_p(r(x)),\]
as was to be shown. \\
\textbf{Step 2: } We need to verify an additional property of $r$, namely that it is close to being stratum-preserving.
We show that for $x \in \stanHood[\I]{\Delta^\J}_{p_n}$
\begin{equation}\label{equ:r_is_not_evil}
    s(r(x)) \geq p_{n-1}.
\end{equation}
Assume, to the contrary that $\alpha_x> 0$ and $s_{q}(r(x)) = 0$ , for all $q \geq p_{n-1}$. Then, for such $q$, we have
\[
 s_{q}(x) = \frac{\alpha_x}{1+ \alpha_x} s_q(x_0)
\]
and 
\[
s_{\not \leq q}(x)  =\frac{\alpha_x}{1+ \alpha_x} s_{\not \leq q}(x_0)
\]
and obtain
\[
s_{p_{n-1}}(x) =  \frac{\alpha_x}{1+ \alpha_x} s_{p_{n-1}}(x_0) <  \frac{\alpha_x}{1+ \alpha_x} s_{ \not \leq p_{n-1}}(x_0) = s_{\not \leq p_{n-1}}(x)
\]
in contradiction to the assumption that $x \in \stanHood[\I]{\Delta^\J}$. \\
\textbf{Step 3: }Denote by $R$ the standard \aspire{} on $\stanHood[\I]{\Delta^\J}$ (see \cref{con:standard_aspire}). Furthermore, denote by $\hat R \colon \stanHood[\I]{\Delta^\J} \times \sReal{\Delta^\I}$ the (extended) \aspire{} obtained by affinely extending 
\begin{align*}
    (x,v_k) &\mapsto x & \textnormal{, for } k=n \\
    (x,v_k) &\mapsto \rho^{p_k}(r(x))  & \textnormal{, for } k<n
\end{align*}
where $v_k$ denotes the $k$-th vertex of $\sReal{\Delta^\I}$ and $\rho^p$ are as in \cref{con:retract_to_strata}.
$\hat R$ is well defined by \cref{lem:retraction_preserve_nbhds} and the fact that $r$ maps into $\stanHood[\I]{\Delta^\J}$. If $x \in \stanHood[\I]{\partial \Delta^\J}_{p_n}$, then $r(x) =x$ and hence $\hat R_{x}$ agrees with $R$.
Furthermore, it follows from the inclusion $\stanHood[\I]{\Delta^\J}_{< p_n} \subset \stanHood[\I]{\partial \Delta^\J}$ that $\hat R$ agrees with $R$ on $\stanHood[\I]{\Delta^\J}_{\leq p} \times \sReal{\Delta^{\I_{\geq p}}}$, for any $p \in \I$, $p<p_n$. Clearly, also $\hat R(x, v_n)=x$. Hence, to see that $\hat R$ does indeed define an \aspire{}, we only need to verify that $\hat R_x$ is stratum-preserving for $x \in \stanHood[\I]{\Delta^\J}_{p_n}$. Just as for the proof of the analogous statement in \cref{prop:stan_aspire_is_aspire}, one shows that $\hat R_x$ being stratum-preserving is equivalent to showing that $r(x) \in \stanHood[\I]{\Delta^\J}_{\geq p_{n-1}}$ which is the content of \cref{equ:r_is_not_evil}. \\
\textbf{Step 4: }The idea of the remainder of the proof is to now combine $\hat R$ and $R$. To do so, we will make use of a (continuous) function $\psi \colon \stanHood[\I]{\Delta^\J}_{p_n} \to [0,1]$, with the properties that
\begin{enumerate}
    \item $\psi^{-1}(0) = r^{-1}(\stanHood[\I]{\partial \Delta^\J}_{< p_n}) \cap \stanHood[\I]{\Delta^\J}_{p_n}$;
    \item $\psi^{-1}(1) = \stanHood[\I]{\partial \Delta^\J}_{p_n}$.
\end{enumerate}
Notice that both of these sets are closed subsets of $\stanHood[\I]{\Delta^\J}_{p_n}$  and that since $r(x)=x$, for $x \in \stanHood[\I]{\partial \Delta^\J}$, they are also disjoint. Hence, such a function $\psi$ exists. 
Furthermore, we are going to need another interpolation function
\begin{align*}
    H \colon \stanHood[\I]{\Delta^\J} \times \sReal{\Delta^\I} \times [0,1] &\to \stanHood[\I]{\Delta^\J} \\
    (x,u,t) &\mapsto (1-t) \hat R(x,u) + t R(r(x),u).
\end{align*}
Then one may verify the following properties of $H$:
\begin{HProps}
    \item \label{proof:extending_aspire_on_simplex:item1} $H_0 = \hat R$ and $H_1$ has value in $\stanHood[\I]{\partial \Delta^\J}_{p_n}$.
    \item \label{proof:extending_aspire_on_simplex:item2} If $x \in \stanHood[\I]{\partial \Delta^\J}$, then $H(x,-,-)$ is the constant homotopy with value $R_{x}$.
    \item \label{proof:extending_aspire_on_simplex:item3}  If $s(x) = p_n$, then for any $t < 1 $, $H(x,-,t)$ is stratum-preserving.
    \item \label{proof:extending_aspire_on_simplex:item4} If $\psi(x) >0$ and $s(x) = p_n$, then $H(x,-,-)$ is stratum-preserving. 
    \item \label{proof:extending_aspire_on_simplex:item4.5} Restricted to $\stanHood[\I]{\Delta^\J} \times \sReal{\Delta^{\I_{<p_n}}}$, $H$ is given by the constant homotopy of value $R_{r(x)}|_{\sReal{\Delta^{\I_{<p_n}}}}$.
    \item \label{proof:extending_aspire_on_simplex:item5} If $s(x) = p_n$ and $\psi(x) = 1$, then $H(x,-,-)$ is the constant homotopy with value $ R_{r(x)} = R_{x}$. 
\end{HProps}
\textbf{Step 5: }We may now finally define the \aspire{} promised in the statement of the proposition.
Consider the map
\begin{align*}
    \overline{R} \colon \stanHood[\I]{\Delta^\J} \times \sReal{\Delta^\I} &\to \stanHood[\I]{\Delta^\J} \\
            (x,u) &\mapsto H(x,u, 1) & \textnormal{, for } x \in  \stanHood[\I]{\Delta^\J}_{<p_n} \\
            (x,u) &\mapsto H(x,u,1) & \textnormal{, for } x\in \stanHood[\I]{\Delta^\J}_{p_n} \textnormal{ and } u_{p_n} \leq \psi(x) \\
            (x,u) &\mapsto H(x,u, \frac{1-u_{p_n}}{1- \psi(x)} ) & \textnormal{, for } x\in \stanHood[\I]{\Delta^\J}_{p_n} \textnormal{ and } u_{p_n} > \psi(x).
 \end{align*}
Let us verify the continuity of $\overline{R}$. Notice that the first two conditions on $(x,t)$ define a closed subspace $D \subset \stanHood[\I]{\Delta^\J} \times \sReal{\Delta^\I}$. Hence, the only thing to check is that for any sequence $(x_i,u_i)$, $i \in \mathbb N$, with $x_i \in \stanHood[\I]{\Delta^\J}_{p_n}$ and $(u_i)_{p_n} > \psi(x_i)$, converging to $(x,u) \in D$, it also follows that $H(x_i,u_i,\frac{1-(u_i)_{p_n}}{1- \psi(x_i)} )$ converges to $H(x,u,1)$. In the following, by convergence of functions we will always mean uniform convergence. There are two cases to consider. If $x \in \stanHood[\I]{\Delta^\J}_{<p_n} \subset \stanHood[\I]{\partial \Delta^\J}$, then by \cref{proof:extending_aspire_on_simplex:item2} $H(x_i, - ,-)$ converges to a constant homotopy and hence 
$
H(x_i,u_i,\frac{1-(u_i)_{p_n}}{1- \psi(x_i)} )
$
converges to $H(x,u,1)$. If $x \in \stanHood[\I]{\Delta^\J}_{p_n}$, then, by assumption, $u_{p_n} = \psi(x)$ and thus if $\psi(x) < 1$ continuity is immediate from the definition. It remains to consider the case $u_{p_n} = \psi(x) = 1$.
In this case, it follows from \cref{proof:extending_aspire_on_simplex:item5} that $H(x_i,-, -)$ converges to the constant homotopy with value $R_{x}$. Hence, again it follows that $H(x_i,u_i,\frac{1-(u_i)_{p_n}}{1- \psi(x_i)} )$ converges to $H(x,u,1)$. \\
\textbf{Step 6: }
Let us now verify that $\overline{R}$ restricts to an \aspire. 
If $x \in \stanHood[\I]{\Delta^\J}_{p_n}$ and $\psi(x) > 0$, then $\overline{R}_x \colon \sReal{\Delta^\I} \to \stanHood[\I]{\Delta^\J}$ is stratum-preserving by \cref{proof:extending_aspire_on_simplex:item4}. If $\psi(x) = 0$, then $\overline{R}_x(s) = H(x,u,1-u_{p_n})$. Hence, by \cref{proof:extending_aspire_on_simplex:item3}, we obtain preservation of strata for $u_{p_n} > 0$. For $u_{p_n} = 0$  and $\psi(x) =0$, it follows by \cref{proof:extending_aspire_on_simplex:item4.5} that then $H(x,u,1)) = R(r(x),u)$. Since $s(r(x)) \geq p_{n-1}$, by \cref{equ:r_is_not_evil}, it follows that $R_{r(x)}|_{\sReal{\Delta^{\I_{<p_{n-1}}}}}$ is stratum-preserving. 
This shows that the restriction of $\overline{R}$ 
\[
R \colon \stanHood[\I]{\Delta^\J}_{p_n} \times \sReal{\Delta^\I} \to \stanHood[\I]{\Delta^\J} 
\]
is stratum-preserving. Finally, let $p \in \I$, $u \in \sReal{\Delta^{\I_{\geq p}}}$ and $x \in \stanHood[\I]{\Delta^\J}_{\leq p}$.
Note first that whenever $s(x) < p_n$ or $\psi(x) = 1$, then $x \in \stanHood[\I]{\partial \Delta^\J}$ and thus 
\[ \overline{R}(x,u) = H(x,u,1) = R(x,u) =x \] by \cref{proof:extending_aspire_on_simplex:item2}. Furthermore, If $s(x) = p_n$ and $\psi(x) <1$ then, by the assumption that $s(x) \leq p$ and $u \in  \sReal{\Delta^{\I_{\geq p}}}$, it follows that $u_{p_n} = 1$ and hence \[\overline{R}(x,u)= H(x,u,0) = \hat R(x,u)  = x\] by \cref{proof:extending_aspire_on_simplex:item1}. To summarize, we have shown that $\overline R$ is an \aspire{}. By \cref{proof:extending_aspire_on_simplex:item2}, $R$ defines an extension of the standard \aspire{} on $\stanHood[\I]{\partial \Delta^\I}$. Finally, if $\psi(x) =u_{p_n}$, for $(x,u) \in \stanHood[\I]{\Delta^\I}_{p_n} \times \sReal{\Delta^\I}$, then $R(x,u) = H(x,u,1) \in \stanHood[\I]{\partial \Delta^\J}$, by \cref{proof:extending_aspire_on_simplex:item1}, which finishes the proof.
\end{proof}
\section*{Acknowledgments}
The author is supported by a PhD-stipend of the Landesgraduiertenförderung Baden-Württemberg. 
 \printbibliography
\appendix
    \section{A series of tools from point-set topology}
In this section, we list a series of elementary results in point-set topology.
\begin{lemma}\label{lemma:neighborhoods_of_pushouts}
    Consider a pushout diagram of compact Hausdorff spaces
    \begin{diagram}
         A \arrow[r] \arrow[d] & B \arrow[d, "f'"] \\
       X \arrow[r, "g'" ]& Y \spaceperiod
    \end{diagram}
     Let $Z \subset U \subset Y$. If both $g'^{-1}(U)$ is a neighborhood of $g'^{-1}(Z)$ in $X$ and $f'^{-1}(U)$ is a neighborhood of $f'^{-1}(Z)$ in $B$, then $U$ is a neighborhood of $Z$ in $Y$. 
\end{lemma}
\begin{proof}
    It is generally true that for any quotient map $\pi \colon T \to \tilde Y$ of compact Hausdorff spaces the image of any neighborhood $\tilde V$ of $\pi^{-1}(\tilde Z)$ is a neighborhood of $\tilde Z$. Indeed, for any open $O \subset \tilde V$, containing $\pi^{-1}(Z)$, the set $ T \setminus \pi^{-1}( \pi( T \setminus O))$ is a saturated open set (this uses that $T \setminus O$ is closed, due to the compact-Hausdorff assumptions), which contains $\pi^{-1}(\tilde Z)$ and is contained in $\tilde V$. Since the map $X \sqcup B \to Y$ is such a quotient map and by assumption $g'^{-1}(U) \sqcup f'^{-1}(U)$ is a neighborhood of $g'^{-1}(Z) \sqcup f'^{-1}(Z)$, it follows that $U = g'(g'^{-1}(U)) \cup f'(f'^{-1}(U))$ is a neighborhood of $Z$.
\end{proof}
\begin{lemma}\label{lem:preservation_of_closed_pushout}
    Let \begin{diagram}
        \tstr[A] \arrow[r] \arrow[d] & \tstr[B]\arrow[d] \\
        \tstr[X] \arrow[r] & \tstr[Y]
    \end{diagram}
    be a cocartesian square in $\TopPN$ such that all arrows pointing into $\utstr[Y]$ are closed maps (or open maps). Then, for any subset $Q \subset P$, the square of general topological spaces
    \begin{diagram}
        \utstr[A]_Q \arrow[r] \arrow[d] & \utstr[B]_Q \arrow[d] \\
        \utstr_Q \arrow[r] & \utstr[Y]_Q
    \end{diagram}
    remains cocartesian. Furthermore, if $Q$ is open or closed in $P$, then the square remains cocartesian without any assumptions on the maps.
\end{lemma}
\begin{proof}
We cover the closed cases. 
    We need to verify that the bijection
    \[
    \phi \colon X_Q \cup_{A_Q} B_Q \to Y_Q
    \]
    is a closed map. Now, any closed set $ Z \subset X_Q \cup_{A_Q} B_Q$ is given by the image of some closed set $Z_X \sqcup Z_B \subset X_Q \sqcup B_Q$. The latter is given by the restriction to $Q$, of some closed set $\tilde Z_X \sqcup \tilde Z_B \subset X \sqcup B$. Denote by $\tilde Z$ the image of $\tilde Z_X \sqcup \tilde Z_B$ in $Y=X \cup_{A} B$. By assumption, $\tilde Z$ is again closed, and by construction we have $\phi(Z) = \tilde Z_Q$. The second statement follows similarly, using the fact that then $Y_{Q} \hookrightarrow Y$ is given by a closed inclusion. 
\end{proof}
\begin{lemma}\label{lem:colim_of_upperset}
    Let $p \in P$. Furthermore, let $\TopN$ be any of the categories of topological spaces in \cref{not:notation_top_spaces}. The functor 
    \[
    (-)_{\geq p} \colon \TopPN \to \TopN
    \]
 preserves all colimits. 
\end{lemma}
\begin{proof}
    First, let us show that $(-)_{\geq p}$ preserves all colimits of general topological spaces. It is an immediate consequence of the more general statement on over-categories of topological spaces, which one may easily verify using the elementary construction of colimits via final topologies. For any space $T \in \TopN$ and any open subspace $U \subset T$, the restriction functor
    \[
    \TopN_{/T} \to \TopN_{/U}
    \]
    preserves all colimits. Since $\{ q \geq p \} \subset \Pos$ is an open subset in the Alexandrow topology and the forgetful functor $\TopN_{/U} \to \TopN$ admits a right adjoint, the result follows.
    Next, let us cover the case of $\Delta$-generated spaces. The one of compactly generated spaces is analogous. If $D \colon I \to \TopPN$ is a diagram of $\Delta$-generated stratified spaces, then since the inclusion into general topological spaces is left adjoint (see, for example, \cite{GaucherDelta}), the colimit of $D$ in $\Delta$-generated spaces also defines the colimit in general topological spaces. Hence, we may apply the previous case, to see that $\colim D_{\geq p} = (\colim D)_{\geq p}$ in general topological spaces. Now, since $(T)_{\geq p} \subset T$ always defines an open subspace, and open subspaces of $\Delta$-generated subspaces are again $\Delta$-generated (see \cite[Sec. 2]{GaucherDelta}), the diagram $D_{\geq p}$ lives in the category of $\Delta$-generated spaces. Since the inclusion into general spaces preserves all colimits, we also have $\colim D_{\geq p} = (\colim D)_{\geq p}$ in $\Delta$-generated spaces, as was to be shown. 
\end{proof}
\begin{lemma}\label{lem:cell_cplxs_are_preserved}
    Let $p \in \pos$. The functors
    \[
    (-)_{\geq p}, (-)_{p} \colon \TopPN \to \TopN
    \]
    and
    \[
    (-)_{\leq p} \colon \TopPN \to \TopPN
    \]
    send relative stratified cell complexes into relative (stratified) cell complexes.
\end{lemma}
\begin{proof}
    Let us begin with $(-)_{\geq p}$.
    By \cref{lem:colim_of_upperset}, it suffices to show that $(-)_{\geq p}$ sends stratified boundary inclusions $\sReal{\partial \Delta^{\J}} \hookrightarrow \sReal{\Delta^{\J}}$ into a relative cell complex of topological spaces. Let us assume that not all elements of $\J$ are smaller then $p$, otherwise both spaces are empty after applying $(-)_{\geq p}$, and there is nothing to be shown. Then, applying $(-)_{\geq p}$ corresponds to removing a face of the simplex $\sReal{\Delta^{\J}}$ from both spaces. In particular, we may reduce to the following general statement:
    Let $T \hookrightarrow T'$ be an inclusion of a piecewise linear closed subspace into a piecewise linear space $T$. Let $A \subset T$ be a further inclusion of a piecewise linear closed subspace. Then, 
    $T \setminus A \hookrightarrow T' \setminus A$ also admits the structure of a closed inclusion of a piecewise linear subspace (this is ultimately a consequence of the existence of piecewise linear regular neighborhoods).
    In particular, there is a compatible triangulation of $T\setminus A$ and $T' \setminus A$, which makes $T \setminus A \hookrightarrow T' \setminus A$ a relative cell complex. 
    The case of $(-)_{\leq p}$ follows similarly by the natural isomorphisms $(\sReal{\str})_{\leq p} \cong \sReal{\str_{\leq p}}$, for $\str \in \sSetPN$. Finally, the case of $(-)_p$ follows from the equality $(-)_{p} = (-)_{\geq p} \circ (-)_{\leq p}$.
\end{proof}
\begin{proposition}\label{prop:top_agree_for_finite_cell}
    Let $\tstr \in \TopPN$ be a finite stratified cell complex. Then for any $Q \subset P$ the relative topology on $\utstr_Q \subset \utstr$ makes $\utstr_Q$ a $\Delta$-generated space.
\end{proposition}
\begin{proof}
    First, let us show that for any flag $\J $ of $P$, the space $(\sReal{\Delta^\J})_Q$ with the relative topology is $\Delta$-generated. 
    $(\sReal{\Delta^\J})_Q \subset \sReal{\Delta^\J} \subset \mathbb R^\J$ may equivalently described by 
    \[
    \{ s \in \sReal{\Delta^\J} \mid \forall p \in P \setminus Q \colon ( \exists q \in Q  \colon  q > p \land s_q > 0) \lor s_p = 0 \}.
    \]
   It follows from this description that $(\sReal{\Delta^\J})_Q \subset \sReal{\Delta^\J} \subset \mathbb R^\J$ is a convex set. 
    It turns out that every convex subset $C$ of $\mathbb R^n$ is $\Delta$-generated. Indeed, let $A \subset C$ be such that $\sigma^{-1}(A)$ is closed, for every continuous map $\sigma \colon \real{\Delta^1} \to C$. Let $x_n$, $n \in \mathbb N$, be a sequence in $A$ which converges to $c \in C$. Since $C$ is convex, we may use affine interpolation to define a continuous map $\sigma \colon [0,1] \to C$ with $\sigma( 2^{-n}) = x_n$. In particular, the inverse image of $A$ under $\sigma$ contains $\{ 2^{-n} \mid n \in \mathbb N\}$. As $\sigma^{-1}(A)$ is closed, it follows that $0 \in \sigma^{-1}(A)$. By continuity of $\sigma$, we hence have $c = \sigma(0) \in A$, showing that $A$ is closed. \\
    We now proceed to show the case of a general complex $\tstr $ via induction over the number of cells. 
    The case $n=0$ is trivial, so let $\tstr$ admit the structure of a stratified cell complex with $n+1$ cells.
    In other words $\tstr$ fits into a pushout diagram
    \begin{diagram}
        \sReal{\partial \Delta^\J} \arrow[r, hook ]  \arrow[d] &\sReal{\Delta^\J} \arrow[d]\\
        \tstr[A] \arrow[r] & \tstr \spacecomma
    \end{diagram}
    with $\tstr[A]$ a stratified space admitting a cell structure with $n$ cells.  By  \cref{lem:preservation_of_closed_pushout}, the diagram
    \begin{diagram}
        (\real{\partial \Delta^\J})_Q \arrow[r, hook ]  \arrow[d] & (\real{\Delta^\J})_Q \arrow[d]\\
        \utstr[A]_Q \arrow[r] & \utstr[X]_Q
    \end{diagram}
    is a pushout diagram of general topological spaces. In particular $\utstr[X]_Q$ is a quotient of $A_Q \sqcup (\sReal{\Delta^\J})_Q$. We have already seen that $(\sReal{\Delta^\J})_Q$ is $\Delta$-generated.
    By the inductive assumption, the same holds for $\utstr[A]_Q$. Thus $\utstr[X]_Q$ is $\Delta$-generated as a quotient of $\Delta$-generated spaces. 
\end{proof}
\begin{example}\label{ex:substrata_not_delta_gen}
    The statement of \cref{prop:top_agree_for_finite_cell} is generally not true for infinite stratified cell complexes, even if $X$ is given by the realization of a stratified simplicial set. Let $P= \{ p_0 < p_1 < p_2\}$ and $Q = \{ p_0 < p_2\}$. Consider the realization of the stratified simplicial set given by gluing countably many $\Delta^P$ along $\Delta^{ \{p_0 < p_1 \}}$, i.e., there is a pushout diagram
    \begin{diagram}
        \bigsqcup_{n \in \mathbb N} \sReal{\Delta^{ \{p_0 < p_1 \}}} \arrow[r] \arrow[d] & \bigsqcup_{n \in \mathbb N} \sReal{\Delta^P} \arrow[d] \\
        \sReal{\Delta^{ \{p_0 < p_1\}}} \arrow[r, "i"] & \tstr \spaceperiod 
    \end{diagram}
    Then, $\utstr_Q$ is not $\Delta$-generated. In the following, we denote closures in the form $\overline{A}$. To see that $\utstr_Q$ is not $\Delta$-generated, consider the subset $S$ of $\utstr$ given by
    \[
    \bigcup_{n \in \mathbb N} S_n
    \]
    where 
    \[
    S_n=\{ s \in \sReal{\Delta^P} \mid  s_{p_0} \leq 1-\frac{1}{n} \}
    \]
    lies in the $n$-th copy of $\sReal{\Delta^{\pos}}$ in $X$. 
    The set $S_Q \subset X_Q$ is not closed in the relative topology on $X_Q$.
    To see this, let $A \subset X$ be any closed set containing $S_Q$. Then, as $(S_n)_Q$ is dense in $S_n$, it follows that $A$ contains $S$. Observe that $S$ contains the image of the $p_1$-stratum of $\sReal{\Delta^{\{p_0<p_1\}}}$ under $i \colon \sReal{\Delta^{\{p_0<p_1\}}} \to \str$, but it does not contain any point in the $p_0$-stratum of $X$. Denote by $x_0$ the unique element in the $p_0$-stratum of $\sReal{\Delta^{\{p_0<p_1\}}}$. It lies in the closure of ${(\sReal{\Delta^{\{p_0<p_1\}}})}_{p_1}$.
    As $A$ is closed and contains $S$, it follows that we have
    \[
    i(x_0) \in  i( \overline{ (\sReal{\Delta^{\{p_0<p_1\}}})_{p_1}}) = \overline{i((\sReal{\Delta^{\{p_0<p_1\}}})_{p_1})} \subset \overline{S} \subset A.
    \]
    To summarize, we have $i(x_0) \in A \cap X_Q$, for any closed subset $A \subset X$, which shows that the closure of $S_Q$ in $X_Q$ contains $i(x_0)$. As $ i(x_0) \notin S_Q$, $S_Q$ is not closed in $X_Q$.
    \\ However, $S_Q$ is $\Delta$-closed in $X_Q$. To see this, denote by $X_n$ the union of the first $n$ copies of $\sReal{\Delta^P}$ in $X$.
    Note that since $\real{\Delta^1}$ is compact, it follows that any map continuous $f \colon \real{\Delta^1} \to X_Q$ factors through some $(X_n)_Q$, for $n$ sufficiently large. Furthermore, we have
    \[
    S_Q \cap (X_n)_Q = (\bigsqcup_{m \leq n}S_m)_Q
    \]
    which shows that $S_Q \cap (X_n)_Q$ is a closed subset of $X_Q$. Consequently, $f^{-1}( S_Q) =f^{-1}(S_Q \cap (X_n)_Q)$ is closed in $\real{\Delta^1}$, which proves that $S_Q$ is $\Delta$-closed.
\end{example}
\section{A characterization of weak equivalences of topological spaces}
The following characterization of weak equivalences is certainly well known. For a lack of convenient reference, we nevertheless give a proof here.
\begin{lemma}\label{lem:criterion_for_weak_equivalence}
    A map $f \colon T \to T'$ is a weak homotopy equivalence in $\TopN$, if and only if for every solid commutative diagram  \begin{diagram}
        S^n \arrow[d] \arrow[r,"g^1"] & T \arrow[d, "f"] \\
        D^{n+1} \arrow[r, "g^0"] \arrow[ru, dashed, "l"] & T' \spacecomma
    \end{diagram}
    with $n \geq -1$, there exists a dashed arrow $l$ such that $(1_T,f) \circ (l|_{S^n},l)$ is homotopic to $(g^1, g^0)$ as a map of arrows (i.e. homotopic in the presheaf category $\Fun(([1]])^{\op}, \TopN)$ with respect to the cylinder given by the pointwise product with $[0,1]$).
\end{lemma}
\begin{proof}
    We use the classical characterization of weak equivalences found for example in \cite[Ch. 9.6]{ConciseMay}.
    Indeed, the classical characterization of weak equivalences even guarantees a lift, where the homotopy may be taken constant on $S^1$. Conversely, if we are given such a lift $l$, together with a homotopy $(H^1, H^0): (g^1, g^0) \implies (1_T,f) \circ (l|_{S^n},l)$, then any extension 
    \begin{diagram}
        {S^n \times [0,1]} \times D^{n+1} \times \{1 \} \arrow[r, "H^1 \cup l" ] \arrow[d, hook]& T \\
        D^{n+1} \times [0,1] \arrow[ru, dashed, "\hat L"'] &
    \end{diagram}
    will provide $\hat l = \Hat L_0$ such that the upper left triangle in
        \begin{diagram}
        S^n \arrow[d] \arrow[r,"g^1"] & T \arrow[d, "f"] \\
        D^{n+1} \arrow[r, "g^0"] \arrow[ru, dashed, "\hat l"] & T'
    \end{diagram}
    commutes on the nose. Furthermore, then $f \circ \Hat L$ and $H^0$ both provide extensions 
    \begin{diagram}
        {S^n \times [0,1]} \times D^{n+1} \times \{1 \} \arrow[r, "f \circ (H^1 \cup l)" ] \arrow[d, hook]& T' \\
        D^{n+1} \times [0,1] \arrow[ru, dashed] &  \spaceperiod
    \end{diagram}
    Since the left hand vertical of the last diagram is an acyclic cofibration, any two such extensions are homotopic relative to ${S^n \times [0,1]} \times D^{n+1} \times \{1 \}$. It follows that $f \circ \hat l = (f \circ \Hat L)_0$ and $g^0 = (H^0)_0$ are homotopic relative to $S^n$.
\end{proof}

\end{document}